\documentclass[11pt,twoside,reqno]{amsart}
\usepackage{amsmath,amsthm,amsfonts,amssymb,amsmath,mathrsfs,stmaryrd,relsize}
\usepackage[dvipsnames]{xcolor}
\usepackage[all]{xy}
\usepackage[alphabetic]{amsrefs} 

\usepackage{url}

\usepackage[utf8]{inputenc}

\usepackage{graphicx}
\usepackage{hyperref}
\hypersetup{  %% Set up hyperref color
	colorlinks=true,
	linkcolor=blue,
	citecolor=green,
	urlcolor=magenta,
}

\usepackage{enumitem}
\usepackage[capitalize]{cleveref}

\usepackage{fixmath}

\usepackage{verbatim}
\usepackage{mathdots}

\RequirePackage{xspace}
\RequirePackage{mathtools}
\newcommand{\fka}{\ensuremath{\mathfrak{a}}\xspace}

\newcommand{\fkd}{\ensuremath{\mathfrak{d}}\xspace}

\newcommand{\fkp}{\ensuremath{\mathfrak{p}}\xspace}

\newcommand{\BA}{\ensuremath{\mathbb{A}}\xspace}

\newcommand{\BC}{\ensuremath{\mathbb{C}}\xspace}

\newcommand{\BE}{\ensuremath{\mathbb{E}}\xspace}
\newcommand{\BF}{\ensuremath{\mathbb{F}}\xspace}
\newcommand{\BG}{\ensuremath{\mathbb{G}}\xspace}
\newcommand{\BH}{\ensuremath{\mathbb{H}}\xspace}

\newcommand{\BL}{\ensuremath{\mathbb{L}}\xspace}

\newcommand{\BP}{\ensuremath{\mathbb{P}}\xspace}
\newcommand{\BQ}{\ensuremath{\mathbb{Q}}\xspace}
\newcommand{\BR}{\ensuremath{\mathbb{R}}\xspace}

\newcommand{\BV}{\ensuremath{\mathbb{V}}\xspace}
\newcommand{\BW}{\ensuremath{\mathbb{W}}\xspace}
\newcommand{\BX}{\ensuremath{\mathbb{X}}\xspace}

\newcommand{\BZ}{\ensuremath{\mathbb{Z}}\xspace}

\newcommand{\bB}{\ensuremath{\mathbf{B}}\xspace}

\newcommand{\bH}{\ensuremath{\mathbf{H}}\xspace}

\newcommand{\bK}{\ensuremath{\mathbf{K}}\xspace}

\newcommand{\bT}{\ensuremath{\mathbf{T}}\xspace}

\newcommand{\ba}{\ensuremath{\mathbf{a}}\xspace}
\newcommand{\bb}{\ensuremath{\mathbf{b}}\xspace}

\newcommand{\bi}{\ensuremath{\mathbf{i}}\xspace}

\newcommand{\bt}{\ensuremath{\mathbf{t}}\xspace}
\newcommand{\bu}{\ensuremath{\mathbf{u}}\xspace}
\newcommand{\bv}{\ensuremath{\mathbf{v}}\xspace}
\newcommand{\bw}{\ensuremath{\mathbf{w}}\xspace}
\newcommand{\bx}{\ensuremath{\mathbf{x}}\xspace}
\newcommand{\by}{\ensuremath{\mathbf{y}}\xspace}
\newcommand{\bz}{\ensuremath{\mathbf{z}}\xspace}

\newcommand{\CA}{\ensuremath{\mathcal{A}}\xspace}

\newcommand{\CC}{\ensuremath{\mathcal{C}}\xspace}
\newcommand{\CD}{\ensuremath{\mathcal{D}}\xspace}
\newcommand{\CE}{\ensuremath{\mathcal{E}}\xspace}
\newcommand{\CF}{\ensuremath{\mathcal{F}}\xspace}
\newcommand{\CG}{\ensuremath{\mathcal{G}}\xspace}

\newcommand{\CL}{\ensuremath{\mathcal{L}}\xspace}
\newcommand{\CM}{\ensuremath{\mathcal{M}}\xspace}
\newcommand{\CN}{\ensuremath{\mathcal{N}}\xspace}
\newcommand{\CO}{\ensuremath{\mathcal{O}}\xspace}

\newcommand{\CS}{\ensuremath{\mathcal{S}}\xspace}

\newcommand{\CV}{\ensuremath{\mathcal{V}}\xspace}

\newcommand{\CX}{\ensuremath{\mathcal{X}}\xspace}
\newcommand{\CY}{\ensuremath{\mathcal{Y}}\xspace}
\newcommand{\CZ}{\ensuremath{\mathcal{Z}}\xspace}

	\renewcommand{\to}{%
		\ifbool{@display}{\longrightarrow}{\rightarrow}%
	}
	\let\shortmapsto\mapsto
	\renewcommand{\mapsto}{%
		\ifbool{@display}{\longmapsto}{\shortmapsto}%
	}

\newcommand{\lr}{\longrightarrow}

\newcommand{\simlr}{\overset{\sim}{\lr}}

\newcommand{\Aut}{\mathrm{Aut}}
\newcommand{\Ch}{\mathrm{Ch}}

\newcommand{\Gal}{\mathrm{Gal}}
\newcommand{\End}{\mathrm{End}}
\newcommand{\Hom}{\mathrm{Hom}}
\newcommand{\GL}{\mathrm{GL}}

\DeclareMathOperator{\Int}{\ensuremath{\mathrm{Int}}\xspace}
\DeclareMathOperator{\Lie}{Lie}

\DeclareMathOperator{\Nm}{Nm}

\DeclareMathOperator{\Orb}{Orb}
\DeclareMathOperator{\ord}{ord}
\DeclareMathOperator{\Pic}{Pic}
\renewcommand{\Re}{{\mathrm{Re}}}

\newcommand{\reg}{{\mathrm{reg}}}
\DeclareMathOperator{\Res}{Res}
\newcommand{\Sh}{\mathrm{Sh}}

\newcommand{\SL}{{\mathrm{SL}}}

\DeclareMathOperator{\Spec}{Spec}
\DeclareMathOperator{\Spf}{Spf}
\newcommand{\SO}{{\mathrm{SO}}}

\newcommand{\val}{{\mathrm{val}}}

\DeclareMathOperator{\tr}{tr}
\newcommand{\triv}{\mathrm{aux}}
\newcommand{\U}{\mathrm{U}}
\DeclareMathOperator{\vol}{vol}
\newcommand{\wt}{\widetilde}
\newcommand{\wh}{\widehat}
\newcommand{\ov}{\overline}

\newcommand{\Herm}{\mathrm{Herm}}
\newcommand{\rd}{\mathrm{d}}

\newcommand{\Diff}{\mathrm{Diff}}
\newcommand{\Ei}{\mathrm{Ei}}

	\DeclareFontFamily{U}{matha}{\hyphenchar\font45}
	\DeclareFontShape{U}{matha}{m}{n}{
		<5> <6> <7> <8> <9> <10> gen * matha
		<10.95> matha10 <12> <14.4> <17.28> <20.74> <24.88> matha12
	}{}
	\DeclareSymbolFont{matha}{U}{matha}{m}{n}
	\DeclareFontFamily{U}{mathx}{\hyphenchar\font45}
	\DeclareFontShape{U}{mathx}{m}{n}{
		<5> <6> <7> <8> <9> <10>
		<10.95> <12> <14.4> <17.28> <20.74> <24.88>
		mathx10
	}{}
	\DeclareSymbolFont{mathx}{U}{mathx}{m}{n}
	
	\DeclareMathSymbol{\obot}         {2}{matha}{"6B}
	
\newtheorem{theorem}[subsubsection]{Theorem}
\newtheorem{proposition}[subsubsection]{Proposition}
\newtheorem{lemma}[subsubsection]{Lemma}

\newtheorem{corollary}[subsubsection]{Corollary}

\theoremstyle{definition}
\newtheorem{definition}[subsubsection]{Definition}
\newtheorem{example}[subsubsection]{Example}

\newtheorem{remark}[subsubsection]{Remark}

\numberwithin{equation}{subsection}

\newtheorem{notation}[subsubsection]{Notation}

\DeclareMathOperator{\Span}{Span}
 
\newcommand{\Eis}{\mathrm{Eis}}

\newcommand{\DCZ}{\prescript{\BL}{}{\CZ}}
\newcommand{\tCM}{\wt{\CM}}
\newcommand{\tCZ}{\wt{\CZ}}
\newcommand{\tDCZ}{\prescript{\BL}{}{\wt\CZ}}
\newcommand{\tCG}{\wt{\CG}}

\newcommand{\GZ}{\mathcal{Z}}
\newcommand{\DGZ}{\prescript{\BL}{}{\GZ}}
\newcommand{\Dotimes}{\otimes^{\BL}}

\newcommand{\btau}{\mathbold{\tau}}

\newcommand{\angler}{\langle r\rangle}
\newcommand{\Stab}{\mathrm{Stab}}
\newcommand{\FJ}{\mathrm{FJ}}
\newcommand{\PFJ}{\partial\mathrm{FJ}}
\newcommand{\KM}{\mathrm{KM}}
\newcommand{\PEis}{\partial\mathrm{Eis}}
\newcommand{\PC}{\partial C}
\newcommand{\PDen}{\partial\mathrm{Den}}

\newcommand{\corr}{\mathrm{corr}}
\newcommand{\modi}{\mathrm{mod}}

\newcommand{\hol}{\mathrm{hol}}
\newcommand{\charfun}{\mathbf{1}}

\title{Kudla--Rapoport conjecture for unramified maximal parahoric level}
	
\author{Yu LUO}
\address{Massachusetts Institute of Technology, Department of Mathematics, 77 Massachusetts Avenue, Cambridge, MA 02139, USA}
\email{yuluo@mit.edu}

\date{\today}
	
\begin{document}

\begin{abstract}
We prove the Kudla--Rapoport conjecture for unramified unitary groups with maximal parahoric level structure.
Our approach differs from the local proof of C. Li and W. Zhang. We reduce the conjecture to a global intersection problem using local-global compatibility. Then we apply an inductive procedure based on the modularity of generating series of global special divisors. This strategy follows the framework developed in the proof of the arithmetic fundamental lemma by W. Zhang and Mihatsch--Zhang and arithmetic transfer identities by Z. Zhang and Luo--Mihatsch--Zhang.
\end{abstract}
\maketitle
\tableofcontents
\section{Introduction}
\subsection{Background}
The classical \emph{Siegel--Weil formula} \cite{Siegel-35,Siegel1951,Weil1964,Weil1965} establishes a fundamental connection between central values of Siegel Eisenstein series and  theta functions -- generating series that encode the arithmetic of quadratic forms.
In a groundbreaking program initiated by Kudla (\cite{Kudla97, Kudla04}), the \emph{arithmetic Siegel--Weil formula} extends this relationship into arithmetic geometry.
It aims to express the central derivative of Siegel Eisenstein series as arithmetic theta functions -- generating series of arithmetic intersection numbers of $n$ special divisors on Shimura varieties associated to $\SO(n-1,2)$ or $\U(n-1,1)$. These special divisors include Heegner points on modular or Shimura curves appearing in the Gross--Zagier formula \cite{Gross-Zagier86, Yuan-Zhang-Zhang} when $n=2$, modular correspondences on the product of two modular curves in the Gross--Keating formula \cite{Gross-Keating} and Hirzebruch--Zagier cycles on Hilbert modular surfaces \cite{Hirzebruch-Zagier} when $n=3$.

The arithmetic Siegel--Weil formula has been established in great generality for $n=1,2$ (orthogonal case) through the seminal work of Kudla, Rapoport, and Yang (\cite{Kudla97, KR99, KRY99,KR00, KRY04,Kudla-Rapoport-Yang}). The \emph{archimedean} component of the formula is also known, due to Liu \cite{Liu11-I} (unitary case), and Garcia--Sankaran \cite{Garcia-Sankaran} in full generality (cf. Bruinier--Yang \cite{Bruinier-Yang21} for an alternative proof in the orthogonal case).

In their works \cite{KR11,KR14}, Kudla and Rapoport refined the non-archimedean component of the conjectural formula by defining arithmetic models of special cycles in the unitary case for arbitrary $n$. These models, now known as \emph{Kudla--Rapoport special cycles}, led to their formulation of the \emph{global Kudla--Rapoport conjecture}, addressing the nonsingular part of the arithmetic Siegel--Weil formula. They proved that this global conjecture would follow from the \emph{local Kudla--Rapoport conjecture}, a deep relationship between derivatives of local representation densities of hermitian forms and arithmetic intersection numbers of Kudla--Rapoport cycles on unitary Rapoport--Zink spaces.

This conjecture was proved in the celebrated paper of C. Li and W. Zhang \cite{LZ21} for unramified unitary groups with hyperspecial level and nonsingular coefficients. Their breakthrough has inspired numerous extensions to new cases: C. Li and Y. Liu \cite{Li-Liu22} and Yao \cite{Yao24} addressed ramified unitary groups with exotic good reduction, Q. He, C. Li, Y. Shi, and T. Yang \cite{He-Li-Shi-Yang} explored the Krämer model, and C. Li and W. Zhang \cite{Li-Zhang22} extended the results to orthogonal groups. For a comprehensive introduction to these developments, we refer the reader to \cite{Li-survey}. 
In parallel developments, Bruinier and Howard \cite{Bruinier-Howard} and Chen \cite{RChenI,RChenII,RChenIII,RChenIV} have made substantial progress in understanding singular coefficients.

The present paper focuses on the Kudla--Rapoport conjectures with maximal parahoric level structures for unramified unitary groups, building on the work of S. Cho \cite{Cho22} and S. Cho, Q. He, and Z. Zhang \cite{Cho-He-Zhang}. In this setting, special divisors appear in two distinct classes: $\CZ$-divisors and $\CY$-divisors. Cho \cite{Cho22} formulated an extended version of the local Kudla--Rapoport conjecture for intersection products of arbitrary combinations of these $\CZ$- and $\CY$-divisors.
Cho, He, and Z. Zhang \cite{Cho-He-Zhang} proved the conjectures for $\CZ$-cycles, conditionally on a strong version of the Tate conjecture, while obtaining unconditional results for the cases $\CN_n^{[n-1]}$ and $\CN_4^{[2]}$.
Their method follows \cite{LZ21}, where the Tate conjecture is used to prove local modularity, which is essential for their induction argument.

In this paper, we reformulate and prove the Kudla--Rapoport conjectures with maximal parahoric level structures for unramified unitary groups and arbitrary mixed cycles. Our approach differs from the local inductive method of \cite{LZ21}, which relies on local modularity, see \cite[\S 1.4]{LZ21}. Instead, we employ an induction argument based on local-global compatibility and global modularity.
Our method builds on W. Zhang's proof of the arithmetic fundamental lemma \cite{Zhang21}, which has been subsequently generalized by Mihatsch and W. Zhang \cite{Mihatsch-Zhang} to arbitrary
$p$-adic fields for $p$ large enough, by Z. Zhang \cite{ZZhang24} to maximal parahoric level over finite unramified extensions of $\BQ_p$, and by the author, Mihatsch, and Z. Zhang \cite{LMZ25} to maximal parahoric level over arbitrary $p$-adic fields.

\subsection{The Kudla--Rapoport conjecture}
Let $p$ be an odd prime. Let $F_0$ be a finite extension of $\BQ_p$ with residue field $\BF_q$ and a uniformizer $\varpi$. Let $F$ be the unramified quadratic extension of $F_0$ and let $\breve{F}$ be the completion of the maximal unramified extension of $F$.
For any integers $n\geq 1$ and $0\leq h\leq n$, the \emph{unitary Rapoport--Zink space} $\CN_n^{[h]}$ (\S \ref{sec:RZ space}) is the formal scheme over $S=\Spf O_{\breve{F}}$ parameterizing hermitian $O_F$-modules of signature $(n-1,1)$ and type $h$, within the supersingular quasi-isogeny class (\S \ref{sec:strict modules}).

Let $\BE$ (resp. $\BX^{[h]}$) be the framing hermitian $O_F$-module of signature $(1,0)$ and type $0$ (resp. $(n-1,1)$ and type $h$) over $\ov{\BF}_q$. The space of \emph{quasi-homomorphisms} $\BV=\BV^{[h]}:=\Hom_{O_F}^{\circ}(\BE,\BX^{[h]})$ carries a natural $F/F_0$-hermitian form, which makes $\BV$ the unique (up to isomorphism) non-degenerate hermitian space that does not contain any type $h$ vertex lattice, i.e., a hermitian lattice $L\subset \BV$ such that $\varpi L^\sharp\subseteq L\overset{h}\subseteq L^\sharp$, where $L^\sharp$ is the hermitian dual of $L$.
Let $V$ be its nearby hermitian space of dimension $n$. It contains a type $h$ vertex lattice, which we denote by $L^{[h]}$.

When $h\neq 0$, for any nonzero $u\in \BV^{[h]}$, there are two classes of special divisors on $\CN_n^{[h]}$: $\CZ$-divisors $\CZ(u)$ and $\CY$-divisors $\CY(u)$ (\S \ref{def:special divisor}), arising due to the non-principal polarization. 
Let
$$
\CZ(u,\charfun_{L^{[h]}}):=\CZ(u),\quad\text{and}\quad\CZ(u,\charfun_{L^{[h],\sharp}}):=\CY(u).
$$

Let $1\leq r\leq n$ be an integer. A \emph{weight vector} is an element $\bw\in \{\CZ,\CY\}^r$. 
For any $0\leq h\leq n$, an\footnote{A priori, the function also depends on the choice of the vertex lattice $L^{[h]}$, but all statements are in fact independent of this choice.} \emph{admissible Schwartz function (of weight $\bw$ and type $h$)} is a function $\phi:=\phi^{[h]}_{\bw}\in\CS(V^r)$ such that 
$$
\phi=\bigotimes_{i=1}^r \phi_i,\quad\text{where}\quad \phi_i=\left\{\begin{array}{ll}
\charfun_{L^{[h]}}     &\text{if }\bw_i=\CZ;\\
\charfun_{L^{[h],\sharp}}  &\text{if }\bw_i=\CY.
\end{array}\right.
$$

For any linearly independent $\bu\in \BV^n$ and any admissible Schwartz function $\phi^{[h]}=\phi^{[h]}_{\bw}$, we define the \emph{arithmetic intersection number} for mixed cycles (Definition \ref{def:local special cycle})
\begin{equation}\label{equ:arithmetic intersection number}
\Int(\bu,\phi_{\bw}^{[h]}):=\chi(\CN_n^{[h]},\CO_{\CZ(u_1,\phi_1)}\Dotimes\cdots\Dotimes \CO_{\CZ(u_n,\phi_n)})\in\BZ,
\end{equation}
where $\CO_{\CZ(u_i,\phi_i)}$ is the structure sheaf of the special divisor $\CZ(u_i,\phi_i)$ and $\Dotimes$ is the derived tensor product of coherent sheaves on $\CN_n^{[h]}$, and $\chi$ is the Euler--Poincar\'e characteristic.

When $h=0$, the hermitian lattice $L=L^{[0]}$ is \emph{unimodular}, meaning that $L=L^\sharp$. Let $\phi^{[0]}=\charfun_{L^n}\in \CS(V^n)$. 
The work of C. Li and W. Zhang \cite{LZ21} shows the identity:
\begin{equation}\label{equ:hyperspecial KR}
\Int(\bu,\phi^{[0]})=\frac{\alpha'(1_n,T)}{\alpha(1_n,1_n)}
\end{equation}
where $1_n$ is the identity matrix of rank $n$, and $T=(\bu,\bu)$ is the moment matrix associated with a collection of linearly independent vectors $\bu\in \BV^n$. 
The function $\alpha(1_n,T)$ is the \emph{representation density function} and $\alpha'(1_n,T)$ is its central derivative, see \cite[\S 10]{KR14} or \S \ref{sec:base case} for precise definitions.

The representation density function satisfies the \emph{linear invariance property} \cite[\S 3.3]{LZ21}, which mirrors the intersection multiplicity behavior on the geometric side \cite[\S 2.8]{LZ21}, see also \cite{Howard-linear}. 
However, in the general maximal parahoric case, the intersections between $\CZ$-divisors and $\CY$-divisors fail to satisfy the linear invariance property. Therefore, they are not expected to directly relate to standard representation density functions.
We introduce the \emph{weighted local density function}, as defined by Kudla \cite[Appendix]{Kudla97}. This formulation generalizes the representation density functions appearing in the hyperspecial case \eqref{equ:hyperspecial KR}.

\begin{definition}
Let $r$ be any non-negative integer and let 
$V_{r,r}$ be the split hermitian space of dimension $2r$.
Given a Schwartz function $\phi\in\CS(V^n)$ in a hermitian space $V$ of dimension $n$ and a hermitian matrix $T\in\Herm_n(F)$, we define the \emph{weighted local density function} as
\begin{equation*}
    \alpha_T(r,\phi):=\int_{\Herm_n(F)}\int_{(V^{\angler})^n}\psi\Bigl(\tr\bigl(b((x,x)-T)\bigr)\Bigr)\phi^{\angler}(x)dxdb,
\end{equation*}
where $V^{\angler}:=V\oplus V_{r,r}$ and  $\phi^{\angler}:=\phi\otimes \charfun_{L^n_{r,r}}\in \CS((V^{\angler})^n)$, with $L_{r,r}\subset V_{r,r}$ being a self-dual lattice. Here $dx$ and $db$ are self-dual Haar measures on $(V^{\angler})^n$ and $\Herm_n(F)$, respectively.
\end{definition}

\begin{remark}
This provides a natural generalization of the representation density function: when $\phi=\charfun_{\Lambda^n}$ is the characteristic function of a lattice $\Lambda\subset V$, for any $r\geq 0$, let $\bu_\Lambda^{\angler}$ be a basis of the lattice $\Lambda\oplus L_{r,r}$ and let $S_\Lambda^{\angler}=(\bu_{\Lambda}^{\angler},\bu_{\Lambda}^{\angler})$ be its moment matrix. By \cite[Proposition A.6]{Kudla97}, for any hermitian matrix $T\in\Herm_n(F)$, we have
\begin{equation}\label{equ:relation of weighted local density}
\alpha_T(r,\charfun_{\Lambda^n})=|\mathrm{Nm}_{F/F_0}\det(S_\Lambda^0)|^{n/2}_{F_0}\cdot \alpha(S_\Lambda^{\angler},T),
\end{equation}
where $|-|_{F_0}$ is the normalized absolute value.
See also \cite[\S 10]{KR14} for further details of this relation.
\end{remark}

Let $\phi^{[h]}\in \CS(V^n)$ be an admissible Schwartz function of type $h$.
It is known that there exists a polynomial $F_T(X,\phi^{[h]})\in \BQ[X]$ such that for each non-negative integer $r$, we have $\alpha_T(r,\phi^{[h]})=F_T(q^{-2r},\phi^{[h]})$.
Let $T\in\Herm_n(F)$ be any hermitian matrix that does not represent the space $V$. By definition, $\alpha_T(0,\phi^{[h]})=0$.
In this case, we define the \emph{derivative of the weighted local density function}
$$
\alpha'_T(0,\phi^{[h]}):=-\frac{\rd}{\rd X}\Big|_{X=1}F_T(X,\phi^{[h]}),\quad X=q^{-2r}.
$$
As already observed in \cite{KR00,Kudla-Rapoport-Yang,He-Li-Shi-Yang,Cho22,Cho-He-Zhang}, the naive identity analogous to \eqref{equ:hyperspecial KR} does not hold in the general maximal parahoric setting, and correction terms must be added to the analytic side. We will explain it after stating the main result.

\begin{theorem}[Theorem \ref{thm:KR-conj}, local Kudla--Rapoport conjecture]\label{thm:local KR conj}
For any $0\leq h\leq n$ and for any weight vector $\bw\in\{\CZ,\CY\}^n$, let $\beta^{[h]}_{\bw,t}$ be the explicit constants  for $0 \leq t < h$, $t \equiv h-1 \bmod{2}$, given by Definition \ref{def:error-coeff}. Then for any collection of linearly independent vectors $\bu\in \BV^n$, we have
\begin{equation}\label{equ:local KR conj}
\Int(\bu,\phi_{\bw}^{[h]})=\frac{1}{\vol(K_n^{[h]},dg)}\Bigl(
\alpha_T'(0,\phi_{\bw}^{[h]})-\sum_{\substack{0\leq t< h\text{ and}\\ t\equiv h-1\bmod{2}}} \beta_{\bw,t}^{[h]}\alpha_T(0,\phi_{\bw}^{[t]})
\Bigr),
\end{equation}
where $T=(\bu,\bu)$ and $\vol(K_n^{[h]},dg)$ is the volume of the stabilizer $K_n^{[h]}=\mathrm{Stab}_{\U(V)}(L^{[h]})$ with respect to the Haar measure $dg$ constructed from self-dual measures, see \eqref{not:gauge form}.
\end{theorem}

\begin{remark}\label{rem:error terms}
The coefficients of the error terms are determined in Definition \ref{def:error-coeff}. We illustrate the construction in the case where $\bw=\CZ^n$ and $\phi^{[h]}_{\bw}=\charfun_{L^{[h],n}}$ following \cite[Definition 1.2]{Cho-He-Zhang}.

For each $0 \leq t < h$ with $t \equiv h-1 \bmod{2}$, let $\bu_t\in \BV^n$ be a basis of a vertex lattice $L^{[t]}$ of type $t$ and let $S_n^{[t]}=(\bu_t,\bu_t)$ be its moment matrix.
We consider the linear combination:
$$
\alpha_T^{\modi}(0,\phi^{[h]}_\bw):=\alpha_T'(0,\phi_{\bw}^{[h]})-\sum_{\substack{0\leq t< h\text{ and}\\  t\equiv h-1\bmod{2}}} \beta_{\bw,t}^{[h]}\alpha_T(0,\phi_{\bw}^{[t]}),
$$
where $\phi^{[t]}_\bw\in \CS(\BV^n)$ are admissible Schwartz functions of weight $\bw$ and type $t$ (while $\phi_{\bw}^{[h]}\in\CS(V^n)$).
The constants $\beta^{[h]}_{\bw,t}\in\BQ$ are chosen to satisfy
$$\alpha^{\modi}_{S_n^{[t]}}(0,\phi^{[h]}_{\bw})=0\quad \text{for all }0\leq t<h\text{ with }t\equiv h-1\bmod{2}.$$
Since $\alpha_{S_n^{[i]}}(0,\phi^{[j]}_{\bw})=0$ for any $i<j$ and $\alpha_{S_n^{[i]}}(0,\phi^{[i]}_{\bw})\neq 0$ for any $i$, this forms a linear system with an invertible upper triangular coefficient matrix, which uniquely determines the constants  $\beta^{[h]}_{\bw,t}$.
\end{remark}

\begin{example}
When $h=0$ and $\phi^{[h]}=\charfun_{L^n}$ for a self-dual lattice $L$, there is no error term in \eqref{equ:local KR conj}. By Lemma \ref{lem:local density volumn}, we have $\vol(K_n^{[0]},dg)=\alpha(S_n^{[0]},S^{[0]}_n)$, where $S_n^{[0]}$ is the identity matrix of rank $n$.
By \eqref{equ:relation of weighted local density}, equation \eqref{equ:local KR conj} in this case is identical to 
$$
\chi(\CN_n^{[0]},\CO_{\CZ(u_1)}\Dotimes\cdots\Dotimes\CO_{\CZ(u_n)})=\frac{\alpha_T'(0,\phi^{[0]})}{\alpha_{S_n^{[0]}}(0,\phi^{[0]})}=\frac{\alpha'(S_n^{[0]},T)}{\alpha(S_n^{[0]},S_n^{[0]})},
$$ 
which reproves the Kudla--Rapoport conjecture at hyperspecial level \cite[Theorem 3.4.1]{LZ21}.
\end{example}

\begin{example}
When $h=1$ and $\phi^{[1]}=\charfun_{L^{[1],n}}$ for an almost self-dual lattice $L^{[1]}$, we have $\phi^{[0]}=\charfun_{L^{[0],n}}$ for a self-dual lattice $L^{[0]}$.
By Lemma \ref{lem:local density volumn} and \eqref{equ:compare two local density}, we have 
$$\vol(K^{[1]},dg)=\alpha_{S_n^{[1]}}(0,\phi_{n,0}^{[1]})=q^{-n}\alpha(S_n^{[1]},S_n^{[1]})\quad\text{where}\quad S_n^{[1]}=\mathrm{diag}(1,\cdots,1,\varpi).
$$
By relation \eqref{equ:relation of weighted local density}, equation \eqref{equ:local KR conj} in this case is identical to (after modifying $\beta_{\bw,0}^{[1]}$ by constant factors in \eqref{equ:relation of weighted local density}) 
\begin{equation}\label{equ:almost self dual}
\chi(\CN_n^{[1]},\CO_{\CZ(u_1)}\Dotimes\cdots\Dotimes\CO_{\CZ(u_n)})=\frac{1}{\alpha(S_n^{[1]},S_n^{[1]})}\left(
\alpha'(S_n^{[1]},T)-q^n\beta_{\bw,0}^{[1]}\alpha(S_n^{[0]},T)
\right).
\end{equation}
As explained in Remark \ref{rem:error terms}, we have
\begin{equation*}
    q^n\beta^{[1]}_{\bw,0}=\frac{\alpha'_{S_n^{[0]}}(0,\phi^{[1]})}{\alpha_{S_n^{[0]}}(0,\phi^{[0]})}=\frac{\alpha'(S_n^{[1]},S_n^{[0]})}{\alpha(S_n^{[0]},S_n^{[0]})}=\frac{1}{1-(-q)^{-n}},
\end{equation*}
where the last equality is due to \cite[Proposition A.5]{Cho22}.
By \cite[after (13.6.1.1)]{LZ21}, we have
\begin{equation*}
\frac{1}{1-(-q)^{-n}}\frac{\alpha(S_n^{[0]},S_n^{[0]})}{\alpha(S_n^{[1]},S_n^{[1]})}=\frac{\alpha(S_{n-1}^{[0]},S_{n-1}^{[0]})}{\alpha(S_n^{[1]},S_n^{[1]})}=\frac{1}{q+1}.
\end{equation*}
Hence \eqref{equ:almost self dual} is equivalent to
$$
\chi(\CN_n^{[1]},\CO_{\CZ(u_1)}\Dotimes\cdots\Dotimes\CO_{\CZ(u_n)})=\frac{1}{q+1}\left(
\frac{\alpha'(S_n^{[1]},T)}{\alpha(S_{n-1}^{[0]},S_{n-1}^{[0]})}-\frac{\alpha(S_{n}^{[1]},T)}{\alpha(S_{n}^{[0]},S_{n}^{[0]})}
\right).
$$
This recovers the Kudla--Rapoport conjecture at almost self-dual level for pure $\CZ$-cycles proved in \cite[Theorem 10.5.1]{LZ21}. All other cases of Theorem \ref{thm:local KR conj} are new.
\end{example}

\subsection{The arithmetic Siegel--Weil formula}\label{sec:ASWF}
Next, we describe the global applications of our local theorem to the arithmetic Siegel--Weil formula for semi-global integral models.

Let us now switch to global notation.
Let $F/F_0$ be a CM extension with $F_0\neq \BQ$.
We fix a CM type $\Phi$ for $F$ with a distinguished element $\varphi_0\in \Phi$. Let $V$ be an $F/F_0$-hermitian space of dimension $n$ with signature $(n-1,1)$ at $\varphi_0$ and signature $(n,0)$ for $\varphi\in\Phi\setminus\{\varphi_0\}$.
Consider the following group data:
\begin{align*}
    Z^\BQ:={}&\{z\in\Res_{F/\BQ}\BG_m\mid \Nm_{F/F_0}(z)\in\BG_m\},\\
    G^\BQ:={}&\{g\in\Res_{F_0/\BQ}\mathrm{GU}(V)\mid c(g)\in\BG_m\},\\
    \wt{G}:={}&Z^\BQ\times_{\BG_m}G^\BQ\xrightarrow{\sim}Z^\BQ\times \Res_{F_0/\BQ}G.
\end{align*}
Associated to $\wt{G}$ is a natural Shimura datum $(\wt{G},\{h_{\wt{G}}\})$ of PEL-type. 

Let $K_{\wt{G}}=K_{Z^\BQ}\times K_G$ be a compact open subgroup. Then the associated Shimura variety $\wt{M}=\Sh_{K_{\wt{G}}}(\wt{G},\{h_{\wt{G}}\})$ is of dimension $n-1$ and has a canonical model over the reflex field $E$.

Let $L$ be an $O_F$-lattice and let $\fkd\in\BZ_{>0}$ be chosen to satisfy the following conditions:
\begin{itemize}
    \item We require that $2\mid \fkd$\footnote{While this assumption may be unnecessary, we do not pursue it further in this paper.}.
    \item If $v_0$ is a finite place of $F_0$ that is ramified in $F$, then $v_0\mid \fkd$.
    \item For any finite place $v_0\nmid \fkd$ of $F_0$, the localization $L_{v_0}$ in $V_{v_0}=V \otimes_{F_0} F_{0, v_0}$ is a vertex lattice.
\end{itemize}
We make the following choice of level structure:
\begin{itemize}
\item The level $K_G\subset G(\BA_{F_0,f})$ factors as 
    \begin{equation*}
        K_G=K_{G,\fkd}\times K_G^\fkd
    \end{equation*}
    where $K_{G,\fkd}\subset G(F_{0,\fkd})$ is any open compact subgroup and where $K_G^\fkd=\mathrm{Stab}(\wh{L}^\fkd)$ is the stabilizer of the completion $\wh{L}$ in $G(\BA_{0,f}^\fkd)$;
\item The level $K_{Z^\BQ}\subset Z^{\BQ}(\BA_{\BQ,f})$ factors similarly as
    \begin{equation*}
        K_{Z^\BQ}=K_{Z^\BQ,\fkd}\times K_{Z^\BQ}^{\fkd}
    \end{equation*}
    where $K_{Z^\BQ,\fkd}\subset Z^\BQ(\BQ_\fkd)$ is any open compact subgroup and where $K_{Z^\BQ}^{\fkd}$  is the maximal compact subgroup in $Z^\BQ(\BA^\fkd_{\BQ,f})$.
\end{itemize}
We construct a \emph{punctured integral model} 
$\wt\CM\to \Spec O_{E}[\fkd^{-1}]$
of $\wt M$ (\S \ref{sec:RSZ-integral}). In general, $\wt\CM$ is neither smooth nor regular, but it is a locally complete intersection, see \cite[Remark 6.10]{LMZ25}.

Let $\phi\in\CS(V(\BA_{F_0,f})^n)^{K_G}$ be an \emph{admissible Schwartz function}. Here, admissible means that 
$$\phi=\phi_{\fkd} \otimes\phi^{\fkd}\quad\text{and}\quad\phi^{\fkd}=\bigotimes_{i=1}^n \phi^{\fkd}_{i}$$ where each $\phi^{\fkd}_{i}$ is either $\charfun_{\wh L^{\fkd}}$ or $\charfun_{\wh L^{\sharp,\fkd}}$.

For any  $T\in\Herm_n(F)$, we define the \emph{weighted Kudla--Rapoport special cycle} (\S \ref{sec:global KR cycle}) $\tCZ(T,\phi)\to \tCM$. Furthermore, for any hermitian matrix $T\in\Herm_n(F)$ with totally positive diagonal entries, we define the \emph{derived special cycles} $\tDCZ(T,\phi)\in \Ch^n(\tCM)$ (\S \ref{sec:derived special cycle}).

On the other hand, let $\BV$ be the incoherent $\BA/\BA_0$-hermitian space nearby $V$, namely, $\BV_{\varphi}$ is positive definite for $\varphi\in\Phi$ and $\BV_{v_0}\cong V_{v_0}$ for all finite places $v_0$ of $F_0$.
Let $\phi\otimes\phi_\infty \in\CS(\BV^n)$, where $\phi_\infty$ is the standard Gaussian function \eqref{equ:standard-gaussian}, there is an \emph{incoherent Eisenstein series} $E(\btau,s,\phi\otimes\phi_{\infty})$ (see \S \ref{sec:Eisenstein series}) on the hermitian upper half space
\begin{equation*}
    \BH_n(F_{\infty})=\{\btau=\bx+i\by\mid  \bx\in\Herm_n(F_{\infty}), \by\in \Herm_{n}^+(F_{\infty})\},\quad\text{where}\quad F_\infty:=F\otimes_{\BQ}\BR\cong \BC^{\Phi}.
\end{equation*}
This is the Siegel Eisenstein series associated to a standard Siegel--Weil section of the degenerate principal series. 
Since we will always fix the standard Gaussian function at archimedean places, we write the Eisenstein series as $E(\btau,s,\phi)$ to simplify the notation.

The Eisenstein series has a meromorphic continuation and a functional equation relating $s$ and $-s$.
The central value $\Eis(\btau,0,\phi)$ vanishes by incoherence. We thus consider its \emph{central derivative}
\begin{equation*}
    \PEis(\btau,\phi):=\frac{\rd}{\rd s}\Big|_{s=0}\Eis(\btau,s,\phi).
\end{equation*}
Let $\Diff(T,\BV)$ be the set of places $v_0$ of $F_0$ where $T$ does not represent $\BV_{v_0}$. Let $\tau(Z^{\BQ})=\#(Z^{\BQ}(\BQ)\backslash Z^{\BQ}(\BA_{\BQ,f})\slash K_{Z^{\BQ}})$.
\begin{theorem}[Arithmetic Siegel--Weil formula]
There exists an explicit coherent Eisenstein series $\Eis(\btau,\phi_{\corr})$ associated with the correction terms in Theorem \ref{thm:local KR conj} (see \S \ref{sec:local-factors}), such that for any positive definite $T\in\Herm_n^+(F)$ satisfying $\Diff(T,\BV)=\{v_0\}$ with $v_0\nmid \fkd$, we have
\begin{equation*}
   \frac{\vol(K_G)}{\tau(Z^{\BQ})[E:F]}\cdot {\deg}\tDCZ(T,\phi)q^T\cdot \log q_{v_0}^2=c_V\cdot \bigl(\partial\Eis_T(\btau,\phi)+\Eis_T(\btau,\phi_{\corr})\bigr),
\end{equation*}
where $c_V$ is an explicit constant determined by local factors at archimedean places (see \eqref{equ:ASWF-constant}).
\end{theorem}

\subsection{Proof strategy}
We present the essential idea for proving Theorem \ref{thm:local KR conj} in the case of hyperspecial level. Note that the generating series $\CE_{T_\flat}(\tau,\phi_1,\phi_\flat)$ presented below is a formal construction, the genuine generating series requires a careful definition in an appropriate function space, see \S \ref{sec:modular form}.

The proof of the local Kudla--Rapoport conjecture proceeds by induction. For $n=1$, this follows from a direct calculation, see \S \ref{sec:base case}.
Let us assume the conjecture holds for all unramified $p$-adic field ($p\neq 2$) extensions $F/F_0$ and all $F/F_0$-hermitian spaces of dimension $n-1$. Our goal is to establish the conjecture for a specific unramified quadratic extension $F/F_0$ of $p$-adic fields ($p\neq 2$) and a hermitian space $V$ of dimension $n$, and a collection of linearly independent vectors $\bu\in \BV^n$.

We embed our local data into global data in \S\ref{sec:ASWF}. When $F_0=\BQ_p$, however, the global embedding should be modified to avoid compactification; see \ref{CM3}.
By the $p$-adic uniformization of derived special cycles (\S \ref{sec:p adic unitofmization of special cycle}) and
the local constancy of both the geometric and analytic sides, we reduce the problem to proving the arithmetic Siegel--Weil formula for
\begin{equation}\tag{$\star$}\label{equ:ASWF}
\tDCZ(\bT,\phi),\quad\text{where}\quad \bT=\begin{pmatrix}
 t_1 & t_{12}\\
t_{21}& T_\flat
\end{pmatrix}\quad\text{with}\quad\Diff(\bT,\BV)=\{ v_0\},
\end{equation}
where the matrices and Schwartz functions are chosen as follows:
\begin{itemize}
    \item $T_\flat$ is an $(n-1)\times (n-1)$ matrix.
    \item $\phi_f^{\fkd}=\charfun_{L^{n,\fkd}}$ with a self-dual lattice $L$.
    \item $\phi_f=\phi_{1,f}\otimes\phi_{\flat,f}$ takes rational values.
\end{itemize}
We consider the following generating series
\begin{equation}\label{equ:KR proof generating series}
\CE_{T_\flat}(\tau,\phi_1,\phi_\flat)=\sum_{w_0}\sum_{\xi\in F_0}\sum_{{\substack{T=\begin{psmallmatrix}\xi& *\\ * & T_\flat\end{psmallmatrix}\\ \Diff(T,\BV)=\{w_0\}}}}\Bigl(\frac{\vol(K_G)}{\tau(Z^{\BQ})[E:F]}\wh{\deg}(\tCZ^{\bK}(T,\phi))-c_V\PEis_T(\tau,\phi)\Bigr)q^\xi,
\end{equation}
where $q^\xi:=\psi_\infty(\tr(\xi\tau))$.
This is a Hilbert modular form of weight $n$ involving two parts:
\begin{itemize}
\item The geometric part is the explicit expansion of the following intersection product
\begin{equation*}
    \langle \tCZ^{\bK}(\tau,\phi_{1}), \tCZ^{\bK}(T_\flat,\phi_{\flat})\rangle, \quad\text{where}\quad \tCZ^{\bK}(\tau,\phi_1)=\sum_{\xi\in F_0}\tCZ^{\bK}(\xi,\phi_1)q^\xi\in\wh{\Ch}^1(\CM)[\![q]\!].
\end{equation*}
The modularity of the intersection product was proved by \cite{ZZhang24} in \cite{LMZ25} when $F_0\neq \BQ$.
\item The analytic part is a Fourier--Jacobi type series of the Siegel Eisenstein series, which is also modular by definition.
\end{itemize}

When $\phi\in\CS(V(\BA_{F_0,f})^n)$ is \emph{regular} at two places (see Theorem \ref{thm:sing-vanish}), only non-singular terms contribute to the summation \eqref{equ:KR proof generating series}, allowing us to compute the local contributions term by term.
Since $T_\flat$ is totally positive definite, by the archimedean Siegel--Weil formula (Theorem \ref{thm:archi Siegel Weil}(1)), the hermitian matrices $T$ must either be totally positive definite or have signature $(n-1,1),(n,0),\cdots,(n,0)$.
In the latter case, the corresponding terms in $\CE_{T_\flat}(\tau,\phi_1,\phi_\flat)$ vanish by the archimedean arithmetic Siegel--Weil formula \cite{Liu11-I}, see also Theorem \ref{thm:archi Siegel Weil}. Consequently, the summation in \eqref{equ:KR proof generating series} involves only totally positive definite terms.

By induction and the cancellation laws on the geometric side (\cite[Lemma 2.11.1]{LZ21}) and on the analytic side (\cite[(3.2.1.1)]{LZ21}), almost all Fourier coefficients in the summation vanish.
By a careful choice of the global embedding, we can apply the vanishing result from \cite[Lemma 13.6]{Zhang21} to conclude that $\CE_{T_\flat}(\tau,\phi_1,\phi_\flat)\equiv 0$.
Taking the $t_1$-th Fourier coefficient, we get
$$
\sum_{\substack{
T=\begin{psmallmatrix}t_1& *\\ * & T_\flat\end{psmallmatrix}>0\\ 
\Diff(T,\BV)=\{v_{0}\}}
}
\left(\frac{\vol(K_G)}{\tau(Z^{\BQ})[E:F]}\wh{\deg}(\tCZ^{\bK}(T,\phi))-c_V\PEis_T(\tau,\phi)\right)=0.
$$

This is a finite sum which includes the index $\bT$ in \eqref{equ:ASWF}. 
By modifying the Schwartz function (a step that should have been included in our choice of global data), we can ensure that only the terms corresponding to $\bT$ are nonzero, thus completing our induction step.

To extend the result to arbitrary maximal parahoric level, two main modifications are required. First, additional error terms must be incorporated into the analytic side. Second, the vanishing result needs to be replaced by the double induction method developed in \cite[Lemma 15.1]{ZZhang24}. A key local ingredient in this generalization is the proof of the cancellation law on the analytic side (Corollary \ref{cor:ana-cancel}).

%\begin{remark}
%The error terms in previous work \cite{He-Li-Shi-Yang,Cho22,Cho-He-Zhang} were constructed by analyzing certain ``trivial'' test cases (see Proposition \ref{prop:unique-error-coeff}). The global proof provides a natural explanation for the sufficiency of these error terms: the Kudla--Rapoport conjecture at a place $v_0$ can be deduced from the ``trivial cases'' appearing in the generating series at other places.
%\end{remark}

\begin{remark}
Theorem \ref{prop:recursion} shows that the cancellation law on the analytic side holds term by term when $h\neq n$. The case $h=n$ is exceptional: in this case, the cancellation law holds only for the entire analytic function, not for individual terms (see the proof of Theorem \ref{thm:ana-cancel}). This is the unique pathological case where term-wise cancellation fails.
\end{remark}

\subsection{Open problems}
We conclude this introduction by discussing several open problems.

First, the formulation of the Kudla--Rapoport conjecture at deeper levels (e.g., Iwahori) or for weighted special cycles remains an open problems. At these levels, additional classes of special divisors emerge, and the almost modularity results become harder and more interesting.
Another promising direction is to establish a connection between our results and the Hecke algebra formulation of the Kudla--Rapoport conjecture, similar to \cite{Li-Rapoport-Zhang-HeckeAFL} and \cite{Li-Rapoport-Zhang-QuasiAFL}. See also the very recent work of He, Z. Zhang and Zhu \cite{HZZ25}.

Second, in the global proof, we eliminate all singular terms through the regularity assumption on $\phi_f$. This condition is fragile and fails to be preserved under the Weil representation $\phi_f\mapsto \omega(h_f)\phi_f$ in the adelic setting.
A natural open problem is to relax this regularity assumption and obtain results for singular coefficients. This relates to the works of R. Chen \cite{RChenI,RChenII,RChenIII,RChenIV} and Z. Guo \cite{ZGuo1,ZGuo2,ZGuo3}.

Third, in contrast to the local proof in \cite{LZ21}, we employ global modularity rather than local modularity. An open problem is to elucidate the relationship between these two forms of modularity, or to utilize global modularity to establish local modularity conjectures such as \cite[Conjecture 8.2]{ZZhang24}.

Fourth, the global proof does not directly extend to ramified cases, as we lack a suitable definition of the weighted integral special divisors under Fourier transform -- a prerequisite for applying the double induction method.
A natural question is to formulate the correct special divisors in the Krämer model (and more broadly, in splitting models \cite{He-Luo-Shi-regular}) that arise as Fourier transforms of special 
$\CZ$-divisors and $\CY$-divisors\footnote{In fact, in this case, even the geometric cancellation law becomes more subtle, see \cite[\S 6]{LRZ25}.}.

Finally, the global method appears applicable to the orthogonal case. However, this extension requires a precise definition of relative GSpin Rapoport--Zink spaces and corresponding comparison results in the totally real case, or a more careful treatment of the compactification for in the case $F_0=\BQ$ case. Since the first draft of the paper, this has been achieved in \cite{XLi26}.

\subsection{The structure of the paper}
In Part \ref{part:local}, we review the local geometric construction in \S \ref{sec:RZ space and special cycles}. 
In \S \ref{sec:weighted local density}, we define and study the weighted local density functions, then formulate the Kudla--Rapoport conjecture and prove the cancellation law on the analytic side.
In \S \ref{sec:symmetric space}, we review the local results at the archimedean places.
In Part \ref{part:global}, we define integral models of Shimura varieties and their special cycles in \S \ref{sec:RSZ} and \S \ref{sec:KR cycle in global}. 
In \S \ref{sec:Eisenstein series}, we review results on Eisenstein series.
In Part \ref{part:proof}, we define the arithmetic Picard group and its arithmetic intersection pairing in \S \ref{sec:generating}.
In \S \ref{sec:local-int}, we compute and compare the local contribution of the intersection numbers. 
In \S \ref{sec:generating modularity}, we construct the generating series \eqref{equ:KR proof generating series} and prove its modularity.
In \S \ref{sec:proof}, we prove the Kudla--Rapoport conjecture.

\subsection{Acknowledgments}
The author wishes to thank Michael Rapoport for his encouragement for this work, and for his comments and corrections on an earlier version of this manuscript.
The author wants to thank Andreas Mihatsch, Tonghai Yang and Wei Zhang for their encouragement and helpful discussions.
The author wants to thank Qiao He and referee for pointing out a gap on the analytic side in the first version of the draft.
The author gratefully acknowledges the valuable insights gained from \cite{LMZ25}, and thanks Andreas Mihatsch and Zhiyu Zhang for their collaboration.

\subsection{Notation and Conventions.}

\subsubsection{Hermitian space}\label{notation:herm}
Let $k/k_0$ be a quadratic extension of local fields or ring of adeles and let $(V,h)$ be a $k/k_0$-hermitian space of dimension $n$. Let $G=\U(V)$ be the corresponding unitary group.
\begin{enumerate}[label= H\arabic*,leftmargin=0.6cm]
\item Let $\Herm_m(k)$ be the set of $m\times m$ hermitian matrices and $\Herm_m^\reg(k)$ be the subset of all regular (non-singular) hermitian matrices.
\item We define the map
\begin{equation*}
    h:V^m\to \Herm_m(k),\quad \bu=(u_1,\cdots,u_m)\mapsto (h(u_i,u_j))_{i,j}.
\end{equation*}
\item\label{not:V_T} For any $T\in \Herm_m(k)$, we define the inverse image
\begin{equation*}
    V_T:=\{\bu\in V^m : h(\bu)=T\}=h^{-1}(T).
\end{equation*}
\item We define the regular subset
\begin{equation*}
    V^n_{\reg}:=\{\bu\in V^n\mid \det(h(\bu))\neq 0\}.
\end{equation*}
For any $\bu\in V^n_{\reg}$ with $h(\bu)=T$, the map
\begin{equation*}
    i_\bu:\U(V)\simlr V_T=h^{-1}(T),\quad g\mapsto g\cdot \bu,
\end{equation*}
is an isomorphism.

\item\label{not:gauge form} Fix basis elements $\alpha\in \wedge^{2n^2}(V^n)^*$ and $\beta \in \wedge^{n^2}(\Herm_n)^*$ and write $\alpha$ and $\beta$ for the translation invariant forms they define, where $(-)^*$ is the dual vector space. By \cite[after Corollary 10.4]{KR14}\footnote{Note that the construction is purely algebraic and hence also carries over to the local setup, see \cite[Lemma 5.3.1 and Lemma 5.3.2]{Kudla-Rapoport-Yang}.}, there is a form $\omega$ of degree $n^2$ on $V^n_{\reg}$ such that
\begin{itemize}
\item $\alpha=h^*\beta\wedge \omega$,
\item $\omega$ is invariant under the action of $\U(V)\times \GL(n)$ on $V^n$,
\item for all points $\bu\in V^n_{\reg}$, the restriction of $\omega$ to $\ker(dh_\bu)$ is non-zero.
In particular, the pullback $\omega_{G}=i_\bu^*\omega$ is a gauge form on $G$, which is independent of the choice of $\bu$.
\end{itemize}
\item\label{not:gauge form 2} 
Similarly, we define the constant $c(\alpha,\psi)$ by
\begin{equation*}
    d_{\alpha}x=c(\alpha,\psi)dx.
\end{equation*}
Here $d_{\alpha}x$ is the measure on $V^n$ determined by the gauge form $\alpha$ and $dx$ is the self-dual measure with respect to $\langle x,y\rangle=\psi(\tr_{F/F_0}(\tr(h(x,y))))$.
\item\label{not:gauge form 3} Similarly, we define a constant:
\begin{equation*}
d_{\beta}b=c(\beta,\psi)db.
\end{equation*}
Here $d_{\beta}b$ is the measure on $\Herm_n(F)$ determined by the gauge form $\beta$ and $db$ is the self-dual measure with respect to $\langle b,c\rangle=\psi(\tr(bc))$.
\end{enumerate}

\subsubsection{Local data}\label{not:loc}
\begin{enumerate}[label= L\arabic*,leftmargin=0.6cm]
\item\label{not:local fields} In the local setup, we fix an unramified quadratic extension of $p$-adic fields $F/F_0$ with uniformizer $\varpi$. Assume $p\neq 2$.
\item\label{not:vertex lattice} Let $V$ be a hermitian space and let $\Lambda\subseteq V$ be a hermitian lattice. The hermitian dual $\Lambda^\sharp$ is defined by
$
\Lambda^\sharp:=\{x\in V \mid (x,\Lambda)\subset O_F\}.
$
A \emph{vertex lattice} is a hermitian lattice $\Lambda\subset V$ satisfying $\varpi\Lambda^\sharp\subseteq \Lambda\subseteq\Lambda^\sharp$. The \emph{type} of a vertex lattice $\Lambda$ is $t(\Lambda):=\dim\Lambda^\sharp/\Lambda$.

\item\label{not:admissible function} Let $0\leq r\leq n$ and let $L\subset V$ be a vertex lattice. A Schwartz function $\phi\in \CS(V^r)$ is called \emph{admissible} if $\phi=\bigotimes_{i=1}^r \phi_i$ with $\phi_i=\charfun_L$ or $\charfun_{L^\sharp}$. It is of \emph{type $t$} if the vertex lattice $L$ has type $t$.
A \emph{weight vector} is an element $\bw\in \{\CZ,\CY\}^r$. An \emph{admissible Schwartz function of weight $\bw$ and type $t$} is a function $\phi:=\phi^{[t]}_{\bw}\in\CS(V^r)$ such that 
$$
\phi=\bigotimes_{i=1}^r \phi_i,\quad \phi_i=\left\{\begin{array}{cc}
\charfun_{L}     &\text{if }\bw_i=\CZ;\\
\charfun_{L^\sharp}  &\text{if }\bw_i=\CY,
\end{array}\right.
$$
where $L=L^{[t]}$ is a vertex lattice of type $t$. An admissible Schwartz function is an admissible function of some weight $\bw$ and type $t$.  
\item\label{not:local match schwartz function}
Two admissible Schwartz functions $\phi_i\in \CS(V^r)$ ($i\in\{1,2\}$) are \emph{matched} if they share the same weight vector. 
An admissible Schwartz function $\phi$ of type $0$ is regarded as matched with every admissible Schwartz function.
\end{enumerate}

\subsubsection{Global data}
\begin{enumerate}[label= G\arabic*,leftmargin=0.6cm]
\item For a global number field $K$, let $\Sigma_{K}$ be the set of all places. Let $\Sigma_{K,f}$ (resp. $\Sigma_{K,\infty}$) be the set of finite (resp. infinite) places.
Let $\fkd\subset O_K$ be an ideal and let $\Sigma_{K,\Box}^{\fkd}:=\Sigma_{K,\Box}\setminus\{v\in \Sigma_K : v\mid \fkd\}$ for $\Box\in\{\text{empty},f,\infty\}$.

\item\label{not:nearby herm} 
Let $F/F_0$ be a CM extension with fixed CM type $\Phi$ and a distinguished element $\varphi_0\in\Phi$. Let $V$ be a hermitian space of signature $(n-1,1)$ at $\varphi_0$ and $(n,0)$ for $\varphi\in\Phi\setminus \{\varphi_0\}$.
Let $v_0$ be a place of $F_0$. If $v_0$ is non-split or archimedean, let $V^{(v_0)}$ be the \emph{$v_0$-nearby hermitian space} of $V$. 
We recall that $V^{(v_0)}\simeq V$ if $\varphi_0\in\Phi$ is the unique archimedean place over $v_0$, and otherwise $V^{(v_0)}$ is characterized by the following conditions:
\begin{itemize}
\item For all $w_0$ other than $v_0$ and $\varphi_0$, we have $V^{(v_0)}_{w_0}\simeq V_{w_0}$.
\item $V^{(v_0)}_{\varphi_0}$ is positive definite.
\item $V^{(v_0)}_{v_0}$ has signature $(n-1,1)$ if $v_0$ is archimedean and $\varphi_0$ is not over $v_0$.
\end{itemize}

\item For a CM extension $F/F_0$ with CM type $\Phi$, let $\Herm_{m}^+(F)$ be the set of all totally positive definite matrices of rank $m$, and $\Herm_{n}^\dagger(F)$ be the set of hermitian matrices that have signature $(n-1,1)$ for one $\varphi\in\Phi$, and positive definite for all $\varphi'\in\Phi\setminus \{\varphi\}$.

\item Let $\BA_0=\BA_{F_0}$ and let $\BA=\BA_F$. We fix the non-trivial additive character $\psi:\BA_{F_0}/F_0\to\BC^\times$ such that $\psi$ is unramified outside $\mathrm{Spl}(F/F_0)$ (the set of finite places of $F_0$ split in $F$). See \cite[\S 12.2]{LZ21}.

\item We let $\bH=\SL_2$, regarded as an algebraic group over $F_0$. Let $\CA_\infty(\bH(\BA_0),K_{\bH},k)$ be the space of smooth functions on $\bH(\BA_0)$, invariant under left $\bH(F_0)$ and right $K_\bH\subset \bH(\BA_{0,f})$ translation, and of parallel weight $k$ under the action of $\prod_{v_0\in\Hom(F_0,\BR)}\SO(2,\BR)$ (we do not impose any finiteness condition under the center of the universal enveloping algebra).

\item In this paper, we use the classical notation for modular forms on the hermitian upper half space
\begin{equation*}
\BH_{1}(F_\infty)=\{\tau=x+iy\in F_\infty :\frac{1}{2i}(\tau-\ov{\tau})>0\},\quad\text{where}\quad F_\infty:=F\otimes_{\BQ}\BR\cong \BC^{\Phi}.
\end{equation*}
For any $f\in \CA_\infty(\bH(\BA_0),K_{\bH},k)$, define the classical modular function 
$$
f(\tau):\BH_{1}(F_\infty)\to \BC,\quad \tau\mapsto \chi_\infty(\det(a))^{-1}\det(y)^{-k/2}\cdot f(g_\tau),
$$
where $\tau=x+iy$ and $a\in F_\infty^\times$ is such that $y=a\ov{a}$, and 
$$g_\tau=\begin{pmatrix}
1&x\\&1
\end{pmatrix}\cdot \begin{pmatrix}
a&\\ &\ov{a}
\end{pmatrix}\cdot 1_f\in\bH(\BA_0).$$ 
Thanks to the strong approximation theorem $\bH(\BA_{0})=\bH(F_0)\bH(F_{0,\infty})K_{\bH}$, we do not lose information.
The Fourier expansion becomes
\begin{equation*}
    f(\tau)=\sum_{\xi\in F_0}f_\xi(y)q^\xi,\quad q^\xi=\psi_\infty(\tr(\xi\tau)).
\end{equation*}

\item $\CA_{\hol}(\bH(\BA_0),K_\bH,k)$: the space of automorphic forms (with moderate growth) on $\bH(\BA_0)$, invariant under $K_{\bH}\subset \bH(\BA_{0,f})$, of parallel weight $k$ under the action of $\prod_{v_0\in\Hom(F_0,\BR)}\SO(2,\BR)$, and holomorphic (i.e., annihilated by the element $\frac{1}{2}\begin{psmallmatrix}i&1\\ 1&-i\end{psmallmatrix}$ in the complexified Lie algebra of $\bH(F_{0,v_0})=\SL_2(\BR)$ for every $v_0\in\Hom(F_0,\BR)$). 
This is a finite dimensional vector space over $\BC$, and it has a $\ov{\BQ}$-structure via Fourier expansion.
For any subfield $L\subset \BC$ (for our application, it suffices to assume $L\supset \ov{\BQ}$), we define $\CA_{\hol}(\bH(\BA_0),K_\bH,k)_L$ as the subspace of $\CA_{\hol}(\bH(\BA_0),K_\bH,k)$ consisting of functions with Fourier coefficients all in $L$. Then for any $L$-vector space $W$, we have an $L$-vector space $\CA_{\hol}(\bH(\BA_0),K_\bH,k)_L\otimes_L W$. Since the differential operator commutes with the right $\bH(\BA_f)$-action, it suffices to verify the holomorphicity at $h_f=1$.
\end{enumerate}

\part{Local theory}\label{part:local}
\section{Rapoport--Zink spaces and special cycles}\label{sec:RZ space and special cycles}
In this section, we define unramified basic relative Rapoport--Zink spaces with maximal parahoric level, and weighted special cycles on them.
Throughout the section, we fix $F/F_0$, an unramified quadratic extension of $p$-adic fields ($p\neq 2$) with uniformizer $\varpi\in F_0$. We assume that the residue field of $O_{F_0}$ has cardinality $q$.

\subsection{Strict $O_{F_0}$-modules}\label{sec:strict modules}
We begin with a review of strict $O_{F_0}$ modules. For a comprehensive treatment, we refer to \cite{Mih22,KRZ,LMZ25}.

For an $\Spf O_{F_0}$-scheme $S$, a \emph{strict $O_{F_0}$-module} over $S$ is a pair $(X,\iota)$ where $X$ is a $p$-divisible group over $S$ and $\iota: O_{F_0}\rightarrow \End(X)$ is an action 
such that $O_{F_0}$ acts on $\Lie(X)$ via the structure morphism $O_{F_0}\rightarrow \CO_S$.
A strict $O_{F_0}$-module $(X,\iota)$ is called \emph{formal} if its underlying $p$-divisible group $X$ is formal. The dimension of a strict $O_{F_0}$-module is the dimension of its underlying $p$-divisible group. For the notion of (relative) \emph{height} of strict $O_{F_0}$-modules, we direct the reader to the aforementioned references.

Let $D(X)$ be the (relative) de Rham homology of $X$, which is a locally free $\CO_S$-module whose rank equals the height of $X$. It fits into a canonical short exact sequence of $\CO_S$-modules:
\begin{equation*}
    0\rightarrow \mathrm{Fil}(X)\rightarrow D(X)\rightarrow \Lie(X)\rightarrow 0.
\end{equation*}

We now specialize to the case where $X=(X,\iota)$ is \emph{biformal}, a notion introduced in \cite[Definition 11.9]{Mih22}. For biformal strict $O_{F_0}$-modules, we can define the \emph{relative dual} $X^\vee$ and consequently the notion of \emph{relative polarization} (see \cite[Definition 11.9]{Mih22}). This notation of relative dual agrees with the one defined in \cite{Faltings-dual}; see \cite[Appendix]{LMZ25}.

\subsection{Rapoport--Zink space}\label{sec:RZ space}
For any $\Spf O_{F}$-scheme $S$, a \emph{hermitian $O_F$-module of dimension $n$ and type $h$} (and signature $(n-1, 1)$) over $S$ is a triple $(X, \iota, \lambda)$, where:
\begin{itemize}
\item
$X$ is a strict biformal $O_{F_0}$-module over $S$ of relative height $2n$ and dimension $n$. 
\item $\iota: O_F \to \End(X)$ is an action of $O_F$ on $X$ that extends the action of $O_{F_0}$. We require that the \emph{Kottwitz signature condition} of signature $(n-1, 1)$ holds for all $a \in O_{F} $:
\begin{equation*}
\mathrm{charac} (\iota(a)\mid \Lie X) = (T-a)^{n-1}(T-\ov a) \in \CO_S[T].
\end{equation*} 
\item $\lambda: X \to X^\vee$ is a polarization on $X$ as strict $O_{F_0}$-module, which is $O_F/O_{F_0}$-semi-linear in the sense that the Rosati involution $\mathrm{Ros}_\lambda$ induces the non-trivial involution $\ov{(-)} \in \Gal({F}/{F_0})$ on $\iota: O_{F} \to \End(X)$.
\item We require that the finite flat group scheme $\ker\lambda$ over $S$ is contained in $X[\varpi]$ and is of order $q^{2h}$.
\end{itemize}

An isomorphism $(X_1, \iota_1, \lambda_1) \overset{\sim}{\longrightarrow} (X_2, \iota_2, \lambda_2)$ between two such triples is an $O_{F}$-linear isomorphism $\varphi\colon X_1 \overset{\sim}{\longrightarrow} X_2$ such that $\varphi^*(\lambda_2)=\lambda_1$.   

Up to $O_F$-linear quasi-isogeny compatible with the polarization, there exists a unique hermitian $O_F$-module $(\BX_n^{[h]}, \iota_{\BX_n^{[h]}}, \lambda_{\BX_n^{[h]}})$ over $\BF=\ov{\BF}_q$ of dimension $n$ and type $h$, within the
supersingular quasi-isogeny class. Fix one choice of $(\BX_n^{[h]}, \iota_{\BX_n^{[h]}}, \lambda_{\BX_n^{[h]}})$ as the \emph{(basic) framing object}.

\begin{definition}\label{defn: local RZ space}
The (relative) Rapoport--Zink space of dimension $n$ and type $h$ is the functor 
$$\CN_n^{[h]} \to \Spf O_{\breve F} $$
sending $S$ to the set of isomorphism classes of tuples $(X, \iota, \lambda, \rho)$, where:
\begin{itemize}
\item $(X, \iota, \lambda)$ is a hermitian $O_F$-module of dimension $n$ and type $h$ over $S$.
\item $\rho: X \times_{S} {\ov S} \to \BX_n^{[h]} \times_\BF \ov S $ is an $O_F$-linear quasi-isogeny of height $0$ over the reduction $\ov S:=S \times_{O_{\breve F}} \mathbb F$ such that $\rho^*(\lambda_{\BX_n^{[h]}, \ov S}) = \lambda_{\ov S}$.
\end{itemize}
\end{definition}
	
By \cite[Theorem 2.16]{RZ96} and \cite{Cho22}, the functor $\CN_n^{[h]}$ is representable by a formal scheme that is flat, regular, and locally formally of finite type over $\Spf O_{\breve F}$ of dimension $n$.

\subsection{Special cycles}\label{sec:non-archi-cycle}

Let $\BE$ be a formal $O_{F_0}$-module over $\BF$ of relative height $2$ and dimension $1$. We can complete it into a triple $(\BE, \iota_{\BE}, \lambda_{\BE})$ of \emph{signature $(1, 0)$}, dimension $1$ and type $0$ over $\BF$, which we still denote by $\BE$. By modifying the $O_F$-action on $\BE$ via the Galois involution $\ov{(-)}$, we obtain a framing object $(\ov \BE, \iota_{\ov \BE} \circ \ov{(-)}, \lambda_{\ov \BE})$ for $\CN_1^{[0]}$, which we denote by $\ov{\BE}$.  The theory of canonical liftings \cite{Gross86} yields an isomorphism $\CN_1^{[0]} \simeq \Spf O_{\breve F}$, where the universal triple $\ov{\CE}$ over $\CN_1^{[0]} \simeq \Spf O_{\breve F}$ is the canonical lifting of $\ov{\BE}$.

The \emph{space of special quasi-homomorphisms} is the $F$-vector space
\begin{equation}\label{eq: space of special quasi-homos}
		\BV=\BV_n^{[h]}=\Hom_{O_F}^{\circ}(\BE, \BX_n^{[h]}):=\Hom_{O_F}(\BE, \BX_n^{[h]})\otimes F
\end{equation}
equipped with the hermitian form such that for any $x, \, y \in \BV$
\begin{equation}
	(x,y)_{\BV}:= \lambda_{\BE}^{-1} \circ y^{\vee} \circ \lambda_{\BX} \circ x \in \Hom_{O_F}^{\circ}(\BE, \BE) \cong F.
\end{equation}
	
The group $\Aut^\circ(\BX, \iota_{\BX}, \lambda_{\BX})$ of quasi-isogenies of $\BX$ to itself acts naturally on $\BV$.
By Dieudonn\'e theory, we have $\Aut^\circ(\BX, \iota_{\BX}, \lambda_{\BX}) \cong \U(\BV)(F_0)$. The group $\U(\BV)(F_0)$ acts on $\CN_n^{[h]}$ by
\begin{equation}
g.(X, \iota, \lambda, \rho)=(X, \iota, \lambda, g \circ \rho).
\end{equation}

\begin{definition}\label{def:special divisor}
For any $u\in \BV_n^{[h]}$, the \emph{Kudla--Rapoport special divisor} $\CZ(u) \subseteq \CN_n^{[h]}$ is the closed formal subscheme of $\CN_n^{[h]}$ characterized as follows: for each $\Spf O_{\breve{F}}$-scheme $S$, the set $\CZ(u)(S)$ consists of those points $(X, \iota_X, \lambda_X, \rho_X) \in \CN_n^{[h]}(S)$ for which the quasi-homomorphism
\begin{equation*}
\rho_X^{-1} \circ u: \BE\times_{\BF} \ov{S} \longrightarrow X \times_S \ov{S}
\end{equation*}
extends to a homomorphism $\CE_S \to X$ over $S$.

Similarly, the \emph{Kudla--Rapoport special divisor} $\CY(u) \subseteq \CN_n^{[h]}$ is the closed formal subscheme of $\CN_n^{[h]}$ characterized as follows: for each $\Spf O_{\breve{F}}$-scheme $S$, the set $\CY(u)(S)$ consists of those points $(X, \iota_X, \lambda_X, \rho_X) \in \CN_n^{[h]}(S)$ for which the quasi-homomorphism
\begin{equation*}
\rho_{X^\vee}^{-1} \circ \lambda_{\BX} \circ u: \BE \times_{\BF} \ov{S} \longrightarrow X^\vee \times_S \ov{S}
\end{equation*}
extends to a homomorphism $\CE_S \to X^\vee$ over $S$.
\end{definition}

\begin{proposition}\label{prop:special-div}
The following properties hold for special divisors:
\begin{enumerate}
\item For any nonzero $u\in \BV$, the functors $\CZ(u)$ and $\CY(u)$ are representable by Cartier divisors on $\CN_n^{[h]}$.
\item For $0\leq h\leq n-1$, if $\val((u,u))=0$, then $\CZ(u)\cong \CN_{n-1}^{[h]}$.
\item For $1\leq h\leq n$, if $\val((u,u))=-1$, then $\CY(u)\cong \CN_{n-1}^{[h-1]}$.
\item The following inclusion relations hold:
\begin{equation*}
    \CZ(u)\subseteq \CY(u)\subseteq \CZ(\varpi u).
\end{equation*}
Moreover, for $h=0$, we have $\CZ(u)=\CY(u)$, and for $h=n$, we have $\CY(u)=\CZ(\varpi u)$.
\end{enumerate}
\end{proposition}

\begin{proof}
Part (1) follows from the same argument as in \cite[Proposition 5.9]{KR11}. Parts (2) and (3) are established in \cite[Proposition 5.10]{Cho18}. Part (4) follows directly from the definition.
\end{proof}

Next, we introduce weighted special cycles.

\begin{definition}\label{def:local special cycle}
Let $0\leq m\leq n$ be an integer, and let $V$ be the ``nearby'' hermitian space associated with $\BV=\BV^{[h]}_n$, i.e., the unique hermitian space of dimension $n$ containing a vertex lattice $L$ of type $h$. 

Recall from \eqref{not:admissible function} that a Schwartz function $\phi\in\CS(V^m)$ is \emph{admissible of type $h$} if $\phi=\bigotimes_i\phi_i$, where each $\phi_i$ is either $\charfun_{L^{[h]}}$ or $\charfun_{L^{[h],\sharp}}$, for a vertex lattice $L^{[h]}$ of type $h$.

\begin{enumerate}
\item For any nonzero $u\in \BV$, we define
\begin{equation*}
    \CZ(u,\charfun_{L^{[h]}}) := \CZ(u), \quad \text{and} \quad \CZ(u,\charfun_{L^{[h],\sharp}}) := \CY(u).
\end{equation*}

\item For nonzero vectors $\bu = (u_1,\dots,u_m) \in \BV^m$ and an admissible Schwartz function $\phi\in\CS(V^m)$ of type $h$, the \emph{weighted special cycle} is defined as the scheme-theoretic intersection:
\begin{equation*}
    \CZ(\bu,\phi) := \CZ(u_1,\phi_1) \cap \dots \cap \CZ(u_m,\phi_m).
\end{equation*}

\item The \emph{derived weighted special cycle} $\DCZ(\bu,\phi)$ is defined as the image of
\begin{equation*}
    \CO_{\CZ(u_1,\phi_1)} \Dotimes \dots \Dotimes \CO_{\CZ(u_m,\phi_m)}
\end{equation*}
in the $m$-th graded piece $\mathrm{Gr}^m K_0^{\CZ(\bu,\phi)}(\CN_n^{[h]})$. 

\item When $m=n$, the \emph{weighted arithmetic intersection number} is defined as
\begin{equation*}
    \Int(\bu,\phi) := \chi(\CN_n^{[h]},\DCZ(\bu,\phi)),
\end{equation*}
where $\chi$ denotes the Euler--Poincaré characteristic. 
\end{enumerate}
\end{definition}

\begin{example}\label{exmp:base case geo}
If $n=1$, then by the theory of canonical liftings \cite{Gross86}, we have:

\begin{enumerate}
\item For any $u$ in the non-split hermitian space $\BV_1^{[0]}$, 
\begin{equation*}
    \Int(u, \charfun_{L^{[0]}}) =
    \begin{cases}
        \frac{\val(u,u) + 1}{2}, & \text{if } \val(u,u) \geq 0; \\
        0, & \text{otherwise}.
    \end{cases}
\end{equation*}

\item For any $u$ in the split hermitian space $\BV_1^{[1]}$,
\begin{equation*}
    \Int(u, \charfun_{L^{[1]}}) =
    \begin{cases}
        \frac{\val(u,u)}{2}, & \text{if } \val(u,u) \geq 1; \\
        0, & \text{otherwise}.
    \end{cases}
\end{equation*}
This follows from the observation that, after identifying the hermitian modules $\BX^{[1]}$ and $\BX^{[1],\vee}$ with $\BE$, the polarization $\BX^{[1]} \to \BX^{[1],\vee}$ corresponds to multiplication by $[\varpi]: \BE \to \BE$. 

\item For any $u$ in the split hermitian space $\BV_1^{[1]}$,
\begin{equation*}
    \Int(u, \charfun_{L^{[1],\sharp}}) =
    \begin{cases}
        \frac{\val(u,u)}{2} + 1, & \text{if } \val(u,u) \geq -1; \\
        0, & \text{otherwise}.
    \end{cases}
\end{equation*}
This follows from (2) and Proposition \ref{prop:special-div}(4).
\end{enumerate}
\end{example}

By Proposition \ref{prop:special-div}(2), (3), and (4), we obtain the following geometric cancellation law.

\begin{proposition}[\protect{\cite[Proposition 5.10]{Cho18}}]\label{prop:geo-cancel}

Let $\phi^{[h]} = \phi^{[h]}_\bw \in \CS(V^{n})$ be an admissible Schwartz function of weight $\bw\in\{\CZ,\CY\}^n$ \eqref{not:admissible function}, and let $\bu = (u_1, \dots, u_n) \in \BV^n$ be a collection of linearly independent vectors.
For any fixed $1 \leq i \leq n$, define:
\begin{itemize}
    \item $\BV_{\flat} := \langle u_i \rangle^\perp$ and let $u_{j,\flat}$ be the orthogonal projection of $u_j$ to $\BV_\flat$ for $j\neq i$.
    \item $\bu_{\flat} := (u_{1,\flat}, \dots, \wh{u_{i,\flat}}, \dots, u_{n,\flat})\in \BV_\flat^{n-1}$.
    \item $\bw_\flat := (\bw_1, \dots, \wh{\bw_i}, \dots, \bw_n)\in\{\CZ,\CY\}^{n-1}$.
\end{itemize}
Then we have the following:
\begin{enumerate}
\item If $\bw_i=\CZ$ and 
$$
\val(u_i,u_i)=\left\{\begin{array}{ll}
0&\text{if }h\neq n;\\
1&\text{if }h=n,
\end{array}\right.
$$
then $\Int(\bu,\phi_{\bw}^{[h]})=\Int(\bu_{\flat},\phi_{\bw_\flat}^{[h]})$.
\item If $\bw_i=\CY$ and 
$$
\val(u_i,u_i)=\left\{\begin{array}{ll}
-1&\text{if }h\neq 0;\\
0&\text{if }h=0,
\end{array}\right.
$$
then $\Int(\bu,\phi_{\bw}^{[h]})=\Int(\bu_{\flat},\phi_{\bw_\flat}^{[h-1]})$.\qed
\end{enumerate}
\end{proposition}

The following result plays a crucial role in the globalization procedure.
\begin{theorem}[\cite{Mih-LC}]\label{thm:local-constancy}
For any admissible Schwartz function $\phi\in\CS(V^n)$ of type $h$, the weighted arithmetic intersection number $\Int(\bu,\phi)$ is a locally constant function of $\bu$ on the set of linearly independent tuples in $\BV^n$.
\end{theorem}

\begin{proof}
By \cite[Proposition 4.2]{Mih-LC}, the special cycles $\CZ(u)$ and $\CY(u)$ satisfy the local constancy axiom in loc.cit. 
The result then follows from \cite[Theorem 1.2]{Mih-LC}.
\end{proof}

\section{Weighted local density functions and the Kudla--Rapoport conjecture}\label{sec:weighted local density}
In this section, we study the weighted local density function introduced in \cite{Kudla97} and formulate the Kudla--Rapoport conjecture for weighted special cycles.
We show that the weighted local density function satisfies the \emph{cancellation law}, mirroring the corresponding geometric cancellation law (Proposition \ref{prop:geo-cancel}). Finally, we recall local Whittaker functions and the local Siegel--Weil formula, and their relation to the weighted local density function.

\subsection{Weighted local density function}\label{sec:def weighted local density}
Let $F/F_0$ be an unramified quadratic extension of $p$-adic fields, equipped with a fixed unramified additive character $\psi$. Let $(V, h)$ be a hermitian space of dimension $n$. When no confusion arises, we denote $h(x,y)$ simply by $(x,y)$. 

The character $\psi$ determines self-dual Haar measure $dx$ on $V$ and $db$ on $\Herm_n(F)$.
Let $dg = dg^{\mathrm{self}}$ be the induced Haar measure on $\U(V)$ as given in \eqref{not:gauge form}.

\begin{definition}\label{def:weighted local density function}
For any natural number $r \geq 0$, define $V^{\angler} := V \oplus V_{r,r}$, where $V_{r,r}$ is the split hermitian space of dimension $2r$, i.e., a hermitian space of dimension $2r$ with basis $\{e_i, f_j \mid 1 \leq i,j \leq r\}$ satisfying $(e_i, f_j) = \delta_{ij}$.

\begin{enumerate}
\item For any Schwartz function $\phi \in \CS(V^n)$, define
\begin{equation*}
    \phi^{\angler} = \phi \otimes \phi_{r,r} \in \CS((V^{\angler})^n),
\end{equation*}
where $\phi_{r,r}$ is the characteristic function of $(L_{r,r})^n$ for a self-dual lattice $L_{r,r}$ in $V_{r,r}$.

\item For any $T \in \Herm_n^{\reg}(F)$, define the complex-valued \emph{weighted local density function} $\alpha_T(r, \phi)$ as the integral
\begin{equation*}
    \alpha_T(r, \phi) := \int_{\Herm_n(F)} \int_{(V^{\angler})^n} \psi\Bigl(\tr\bigl(b((\bu,\bu) - T)\bigr)\Bigr) \phi^{\angler}(\bu) \, d\bu \, db.
\end{equation*}
This notion was introduced by Kudla \cite[Appendix]{Kudla97}. 
\end{enumerate}

When $\phi$ is the characteristic function of $L^n$ for some lattice $L$, it coincides with the representation density function; see \cite[Proposition A.6]{Kudla97} or \S \ref{sec:base case}. The definition is independent of the choice of the self-dual lattice $L_{r,r}$.
\end{definition}

In this paper, we focus on admissible functions $\phi \in \CS(V^m)$ as defined in \eqref{not:admissible function}.  
For clarity, throughout this section, we restrict our attention to Schwartz functions of the form:
\begin{equation}\label{equ:special phi ab}
    \phi^{[t]}_{a,b} := (\charfun_{L^{[t]}})^{\otimes a} \otimes (\charfun_{L^{[t],\sharp}})^{\otimes b},
\end{equation}
where $L^{[t]}$ is a vertex lattice of type $t$.
The results presented in this section extend naturally to general admissible Schwartz functions.

Recall that in the unramified case, the type of a hermitian lattice uniquely determines the Hasse invariant of the hermitian space.
By a direct change of variables, we obtain the following result.

\begin{lemma}\label{lem:ind-lattice}
Let $\phi$ be an admissible Schwartz function of type $t$, and let $T \in \Herm_n(F)$.  
The integral $\alpha_T(r, \phi)$ is independent of the choice of the vertex lattice $L^{[t]}$ of type $t$.
\qed
\end{lemma}

%\begin{proof}
%Without loss of generality, we may assume $r=0$. Let $L_1$ and $L_2$ be two vertex lattices of type $t$. There exists $g\in \U(V)$ that $gL_1=L_2$. Denote the corresponding Schwartz functions by $\phi_1$ and $\phi_2$ resp. By substituting from $x$ to $gx$, we have
%\begin{align*}
%\alpha_T(0,\phi_1)={}&\int_{\Herm_n(F)}\int_{V_{\reg}^n}\psi\Bigl(\tr\bigl(b((x,x)-T)\bigr)\Bigr)\phi_1(x)dxdb,\\
%={}&\int_{\Herm_n(F)}\int_{V_{\reg}^n}\psi(\tr(b(gx,gx)-T))\phi_1(gx)dgx db,\\
%={}&\psi(\tr_{F/F_0}(\det(g)))\int_{\Herm_n(F)}\int_{V_{\reg}^n}\psi(\tr(b(x,x)-T))\phi_2(x)dx db,\\
%={}&\alpha_T(0,\phi_2).\qedhere
%\end{align*}
%\end{proof}

\subsection{A Fubini-type result}  
In this subsection, we state and prove a Fubini-type result, which was suggested and proven by T. Yang.  

Recall from \eqref{not:V_T} that for any $T \in \Herm^{\reg}_n(F)$, the inverse image $V_T := h^{-1}(T)$ consists of vectors $\bu \in V^n$ satisfying $(\bu, \bu) = T$.  

\begin{theorem}\label{thm:bertini}  
Let $T \in \Herm^{\reg}_n(F)$ be a nonsingular matrix. For any admissible Schwartz function $\phi \in \CS(V^n)$, we have the equality:  
\begin{equation*}
    \alpha_T(r, \phi) = \int_{(V^{\angler})^n} \int_{\Herm_n(F)} \psi\Bigl(\tr\bigl(b((u,u) - T)\bigr)\Bigr) \phi^{\angler}(u) \, db \, du
    = \int_{V_T^{\angler}} \phi^{\angler}(u) \, du.
\end{equation*}
\end{theorem}
\begin{proof}
The degenerate subspace  
\begin{equation*}
    \{\bu\in (V^{\angler})^n \mid \det (\bu,\bu) = 0\}
\end{equation*}  
is Zariski closed and has measure zero.
Similarly, the degenerate subset $\Herm_n(F) \setminus \Herm_n^{\reg}(F)$ has measure zero. Therefore, we have  
\begin{align*}
    \alpha_T(r, \phi) &= \int_{\Herm_n(F)}\int_{(V^{\angler})_{\reg}^n} \psi\Bigl(\tr\bigl(b((\bu,\bu) - T)\bigr)\Bigr) \phi^{\angler}(\bu) \, d\bu \, db, \\
    &= \int_{\Herm_n(F)}\int_{\Herm_n^{\reg}(F)}\int_{V_T^{\angler}} \psi(\tr(b(\wt{T} - T))) \phi^{\angler}(\bu_{\wt{T}}) \, d\bu_{\wt{T}} \, d\wt{T} \, db.
\end{align*}

For any $T \in \Herm_n^\reg(F)$, define  
\begin{equation*}
    M(T) = M(T, \phi) := \int_{V_T} \phi^{\angler}(\bu) \, d\bu = \int_{\U(V)} \phi^{\angler}(g\bu_T) \, dg,
\end{equation*}  
where $\bu_T$ is any vector in $V_T$, and $dg = dg^{\mathrm{self}}$.  

The function $M(T, \phi)$ is continuous on $\Herm_n^{\reg}(F)$ and extends to $\Herm_n(F)$ (as an $L^1$-function or a distribution). Indeed, for any $\wt{T}_1$ and $\wt{T}_2$ in a small neighborhood $U \subset \Herm_n^{\reg}(F)$, the continuity of the map $h: (V^{\angler})^n \to \Herm_n(F)$ ensures that we can find $\bu_{\wt{T}_1}$ and $\bu_{\wt{T}_2}$ in a small neighborhood of the preimage $h^{-1}(U)$. Since the function  
\begin{equation*}
    \bu \mapsto \int_{\U(V^{\angler})} \phi^{\angler}(g\bu) \, dg
\end{equation*}  
is continuous, it follows that $M(T)$ is continuous.  

Since $\phi^{\angler}$ has compact support, $M(T)$ is a Schwartz function on $\Herm_n(F)$. We have  
\begin{align*}
    \alpha_T(r, \phi) &= \int_{\Herm_n(F)}\int_{\Herm_n^{\reg}(F)} \psi(\tr(b(\wt{T} - T))) M(\wt{T}) \, d\wt{T} \, db, \\
    &= \int_{\Herm_n(F)} \psi(\tr(-bT)) \wh{M}(-b) \, db = \wh{\wh{M}}(-T) = M(T),
\end{align*}  
where in the last line, $\wh{M}(b)$ denotes the Fourier transform of $M$, and the final equality follows from the Fourier inversion formula. The proof is completed by evaluating these integrals.
\end{proof}

\subsection{A recursion formula}  
In this subsection, we establish a recursion formula for weighted local density functions, which will play a crucial role in proving the analytic cancellation law.  
We begin with a lemma that is well known to experts.  

\begin{lemma}\label{lem:lattice-orth}  
Let $L\subset V$ be a vertex lattice, so that $\varpi L^\sharp\subseteq L\subseteq L^\sharp$. Let $u\in V$.
Then the following statements hold:  
\begin{enumerate}  
\item If $\val(u,u)=0$ and $u \in L$, then  
\begin{equation*}
    L = \langle u \rangle \obot L_\flat \quad \text{and} \quad L^\sharp = \langle u \rangle \obot L_\flat^\sharp.
\end{equation*}

\item If $\val(u,u)=-1$ and $u \in L^\sharp$, then  
\begin{equation*}
    L = \langle \varpi u \rangle \obot L_\flat \quad \text{and} \quad L^\sharp = \langle u \rangle \obot L_\flat^\sharp.
\end{equation*}
\end{enumerate}
\end{lemma}

\begin{proof}  
The Lemma follows from the Gram--Schmidt process.
For case (1), let $x \in L$ be any vector. Then we can write  
\begin{equation*}
    x = \frac{(x,u)}{(u,u)} u + \left(x - \frac{(x,u)}{(u,u)} u\right).
\end{equation*}  
Since $x \in L \subseteq L^\sharp$ and $\val(u,u) = 0$, we have $\frac{(x,u)}{(u,u)} \in O_F$.  
Thus, applying the Gram--Schmidt process, we obtain the orthogonal complement $\langle u \rangle^\perp \subset L$. The same argument applies to $L^\sharp$.  

For case (2), one can directly use the dual hermitian space (see \cite[after Proposition 5.3]{ZZhang24}), but we provide a direct proof here.  
Since $u \in L^\sharp$, it follows that $\varpi u \in \varpi L^\sharp \subseteq L$. Moreover, since $(\varpi u, u)$ is a unit, we conclude that $u \in L^\sharp$ and $\varpi u \in L$ are primitive vectors.  

For any $x \in L$, consider the decomposition  
\begin{equation*}
    x = \frac{(x,u)}{(u,u)} u + \left(x - \frac{(x,u)}{(u,u)} u\right) 
    = \frac{(x,u)}{\varpi (u,u)} \varpi u + \left(x - \frac{(x,u)}{\varpi (u,u)} \varpi u\right).
\end{equation*}  
Since $x \in L$ and $u \in L^\sharp$, we have $(x,u) \in O_F$. Given that $\val(u,u) = -1$, we obtain $\frac{(x,u)}{\varpi (u,u)} \in O_F$, and hence the Gram--Schmidt process applies.  

For any $x \in L^\sharp$, consider the decomposition  
\begin{equation*}
    x = \frac{(x,u)}{(u,u)} u + \left(x - \frac{(x,u)}{(u,u)} u\right) 
    = \frac{(x, \varpi u)}{\varpi (u,u)} u + \left(x - \frac{(x, \varpi u)}{\varpi (u,u)} u\right).
\end{equation*}  
Since $\varpi u \in L$ and $x \in L^\sharp$, we have $(x, \varpi u) \in O_F$. As $\val(u,u) = -1$, we conclude that $\frac{(x, \varpi u)}{\varpi (u,u)} \in O_F$, so the Gram--Schmidt process applies.
\end{proof}

\begin{theorem}\label{prop:recursion}  
Fix $a,b\geq 0$ with $a+b=n$. Let $T \in \Herm_n^{\reg}(F)$ be a hermitian matrix such that the weighted local density function $\alpha_T(r, \phi^{[t]}_{a,b})$, where $\phi^{[t]}_{a,b}$ is defined in \eqref{equ:special phi ab}, is not identically zero.  

\begin{enumerate}  

\item Define  
\begin{equation*}
    T_\flat := T_2 - t_1^{-1} t_{21} t_{12}, \quad \text{where} \quad 
    T = \begin{pmatrix}
        t_1 & t_{12} \\ 
        t_{21} & T_2
    \end{pmatrix},
\end{equation*}  
where $T_2$ is an $(n-1) \times (n-1)$ matrix.  
Suppose $t \neq n$ and $\val(t_1) = 0$. Then  
\begin{equation*}
    \alpha_T(r, \phi^{[t]}_{a,b}) = \alpha_{t_1}(r, \charfun_{L^{[t]}}) \alpha_{T_\flat}(r, \phi^{[t]}_{a-1,b}),
\end{equation*}  
where $\phi^{[t]}_{a-1,b} \in \CS(V_\flat^{n-1})$ is the corresponding Schwartz function for a hermitian space $V_\flat$ of dimension $n-1$. The space $V_\flat$ is split (resp. non-split) if $V$ is split (resp. non-split).

\item Define  
\begin{equation*}
    T_\flat := T_1 - t_2^{-1} t_{12} t_{21}, \quad \text{where} \quad 
    T = \begin{pmatrix}
        T_1 & t_{12} \\ 
        t_{21} & t_2
    \end{pmatrix},
\end{equation*}  
where $T_1$ is an $(n-1) \times (n-1)$ matrix.  
Suppose $t \neq 0$ and $\val(t_2) = -1$. Then  
\begin{equation*}
    \alpha_T(r, \phi_{a,b}^{[t]}) = q^{-(n-1)}\alpha_{t_2}(r, \charfun_{L^{[t],\sharp}}) \alpha_{T_\flat}(r, \phi^{[t-1]}_{a,b-1}),
\end{equation*}  
where $\phi^{[t-1]}_{a,b-1} \in \CS(V_\flat^{n-1})$ is the corresponding Schwartz function for a hermitian space of dimension $n-1$. The space $V_\flat$ is split (resp. non-split) if $V$ is non-split (resp. split).  

\end{enumerate}  
\end{theorem}

\begin{proof}  
We first consider the case where $T \in \Herm_n^{\reg}(F)$ is of the form  
\begin{equation*}
    T = \begin{pmatrix} t_1 & 0 \\ 0 & T_2 \end{pmatrix},
\end{equation*}  
where $T_2 \in \Herm_{n-1}^{\reg}(F)$ is an $(n-1) \times (n-1)$ matrix.  
It is straightforward to verify that:  
\begin{align*}
    \alpha_T(r, \phi^{[t]}_{a,b}) &= \int_{V^{\angler}_T} \phi_{a,b}^{[t],\angler}  
    = \int_{V_T^{\angler}} \left(\charfun_{L^{[t],\angler}}\right)^{\otimes a} \otimes \left(\charfun_{L^{[t],\angler,\sharp}}\right)^{\otimes b} \, dx \\  
    &= \int_{\substack{x\in V^{\angler} \\ (x_1, x_1) = t_1}} \charfun_{L^{[t],\angler}} \int_{\substack{x_2\in (\langle x_1 \rangle^{\perp})^{n-1} \\ (x_2, x_2) = T_2}}  
    \left(\charfun_{(L^{[t],\angler})^\perp}\right)^{\otimes a-1} \otimes \left(\charfun_{(L^{[t],\angler})^{\perp,\sharp}}\right)^{\otimes b} \, dx_2 \, dx_1 \\  
    &= \alpha_{t_1}(r, \charfun_{L^{[t]}}) \alpha_{T_2}(r, \phi_{a-1,b}^{[t]}).
\end{align*}  
The last equality follows from Lemma \ref{lem:ind-lattice}. A similar computation applies to case (2).

Now we prove the general case. For case (1), by assumption, we can find infinitely many values of $r$ such that $\alpha_T(r, \varphi_{a,b}^{[t]}) \neq 0$. Fix one such $r$, then there exists a collection of vectors $\bv = (v_i) \in (V^{\angler})^{n}$ such that $(\bv, \bv) = T$.

Define  
\begin{equation*}
    A = \begin{pmatrix} 1 & -t_1^{-1}t_{12} \\ 0 & 1_{n-1} \end{pmatrix} \in \SL_n(F),
\end{equation*}  
so that  
\begin{equation*}
    \prescript{t}{}{\ov{A}} T A = \begin{pmatrix} t_1 & 0 \\ 0 & T_2 - t_1^{-1} t_{21} t_{12} \end{pmatrix} = \begin{pmatrix} t_1 & 0 \\ 0 & T_\flat \end{pmatrix}.
\end{equation*}  

We claim that $\varphi_{a,b}^{[t],\angler}(\bu A) = \varphi_{a,b}^{[t],\angler}(\bu)$ for any $\bu \in V^{\angler}_T$.  
Checking each factor, we reduce to showing that for $\Lambda = L^{[t]}$ or $L^{[t],\sharp}$ and $\bv=(v_i) \in \Lambda^n$, the maps  
\begin{equation*}
    \Lambda\lr \Lambda,\quad x \mapsto x - \frac{(v_1, v_i)}{(v_1, v_1)} v_1,\quad i\in\{1,\cdots,n\}
\end{equation*}  
are well-defined and are translations of $\Lambda$.
In fact, we only need to prove that $\frac{(v_1, v_i)}{(v_1, v_1)} v_1 \in \Lambda$ for $i\in\{1,\cdots,n\}$, which follows from the proof of Lemma \ref{lem:lattice-orth}. A similar argument holds for case (2).

Thus, replacing $\bu$ by $\bu A$ and noting that $\det(A)=1$, we obtain
\begin{equation*}
    \alpha_T(r, \varphi) = \int_{V_T^{\angler}} \varphi^{\angler}(\bu) \, d\bu = \int_{V_{\prescript{t}{}{\ov{A}}TA}^{\angler}} \varphi^{\angler}(u) \, d\bu.
\end{equation*}  
Theorem \ref{prop:recursion} now follows from the diagonal case.
\end{proof}

\subsection{Kudla--Rapoport conjecture}  
We now formulate the Kudla--Rapoport conjecture for mixed special cycles in $\CN_n^{[h]}$.  
For $a+b=n$, we write  
\begin{equation*}
    S^{[h]}_{a,b} :=  
    \begin{cases}  
        \mathrm{diag}(1^{(n-h)}, \varpi^{(h-b)}, (\varpi^{-1})^{b}) & \text{if } h \geq b, \\  
        \mathrm{diag}(1^{(n-h)}, (\varpi^{-1})^{h}) & \text{if } h < b.  
    \end{cases}
\end{equation*}

For any admissible Schwartz function $\phi^{[h]} \in \CS(V^n)$, it is known that there exists a polynomial $F_T(X, \phi^{[h]}) \in \BQ[X]$ such that for each integer $r \geq 0$,  
\begin{equation*}
    \alpha_T(r, \phi^{[h]}) = F_T((-q)^{-2r}, \phi^{[h]}).
\end{equation*}  

Let $T = (\bu, \bu)$ be the moment matrix of a basis $\bu$ of $\BV$. Since $T$ is not represented in $V$, we have $\alpha_T(0, \phi^{[h]}) = 0$.
In this case, we define the derivative of the weighted local density function as  
\begin{equation}\label{equ:derivative of local density}
    \alpha'_T(0, \phi^{[h]}) := -\frac{\rd}{\rd X} \Big|_{X=1} F_T(X, \phi^{[h]}).
\end{equation}

Recall that $\BV$ is the unique $F/F_0$-hermitian space of dimension $n$ that is not isomorphic to $V$.
\begin{proposition}\label{prop:unique-error-coeff}  
For any $n$ and $0 \leq h \leq n$, let $\phi^{[h]}_{n,0} \in \CS(V^n)$ be an admissible Schwartz function of type $h$. 
For $0 \leq t < h$, $t \equiv h-1 \bmod{2}$, let $\phi_{n,0}^{[t]} \in \CS(\BV^n)$ be the admissible functions matched with $\phi^{[h]}_{n,0}$, see \eqref{not:local match schwartz function}.
Then there exist unique constants $\beta_{n,0}^{[h]}(t)$ for $0 \leq t < h$ such that  
\begin{equation}\label{equ:linear system}
    \alpha_T'(0, \phi_{n,0}^{[h]}) - \sum_{\substack{0 \leq t < h \\ t \equiv h-1 \bmod{2}}} \beta_{n,0}^{[h]}(t) \alpha_T(0, \phi_{n,0}^{[t]}) = 0
\end{equation}  
for all $T = S_{n,0}^{[s]}$ such that $s<h$ and $s \equiv h-1 \bmod{2}$.  
\end{proposition}  

\begin{proof}  
By definition or Theorem \ref{prop:recursion}, we observe that for any $t_1 < t_2 < h$ with $t_i \equiv h-1 \bmod{2}$, we have  
\begin{equation*}
    \alpha_{S^{[t_1]}_{n,0}}(0, \phi^{[t_2]}) = 0.
\end{equation*}  
This implies that the equations \eqref{equ:linear system} form a linear system with an upper triangular coefficient matrix.

By direct computation or by applying Lemma \ref{lem:local density volumn}, we see that the diagonal entries of the coefficient matrix are nonzero. Since an upper triangular matrix with nonzero diagonal entries is invertible, the corresponding linear system has a unique solution.
\end{proof}

\begin{definition}\label{def:error-coeff}  
\begin{enumerate}
\item For $0\leq t<h$ with $t\equiv h-1\bmod{2}$, let $\beta_{n,0}^{[h]}(t)$ be the constant of Proposition \ref{prop:unique-error-coeff}. 

\item For any weight vector $\bw$ of rank $n$, let $a$ be the number of $\CZ$-entries in $\bw$ and $b$ the number of $\CY$-entries in $\bw$.  Define
$$
\beta^{[h]}_{\bw,t}:=\beta^{[h]}_{a,b}(t):=(-q^{h-t})^b\cdot \beta^{[h]}_{n,0}(t).
$$
\item For any weight vector $\bw$ and a matrix $T\in \Herm_n^{\reg}(F)$ which represents the hermitian form of $\BV$, let
$$
\PDen(T,\phi^{[h]}_{\bw}):=\frac{1}{\vol(K_n^{[h]},dg)}\Bigl(\alpha_T'(0,\phi^{[h]}_{\bw})-\sum_{\substack{0\leq t<h\\ t\equiv h-1\bmod{2}}}\beta^{[h]}_{\bw,t}\,\alpha_T(0,\phi^{[t]}_{\bw})\Bigr).
$$
\end{enumerate}
\end{definition}  
The constants are chosen to be compatible with the Weil representation\footnote{In the first version, this was deduced from the cancellation law. As pointed out by Qiao He and the referee, the original formulation contained gaps.}. 
We compare our formulation with \cite{Cho22} and \cite{Cho-He-Zhang}. First, we have the following lemma.

\begin{lemma}\label{lem:local density volumn}  
For any $a + b = n$, we have  
\begin{equation*}
    \vol(K^{[h]}_n, dg) = \alpha_{S_{a,b}^{[h]}}(0, \phi_{a,b}^{[h]}).
\end{equation*}
In particular, when $(a,b)=(n,0)$, the expression in Definition \ref{def:error-coeff}(3) is identical to the analytic side of \cite[Conjecture 1.3]{Cho-He-Zhang}.
\end{lemma}  

\begin{proof}  
We prove the case when $h \geq b$ and leave the case $h < b$ to the reader.  
Let $\bu = (u_1, \dots, u_n) \in L^a \oplus L^{\sharp,b}$ be a collection of vectors with moment matrix  
\begin{equation*}
    (\bu, \bu) = S^{[h]}_{a,b} =
    \begin{pmatrix}
        1_{n-h} & & \\  
        & \varpi 1_{h-b} & \\  
        & & \varpi^{-1} 1_b  
    \end{pmatrix}.
\end{equation*}  

By Lemma \ref{lem:lattice-orth}(2), we have $\varpi u_{a+1}, \dots, \varpi u_n \in L$, hence  
\begin{equation*}
    \bu' := (u_1, \dots, u_a, \varpi u_{a+1}, \dots, \varpi u_{a+b}) \in L^n.
\end{equation*}  
Furthermore, $\bu'$ forms a basis of $L$ as it spans a type $h$ sublattice in $L$.  
By Theorem \ref{thm:bertini}, we can rewrite the integral as  
\begin{equation*}
    \alpha_{S_{a,b}^{[h]}}(0, \phi_{a,b}^{[h]}) = \int_{V_{S_{a,b}^{[h]}}} \phi^{[h]}_{a,b}(u) \, du = \int_{\U(V)} \phi^{[h]}_{a,b}(g^{-1} \bu) \, dg.
\end{equation*}  

We claim that the support of the function $g \mapsto \phi^{[h]}_{a,b}(g^{-1} \bu)$ is exactly $\Stab_{\U(V)}(L^{[h]})$, which will complete the proof.  
Indeed, for any $g \in \Stab_{\U(V)}(L^{[h]})$, we clearly have $g^{-1} \bu \in L^a \oplus L^{\sharp,b}$. Conversely, suppose $g \in \U(V)$ satisfies  
\begin{equation*}
    g^{-1} \bu = (g^{-1} u_1, \dots, g^{-1} u_n) \in L^a \oplus L^{\sharp,b}.
\end{equation*}  
By Lemma \ref{lem:lattice-orth}(2), we see that
\begin{equation*}
    g^{-1} \bu' = (g^{-1} u_1, \dots, g^{-1} u_a, g^{-1} \varpi u_{a+1}, \dots, g^{-1} \varpi u_{a+b}) \in L^n
\end{equation*}  
forms a basis of $L$. Therefore,  we have
\begin{equation*}
    g^{-1} L = g^{-1} \Span_{O_F}(\bu') = \Span_{O_F}(g^{-1} \bu') = L,
\end{equation*}  
which completes the proof.
\end{proof}

Next, we compare our formula with \cite{Cho22}. We first recall the results of loc.cit. in the notation of this paper, see Remark \ref{rmk:compare-cho-notation}.

We will first work on a hermitian space of dimension $2n$. Let $0\leq h\leq 2n$. For any matrix 
$$
B=\begin{pmatrix}
E&F\\ G&H
\end{pmatrix}\in \Herm_{2n}(F),
$$
with $E$ of size $(2n-h)\times (2n-h)$, following \cite[Definition 3.3]{Cho22}, we define
$$
B^{\vee_h}:=\begin{pmatrix}
\varpi H & G\\
F& \varpi^{-1}E
\end{pmatrix}.
$$

\begin{theorem}[\protect{\cite[Theorem 3.16]{Cho22}}]\label{thm:main-thm-cho}
There are constants $\beta_i^h$ $(0\leq h\leq 2n, 0\leq i\leq n-1)$ such that for any $B\in \Herm_{2n}^{\reg}(F)$ that is represented by the hermitian space $\BV$, we have
\begin{multline*}
    \alpha'_B(0,\phi^{[n]}_{2n-h,h})-\sum_{\substack{0\leq i<n\\ i\equiv n-1\bmod{2}}}q^{2n(n-i)}\beta^h_i\,\alpha_B(0,\phi^{[i]}_{2n-h,h})\\
    =\alpha'_{B^{\vee_h}}(0,\phi^{[n]}_{h,2n-h})-\sum_{\substack{0\leq j<n\\ j\equiv n-1\bmod{2}}}q^{2n(n-j)}\beta^{2n-h}_j\,\alpha_{B^{\vee_h}}(0,\phi^{[j]}_{h,2n-h}).
\end{multline*}
\end{theorem}
\begin{proof}
This is \cite[Theorem 3.16]{Cho22}. In loc.cit., there is an additional term with $\delta_h$, which does not contribute to the equality by our assumption of the matrix $B$.
\end{proof}
We also have the following relation between these constants.
\begin{theorem}\label{thm:compare-coeff-cho}
For any $0\leq h<2n$ and $0\leq i\leq n-1$, we have $\beta^{h+1}_i=(-q)^{n-i}\cdot\beta^h_i$.
\end{theorem}
\begin{proof}
This follows directly from the formula in \cite[Proposition 2.7]{Cho23}, in which $h$ enters only through the factor $\alpha_{i+1,h}$.
See also \cite[Definition 7.5]{Cho-He-Zhang}.
\end{proof}

\begin{remark}\label{rmk:compare-cho-notation}
In \cite{Cho22} the lattices are named by their duals. 
The lattice $L_t$ of \cite[Definition 3.1]{Cho22} has rank $2n$ and moment matrix $A_t=\mathrm{diag}(1_{2n-t},\varpi^{-1}1_t)$; it is our $L^{[t],\sharp}$, and $L_t^\vee$ is our $L^{[t]}$. 
The function $1_{h,t}$ of loc.cit. is our $\phi^{[t]}_{2n-h,h}$. 
The Haar measures in \cite[Definition 3.1]{Cho22} give volume $1$ to integral matrices, and the self-dual measure gives $L^{[t],\sharp}$ the volume $q^{t}$. Hence the function $W_{h,t}$ of loc.cit. is $W_{h,t}(B,r)=q^{-2nt}\cdot\alpha_B(r,\phi^{[t]}_{2n-h,h})$, and $W_{n,n}(A_n,0)=q^{-2n^2}\vol(K_{2n}^{[n]},dg)$ by Lemma \ref{lem:local density volumn}.\
The derivative $W'_{h,t}(B,0)$ of \cite[Definition 3.10]{Cho22} is taken with respect to $X=(-q)^{-2r}$. So the analytic side 
$$
\frac{1}{W_{n,n}(A_n,0)}\bigl\{W'_{h,n}(B,0)-\sum_{0\leq i\leq n-1}\beta^h_i\,W_{h,i}(B,0)\bigr\}
$$ 
of \cite[Conjecture 3.25]{Cho22} is 
$$
\frac{1}{\vol(K_{2n}^{[n]},dg)}\bigl(\alpha'_B(0,\phi^{[n]}_{2n-h,h})-\sum_{0\leq i\leq n-1}q^{2n(n-i)}\beta^h_i\,\alpha_B(0,\phi^{[i]}_{2n-h,h})\bigr)
$$ 
in our notation.
\end{remark}

\begin{proposition}
Let $0\leq h\leq n$ and let $a+b=n$.
\begin{enumerate}
\item For any $0\leq s<h$ with $s\equiv h-1 \bmod 2$, we have
\begin{equation}\label{equ:volume-relation}
\beta^{[h]}_{a,b}(s)\cdot\frac{\vol(K_n^{[s]},dg)}{\vol(K_n^{[h]},dg)}=q^{2n(h-s)}\,\beta^{b+n-h}_{s+n-h}\cdot\frac{\vol(K_{2n}^{[s+n-h]},dg)}{\vol(K_{2n}^{[n]},dg)}.
\end{equation}
\item Let $T\in\Herm^{\reg}_n(F)$ represent the hermitian form of $\BV$, and let 
\begin{equation}\label{equ:definition for B}
B=\begin{pmatrix}
1_h&&\\
&T&\\
&&\varpi^{-1}1_{n-h}
\end{pmatrix}\in \Herm_{2n}(F).
\end{equation}
Then we have
\begin{equation}\label{equ:cho-cancellation}
\PDen(T,\phi_{a,b}^{[h]})=\frac{1}{\vol(K_{2n}^{[n]},dg)}\Bigl(\alpha'_B(0,\phi^{[n]}_{a+h,b+n-h})-\sum_{\substack{0\leq i<n\\ i\equiv n-1\bmod{2}}}q^{2n(n-i)}\beta^{b+n-h}_i\,\alpha_B(0,\phi^{[i]}_{a+h,b+n-h})\Bigr)
\end{equation}
\end{enumerate}
\end{proposition}

\begin{proof}
For any $0\leq i\leq n$, by Lemma \ref{lem:local density volumn} and Theorem \ref{prop:recursion}, we have
\begin{equation*}
\frac{\alpha_B(0,\phi^{[i]}_{a+h,b+n-h})}{\vol(K_{2n}^{[i]},dg)}=\left\{\begin{array}{cc}
\frac{\alpha_T(0,\phi^{[i-n+h]}_{a,b})}{\vol(K_n^{[i-n+h]},dg)}&\quad n-h\leq i<n;\\
0&0\leq i<n-h.
\end{array}\right.
\end{equation*}
Therefore, we have 
\begin{equation*}
\frac{\alpha'_B(0,\phi^{[n]}_{a+h,b+n-h})}{\vol(K_{2n}^{[n]},dg)}=\frac{\alpha'_T(0,\phi^{[h]}_{a,b})}{\vol(K_n^{[h]},dg)}.
\end{equation*}
Write 
$$
T=\begin{pmatrix}
T_{11}&T_{12}\\
T_{21}&T_{22}
\end{pmatrix}
\quad\text{and}\quad
T^{\vee_b}:=
\begin{pmatrix}
\varpi T_{22}&T_{21}\\
T_{12}&\varpi^{-1}T_{11}
\end{pmatrix},
$$
with $T_{22}$ of size $b\times b$. Then we have
\begin{equation*}
    B^{\vee_{b+n-h}}=\begin{pmatrix}
    \varpi T_{22}&&&T_{21}\\
    &1_{n-h}&&\\
    &&\varpi^{-1}1_h&\\
    T_{12}&&&\varpi^{-1}T_{11}
    \end{pmatrix}\in \Herm_{2n}(F).
\end{equation*}
Then for any $0\leq j\leq n$, we have
\begin{equation*}
\frac{\alpha_{B^{\vee_{b+n-h}}}(0,\phi^{[j]}_{b+n-h,a+h})}{\vol(K_{2n}^{[j]},dg)}=\left\{\begin{array}{cc}
\frac{\alpha_{T^{\vee_b}}(0,\phi^{[j-h]}_{b,a})}{\vol(K_n^{[j-h]},dg)}&\quad h\leq j<n;\\
0&0\leq j<h,
\end{array}\right.
\end{equation*}
and
\begin{equation*}
\frac{\alpha'_{B^{\vee_{b+n-h}}}(0,\phi^{[n]}_{b+n-h,a+h})}{\vol(K_{2n}^{[n]},dg)}=\frac{\alpha'_{T^{\vee_b}}(0,\phi^{[n-h]}_{b,a})}{\vol(K_n^{[n-h]},dg)}.
\end{equation*}
Applying these formulas to Theorem \ref{thm:main-thm-cho}, we obtain
\begin{multline}\label{equ:cho-induction-n}
\frac{\alpha'_T(0,\phi^{[h]}_{a,b})}{\vol(K_n^{[h]},dg)}-\sum_{\substack{0\leq s<h\\ s\equiv h-1\bmod{2}}}q^{2n(h-s)}\,\beta^{b+n-h}_{s+n-h}\cdot\frac{\vol(K_{2n}^{[s+n-h]},dg)}{\vol(K_{2n}^{[n]},dg)}\cdot\frac{\alpha_T(0,\phi^{[s]}_{a,b})}{\vol(K_n^{[s]},dg)}\\
=\frac{\alpha'_{T^{\vee_b}}(0,\phi^{[n-h]}_{b,a})}{\vol(K_n^{[n-h]},dg)}-\sum_{\substack{0\leq s<n-h\\ s\equiv n-h-1\bmod{2}}}q^{2n(n-h-s)}\,\beta^{a+h}_{s+h}\cdot\frac{\vol(K_{2n}^{[s+h]},dg)}{\vol(K_{2n}^{[n]},dg)}\cdot\frac{\alpha_{T^{\vee_b}}(0,\phi^{[s]}_{b,a})}{\vol(K_n^{[s]},dg)}.
\end{multline}

We apply \eqref{equ:cho-induction-n} to the case when $(a,b)=(n,0)$ and $T=S_{n,0}^{[t]}$ for $0\leq t<h$, $t\equiv h-1\bmod{2}$. Then $T^{\vee_0}=\varpi^{-1}T$ and the right hand side of \eqref{equ:cho-induction-n} is zero.
Hence the left hand side of \eqref{equ:cho-induction-n} is zero for all such $T$. By the uniqueness of the solution in Proposition \ref{prop:unique-error-coeff}, we have
\begin{equation*}
\beta_{n,0}^{[h]}(s)=q^{2n(h-s)}\,\beta^{n-h}_{s+n-h}\cdot\frac{\vol(K_{2n}^{[s+n-h]},dg)}{\vol(K_{2n}^{[n]},dg)}\cdot\frac{\vol(K_n^{[h]},dg)}{\vol(K_n^{[s]},dg)}.
\end{equation*}
This proves \eqref{equ:volume-relation} when $b=0$.

For general $b$, the identity \eqref{equ:volume-relation} follows from Definition \ref{def:error-coeff}(2) and Theorem \ref{thm:compare-coeff-cho}. Hence part (1) holds for all $b$.
Part (2) then follows from part (1) and Definition \ref{def:error-coeff}(3).
\end{proof}

\begin{corollary}\label{cor:hermitian-duality}
Let $0\leq h\leq n$ and let $a+b=n$. Let
$$
T=\begin{pmatrix}
T_{11}&T_{12}\\
T_{21}&T_{22}
\end{pmatrix}\in\Herm_n^{\reg}(F)
$$
be represented by the hermitian form of $\BV$, with $T_{22}$ of size $b\times b$. Let
$$
T^{\vee_b}:=
\begin{pmatrix}
\varpi T_{22}&T_{21}\\
T_{12}&\varpi^{-1}T_{11}
\end{pmatrix}.
$$
Then we have 
\begin{equation}\label{equ:analytic-side-dual}
\PDen(T,\phi_{a,b}^{[h]})=\PDen(T^{\vee_b},\phi_{b,a}^{[n-h]}).
\end{equation}
\end{corollary}
\begin{proof}
This follows by applying \eqref{equ:volume-relation} and \eqref{equ:cho-cancellation} to \eqref{equ:cho-induction-n}.
\end{proof}

We are ready to prove the cancellation law on the analytic side.
\begin{theorem}\label{thm:ana-cancel}  
Let $T \in \Herm_n^{\reg}(F)$ be a hermitian matrix that represents the hermitian form of $\BV$.  

\begin{enumerate}  
\item Consider the following partition, where $T_2$ is an $(n-1)\times (n-1)$ matrix:
\begin{equation*}
    T = \begin{pmatrix} t_1 & t_{12} \\ t_{21} & T_2 \end{pmatrix}, \quad  
    T_\flat = T_2 - t_1^{-1} t_{21} t_{12}.
\end{equation*}  

\begin{enumerate}  
\item If $h \neq n$ and $\val(t_1) = 0$, then  
\begin{equation*}
    \PDen(T, \phi^{[h]}_{a,b}) = \PDen(T_{\flat}, \phi^{[h]}_{a-1,b}).
\end{equation*}  

\item If $h = n$ and $\val(t_1) = 1$, then  
\begin{equation*}
    \PDen(T, \phi^{[n]}_{a,b}) = \PDen(T_{\flat}, \phi^{[n-1]}_{a-1,b}).
\end{equation*}  
\end{enumerate}  

\item Consider the following partition, where $T_1$ is an $(n-1)\times (n-1)$ matrix:
\begin{equation*}
    T = \begin{pmatrix} T_1 & t_{12} \\ t_{21} & t_2 \end{pmatrix}, \quad  
    T_\flat = T_1 - t_2^{-1} t_{12} t_{21}.
\end{equation*}  

\begin{enumerate}  
\item If $h \neq 0$ and $\val(t_2) = -1$, then  
\begin{equation*}
    \PDen(T, \phi^{[h]}_{a,b}) = \PDen(T_{\flat}, \phi^{[h-1]}_{a,b-1}).
\end{equation*}  

\item If $h = 0$ and $\val(t_2) = 0$, then  
\begin{equation*}
    \PDen(T, \phi^{[0]}_{a,b}) = \PDen(T_{\flat}, \phi^{[0]}_{a,b-1}).
\end{equation*}  
\end{enumerate}  
\end{enumerate}  
\end{theorem}

\begin{proof}  
By Theorem \ref{prop:recursion}, for case (1)(i), we have for any $0\leq t\leq n$,
\begin{equation*}
    \alpha_T(r, \phi^{[t]}_{a,b}) = \alpha_{t_1}(r, \charfun_{L^{[t]}}) \alpha_{T_\flat}(r, \phi^{[t]}_{a-1,b}),\quad\text{and}\quad\alpha'_T(0, \phi^{[t]}_{a,b}) = \alpha_{t_1}(0, \charfun_{L^{[t]}}) \alpha'_{T_\flat}(0, \phi^{[t]}_{a-1,b}).
\end{equation*}  
Therefore, we obtain  
\begin{align}
    \PDen(T, \phi_{a,b}^{[h]}) &= \frac{1}{\alpha_{S^{[h]}_{a,b}}(0, \phi_{a,b}^{[h]})} \Biggl(
    \alpha_T'(0, \phi_{a,b}^{[h]}) - \sum_{\substack{0 \leq t < h \\ t \equiv h-1\bmod{2}}} \beta_{a,b}^{[h]}(t) \alpha_T(0, \phi_{a,b}^{[t]})
    \Biggr) \notag \\  
    &= \frac{\alpha_{t_1}(0, \charfun_{L^{[h]}})}{\alpha_{1}(0, \charfun_{L^{[h]}}) \alpha_{S^{[h]}_{a-1,b}}(0, \phi_{a-1,b}^{[h]})} \Biggl(
    \alpha_{T_\flat}'(0, \phi_{a-1,b}^{[h]}) - \sum_{\substack{0 \leq t < h \\ t \equiv h-1\bmod{2}}} \frac{\alpha_{t_1}(0,\charfun_{L^{[t]}})}{\alpha_{t_1}(0,\charfun_{L^{[h]}})}\beta_{a,b}^{[h]}(t) \alpha_{T_\flat}(0, \phi_{a-1,b}^{[t]})
    \Biggr) \notag \\  
    &= \frac{1}{\alpha_{S^{[h]}_{a-1,b}}(0, \phi_{a-1,b}^{[h]})} \Biggl(
    \alpha_{T_\flat}'(0, \phi_{a-1,b}^{[h]}) - \sum_{\substack{0 \leq t < h \\ t \equiv h-1\bmod{2}}} \frac{\alpha_{t_1}(0,\charfun_{L^{[t]}})}{\alpha_{t_1}(0,\charfun_{L^{[h]}})}\beta_{a,b}^{[h]}(t) \alpha_{T_\flat}(0, \phi_{a-1,b}^{[t]})
    \Biggr).\label{equ:pre-cancellation}
\end{align}  
We claim that 
\begin{equation*}
\frac{\beta^{[h]}_{a-1,b}(t)}{\beta^{[h]}_{a,b}(t)}=\frac{\beta^{[h]}_{n-1,0}(t)}{\beta^{[h]}_{n,0}(t)}= \frac{\alpha_{t_1}(0,\charfun_{L^{[t]}})}{\alpha_{t_1}(0,\charfun_{L^{[h]}})}.
\end{equation*}
The first equality follows directly from Definition \ref{def:error-coeff}(2). For the second equality, take $T=S_{n,0}^{[t']}=\mathrm{diag}(1,S_{n-1,0}^{[t']})$ with $0\leq t'<h$ and $t'\equiv h-1\bmod{2}$. Then the left hand side of \eqref{equ:pre-cancellation} is zero by the definition in Proposition \ref{prop:unique-error-coeff}. 
By Theorem \ref{prop:recursion}(1), this gives a solution of the linear system of Proposition \ref{prop:unique-error-coeff} in dimension $n-1$, and the claim follows by uniqueness.
With the claim, case (1)(i) follows directly from \eqref{equ:pre-cancellation}.
Cases (1)(ii) and (2)(i) follow from the hermitian duality (Corollary \ref{cor:hermitian-duality}).
\end{proof}

We reformulate Theorem \ref{thm:ana-cancel} to align with the geometric cancellation law (Proposition \ref{prop:geo-cancel}).

\begin{corollary}\label{cor:ana-cancel}
Let $\phi^{[h]}=\phi^{[h]}_\bw\in\CS(V^n)$ be an admissible Schwartz function and let $\bu=(u_1,\cdots,u_n)\in \BV^n$ be vectors that are linearly independent.
For any fixed $1 \leq i \leq n$, define:
\begin{itemize}
    \item $\BV_{\flat} := \langle u_i \rangle^\perp$ and set $u_{j,\flat}$ being the orthogonal projection of $u_j$ for $j\neq i$.
    \item $\bu_{\flat} := (u_{1,\flat}, \dots, \wh{u_{i,\flat}}, \dots, u_{n,\flat})\in \BV_\flat^{n-1}$.
    \item $\bw_\flat := (\bw_1, \dots, \wh{\bw_i}, \dots, \bw_n)\in\{\CZ,\CY\}^{n-1}$.
\end{itemize}
Then we have the following:
\begin{enumerate}
\item If $\bw_i=\CZ$ and 
$$
\val(u_i,u_i)=\left\{\begin{array}{ll}
0&\text{if }h\neq n;\\
1&\text{if }h=n,
\end{array}\right.
$$
then we have $\PDen(\bu,\phi_{\bw}^{[h]})=\PDen(\bu_{\flat},\phi_{\bw_\flat}^{[h]})$.
\item If $\bw_i=\CY$ and 
$$
\val(u_i,u_i)=\left\{\begin{array}{ll}
-1&\text{if }h\neq 0;\\
0&\text{if }h=0,
\end{array}\right.
$$
then we have $\PDen(\bu,\phi_{\bw}^{[h]})=\PDen(\bu_{\flat},\phi_{\bw_\flat}^{[h-1]})$.\qed
\end{enumerate}
\end{corollary}

We can now state the main theorem in the $p$-adic setting.  

\begin{theorem}[Kudla--Rapoport conjecture for unramified maximal parahorics]\label{thm:KR-conj}  
Let $\bu \in \BV^n$ be a collection of vectors that are linearly independent, and let $T = (\bu, \bu) \in \Herm^{\reg}_n(F)$ be its moment matrix. Let $\phi^{[h]} = \phi^{[h]}_{\bw} \in \CS(V^n)$ be an admissible Schwartz function, and let $\phi_{\bw}^{[t]} \in \CS(\BV^n)$ be the corresponding matched admissible Schwartz functions of type $t$ for $t < h$ and $t \equiv h-1 \bmod{2}$, see \eqref{not:local match schwartz function}. 
Then, we have:  
\begin{equation*}
    \Int(\bu, \phi_{\bw}^{[h]}) = \frac{1}{\vol(K_n^{[h]}, dg)} \Biggl(
    \alpha_T'(0, \phi_{\bw}^{[h]}) - \sum_{\substack{0 \leq t < h \\ t \equiv h-1\bmod{2}}} \beta_{\bw,t}^{[h]} \alpha_T(0, \phi_{\bw}^{[t]})
    \Biggr).
\end{equation*}  
Here, $K_n^{[h]} = \Stab_{\U(V)}(L^{[h]}) \subset \U(V)$ is the parahoric subgroup that defines the level structure, and the measure $dg = dg^{\mathrm{self}}$ is defined at the beginning of \S \ref{sec:def weighted local density}.  
We denote the right-hand side by $\PDen(T, \phi_{\bw}^{[h]})$ or $\PDen(\bu, \phi_{\bw}^{[h]})$.  
\end{theorem}  

We will prove Theorem \ref{thm:KR-conj} in \S \ref{sec:proof}.

Combining Proposition \ref{prop:geo-cancel} and Corollary \ref{cor:ana-cancel}, we obtain the following cancellation law.

\begin{proposition}\label{lem:cancellation law}
Let $\phi^{[h]} = \phi^{[h]}_\bw \in \CS(V^n)$ be an admissible Schwartz function, and let $\bu = (u_1, \dots, u_n) \in \BV^n$ be vectors that are linearly independent.
For any fixed $1 \leq i \leq n$, define:
\begin{itemize}
    \item $\BV_{\flat} := \langle u_i \rangle^\perp$ and set $u_{j,\flat}$ being the orthogonal projection of $u_j$ for $j\neq i$.
    \item $\bu_{\flat} := (u_{1,\flat}, \dots, \wh{u_{i,\flat}}, \dots, u_{n,\flat})\in \BV_\flat^{n-1}$.
    \item $\bw_\flat := (\bw_1, \dots, \wh{\bw_i}, \dots, \bw_n)\in\{\CZ,\CY\}^{n-1}$.
\end{itemize}
Then we have the following:
\begin{enumerate}
\item If $\bw_i=\CZ$ and 
$$
\val(u_i,u_i)=\left\{\begin{array}{ll}
0&\text{if }h\neq n;\\
1&\text{if }h=n,
\end{array}\right.
$$
Then the identity $\Int(\bu,\phi_{\bw}^{[h]})=\PDen(\bu,\phi_{\bw}^{[h]})$ follows from the identity
$\Int(\bu_{\flat},\phi_{\bw_\flat}^{[h]})=\PDen(\bu_{\flat},\phi_{\bw_\flat}^{[h]})$.
\item If $\bw_i=\CY$ and 
$$
\val(u_i,u_i)=\left\{\begin{array}{ll}
-1&\text{if }h\neq 0;\\
0&\text{if }h=0.
\end{array}\right.
$$
Then the identity $\Int(\bu,\phi_{\bw}^{[h]})=\PDen(\bu,\phi_{\bw}^{[h]})$ follows from the identity
$\Int(\bu_{\flat},\phi_{\bw_\flat}^{[h-1]})=\PDen(\bu_{\flat},\phi_{\bw_\flat}^{[h-1]})$.\qed
\end{enumerate}
\end{proposition}

We conclude the subsection by generalizing Proposition \ref{prop:unique-error-coeff}.
\begin{proposition}\label{prop:unique-error-coeff-any}  
For any weight vector $\bw$ of rank $n$, let $s \in S_n$ be a fixed permutation such that $s \cdot \bw = \bw_{a,b}:= (\CZ^a, \CY^b)$ for some $a + b = n$.  
Assume $b<h$.
Let $\beta^{[h]}_{\bw,t}$ be as in Definition \ref{def:error-coeff}, and set  
$S^{[t]}_\bw := s \cdot S^{[t]}_{a,b}s^{-1}$, where we view $s\in S_n=N(T)/T$ as an element in the Weyl group\footnote{$S^{[t]}_\bw$ depends on the choice of $s \in S_n$.}. Then, we have  
\begin{equation*}
    \alpha_T'(0, \phi_{\bw}^{[h]}) - \sum_{\substack{0 \leq t < h \\ t \equiv h-1\bmod{2}}} \beta_{\bw,t}^{[h]}\alpha_T(0, \phi_{\bw}^{[t]}) = 0
\end{equation*}  
for all $T = S_{\bw}^{[t]}$ such that $t < h$ and $t \equiv h-1 \bmod{2}$.
\end{proposition}
\begin{proof}
By Corollary \ref{cor:hermitian-duality}, we reduce to show that $\PDen(T^{\vee_b},\phi^{[n-h]}_{b,a})=0$ for all $T=S_{\bw}^{[t]}$. This follows directly from Lemma \ref{lem:lattice-orth}(2).
\end{proof}

\subsection{Base case}\label{sec:base case}
We verify Theorem \ref{thm:KR-conj} for the case $n = 1$ as the starting point for the induction in \S \ref{sec:proof}. To apply the result from \cite{Hironaka98}, we recall and use the representation density function from \cite{KR14}.  

For hermitian matrices $S$ and $T$ of sizes $m$ and $n$, respectively, the \emph{representation density function} is defined as
\begin{equation*}
    \alpha(S, T) = \lim_{k \rightarrow \infty} (q^{-k})^{n(2m-n)} \# A_{\varpi^k}(S, T),
\end{equation*}  
where $q$ is the cardinality of the residue field of $O_{F_0}$, and the term in the limit is defined as
\begin{equation*}
    A_{\varpi^k}(S, T) :=  
    \left\{
    x \in M_{m,n}(O_F / \varpi^k O_F) \mid S[x] \equiv T \bmod{\varpi^k \Herm_n(O_F)^\vee}
    \right\}.
\end{equation*}  

By \cite[Proposition A.6]{Kudla97}, the representation density function is related to the weighted local density function of Definition \ref{def:weighted local density function} by the relation
\begin{equation}\label{equ:compare two local density}
    \alpha_T(r, \charfun_{L^n}) = |\mathrm{Nm}_{F/F_0} \det S_0|^{n/2} \alpha(S_r, T).
\end{equation}
Here $S_0=S$ is the moment matrix of a basis of $L$, and $S_r=\mathrm{diag}(S,S')$ with $S'$ the moment matrix of a basis of $L_{r,r}$.

We define the derivative $\alpha'(S, T)$ of the representation density function analogously to \eqref{equ:derivative of local density}, see \cite[\S 9]{KR11}.
By direct computation using the formula in \cite{Hironaka98}, we obtain the following result:  
\begin{proposition}\label{prop:hironaka-computation}
\begin{enumerate}
\item For any integer $\lambda\geq 0$, we have 
$$
\alpha\bigl((1)_{1\times 1},(\varpi^\lambda)_{1\times 1}\bigr)=\left\{\begin{array}{ll}
1-(-q)^{-1}&\lambda\text{ even};\\
0&\lambda\text{ odd}.
\end{array}\right.
$$
Moreover, for any odd $\lambda\geq 0$, we have
$
\alpha'\bigl((1)_{1\times 1},(\varpi^{\lambda})_{1\times 1}\bigr)=(1-(-q)^{-1})\frac{\lambda+1}{2}.
$
\item For any integer $\lambda\geq 0$, we have
$$
\alpha\bigl((\varpi)_{1\times 1},(\varpi^\lambda)_{1\times 1}\bigr)=\left\{\begin{array}{ll}
0&\lambda\text{ even};\\
1+q&\lambda\text{ odd}.
\end{array}\right.
$$
Moreover, for any even $\lambda\geq 0$, we have $\alpha'\bigl((\varpi)_{1\times 1},(\varpi^{\lambda})_{1\times 1}\bigr)=(1+q)\frac{\lambda}{2}+1$.
\item For any integer $\lambda \geq -1$, we have
$$
\alpha\bigl((\varpi^{-1})_{1\times 1},(\varpi^\lambda)_{1\times 1}\bigr)=\left\{\begin{array}{ll}
q^{-1}\bigl(
    (1+q^{-1})\frac{\lambda}{2}+q^{-1}
\bigr)&\lambda\text{ even};\\
q^{-1}(1+q^{-1})&\lambda\text{ odd}.
\end{array}\right.
$$
Moreover, for any even $\lambda\geq 0$, we have $\alpha'\bigl((\varpi^{-1})_{1\times 1},(\varpi^{\lambda})_{1\times 1}\bigr)=(1+q^{-1})\frac{\lambda}{2}+q^{-1}.$
\end{enumerate}
\end{proposition}
\begin{proof}
We refer the reader to \cite[Theorem 3.2]{Sankaran17} for the relevant notation and statements. A straightforward computation shows that
\begin{equation*}
\prod_j I_j((\alpha),(\lambda))=\left\{\begin{array}{ll}
1&\alpha=0;\\
(-q)^{\alpha-1}(1+(-q))&1\leq \alpha\leq \lambda;\\
(-q)^\lambda&\alpha=\lambda+1.
\end{array}\right.
\end{equation*}
By \cite[Theorem 3.2]{Sankaran17}, for $\lambda=(\lambda)\in\BZ^1_{\succ 0}$ and $\xi=(\xi_i)\in\BZ^{1+k}_{\succ 0}$, we have
\begin{multline*}
\alpha\bigl((\mathrm{diag}(\varpi^{\xi_1},\cdots,\varpi^{\xi_{1+k}})),(\varpi^{\lambda})\bigr)\\
=1+(1+(-q)^{-1})\cdot\sum_{1\leq \alpha\leq \lambda}\Bigl((-1)^{\alpha}(-q)^{-k\alpha+\langle \xi', (1^\alpha)\rangle}\Bigr)+(-1)^{\lambda+1}(-q)^{-1}(-q)^{-k(\lambda+1)+\langle \xi', (1^{\lambda+1})\rangle},
\end{multline*}
where $\xi'$ denotes the conjugate of $\xi$ as Young tableau, and $\langle \xi',\mu'\rangle:=\sum_{i\geq 1}\xi'_i\mu'_i$.

In case (1), we choose $\xi=(0^{1+k})$ and we have
$$
\alpha\Bigl(
1_{1+k},(\varpi^{\lambda})_{1\times 1}
\Bigr)=1+(1+(-q)^{-1})\sum_{1\leq \alpha\leq \lambda}((-1)^\alpha (-q)^{-k\alpha})+(-1)^{\lambda+1}(-q)^{-(k+1)(\lambda+1)+\lambda}.
$$

In case (2), we choose $\xi=(1,0^k)$, and we have
$$
\alpha\Biggl(
\begin{pmatrix}
\varpi &\\&1_k
\end{pmatrix}
,(\varpi^\lambda)_{1\times 1}\Biggr)=1+(1+(-q))\sum_{1\leq \alpha\leq \lambda}\Bigl((-1)^\alpha(-q)^{-k\alpha}\Bigr)+(-1)^{\lambda+1}(-q)^{-k(\lambda+1)}.
$$

In case (3), we choose $\xi=(1^k,0)$, and we have
$$
\alpha\Biggl(
\begin{pmatrix}
\varpi 1_k &\\&1
\end{pmatrix}
,(\varpi^\lambda)_{1\times 1}\Biggr)=1+(1+(-q)^{-1})\sum_{1\leq \alpha\leq \lambda}\Bigl( (-1)^\alpha(-q)^{-k(\alpha-1)}\Bigr)+(-1)^{\lambda+1}(-q)^{-1}(-q)^{-k\lambda},
$$
and we have 
$$
\alpha\Biggl(
\begin{pmatrix}
\varpi^{-1} &\\&1_k
\end{pmatrix}
,(\varpi^\lambda)_{1\times 1}\Biggr)=q^{-1}\cdot\alpha\Biggl(
\begin{pmatrix}
\varpi 1_k &\\&1
\end{pmatrix}
,(\varpi^{\lambda+1})_{1\times 1}\Biggr).
$$
All assertions now follow from direct computation.
\end{proof}

\begin{corollary}\label{cor:KR for n=1}  
Theorem \ref{thm:KR-conj} holds for $n = 1$.  
\end{corollary}  

\begin{proof}  
This follows from Example \ref{exmp:base case geo} and Proposition \ref{prop:hironaka-computation}. When $h=0$, this is Proposition \ref{prop:hironaka-computation}(1) and Example \ref{exmp:base case geo}(1).

By Proposition \ref{prop:hironaka-computation}(1)(2), the constant of Proposition \ref{prop:unique-error-coeff} is 
\begin{equation*}
    \beta^{[1]}_{1,0}(0)=\frac{q^{-1}\cdot\alpha'((\varpi)_{1\times 1}, (1)_{1\times 1})}{\alpha((1)_{1\times 1}, (1)_{1\times 1})}=\frac{q^{-1}}{1+q^{-1}}=\frac{1}{q+1}.
\end{equation*}
By Definition \ref{def:error-coeff}, we have $\beta^{[1]}_{0,1}(0)=-q\cdot\beta^{[1]}_{1,0}(0)=-\frac{q}{q+1}$. Hence
\begin{equation*}
    \PDen(\varpi^\lambda, \charfun_{L^{[1]}})=\frac{1}{1+q}\Bigl(
    \alpha'\bigl((\varpi)_{1\times 1},(\varpi^{\lambda})\bigr)-1
    \Bigr) = \frac{\lambda}{2},
\end{equation*}  
for $\lambda\geq 0$ even. This matches the geometric side as described in Example \ref{exmp:base case geo}.

By \eqref{equ:compare two local density}, we have
$$
\alpha_T(r,\charfun_{L^{[1],\sharp}})=q\cdot\alpha((\varpi^{-1})\obot 1_{2r},T).
$$
By Proposition \ref{prop:hironaka-computation}(3), we have
\begin{equation*}
    \alpha_T'(0,\charfun_{L^{[1],\sharp}})=(1+q^{-1})\frac{\lambda}{2}+q^{-1},\qquad \alpha_T(0,\charfun_{L^{[0]}})=1+q^{-1}=\vol(K_1^{[1]},dg),
\end{equation*}
and therefore
\begin{equation*}
    \PDen\bigl((\varpi^\lambda)_{1\times 1}, \charfun_{L^{[1],\sharp}}\bigr)=\frac{\lambda}{2}+\frac{1}{q+1}+\frac{q}{q+1}=\frac{\lambda}{2}+1.
\end{equation*}
This also matches the geometric side as described in Example \ref{exmp:base case geo}.
\end{proof}

\subsection{Local Whittaker functions}\label{sec:local weil}
In this subsection, we recall the local Whittaker function and local Siegel--Weil formula.
We change our local setup and let $F/F_0$ be a quadratic extension of local fields. 
If $F_0$ is archimedean, we take $F/F_0=\BC/\BR$.
If $F_0$ is non-archimedean, we assume $F/F_0$ is unramified.

Let $W$ be a $2n$-dimensional split hermitian space over $F$, with unitary group $G_n=\U(W)$. 
Let $V$ be a hermitian space of dimension $n$, 
which we assume to be positive definite when $F/F_0=\BC/\BR$.
Let $\chi_V$ be the quadratic character associated to $V$ and let $\psi$ be a fixed additive character. Then we have a Weil representation $\omega$ of $G_n\times \U(V)$ on $\CS(V^n)$, such that for each $\phi\in\CS(V^n)$ and $\bu\in V^n$, we have:
\begin{align*}
\omega(m(\ba))\phi(\bu)&=\chi(\det(\ba))|\det \ba|_F^{n/2}\phi(\bu\cdot \ba),&m(\ba)=
\begin{psmallmatrix}
\ba&\\
&\prescript{t}{}{\ov{\ba}^{-1}}
\end{psmallmatrix}\in M_n(F_0),\\
\omega(n(\bb))\phi(\bu)&=\psi(\tr \bb(\bu,\bu))\phi(\bu),&n(\bb)=\begin{psmallmatrix}
1_n& \bb \\ & 1_n
\end{psmallmatrix}\in N_n(F_0),\\
\omega(w_n)\phi(\bu)&=\gamma_{V}^n\cdot\widehat \phi(\bu),&w_n=\begin{psmallmatrix}
 & 1_n\\
  -1_n & \\
\end{psmallmatrix},\\
\omega(h)\phi(\bu)&=\phi(h^{-1}\cdot\bu),& h\in \U(V).
\end{align*}
Here $\gamma_{V}$ is the local Weil constant (see \cite[(10.3)]{KR14}), and $\widehat\phi$ is the Fourier transform of $\phi$ using the self-dual Haar measure on $V^n$ with respect to $\psi\circ\tr_{F/F_0}$.

There is an intertwining map  
\begin{equation*}
    \lambda: \CS(V^n) \to I_n(0, \chi_V), \quad  
    \lambda(\phi)(g) = \omega(g) \phi(0),
\end{equation*}  
where $I_n(s, \chi_V)$ denotes the principal series representation associated with $G_n$.  
We extend $\lambda(\phi)$ to a Siegel--Weil section $\lambda(\phi)(s) \in I_n(s, \chi_V)$.  

For any $T \in \Herm_n(F)$, we define the \emph{local Whittaker function} for $G_n$ as  
\begin{equation*}
    W_T(g,s,\phi) = \int_{\Herm_n(F)} \lambda(\omega(g) \phi)(w_n^{-1} n(b), s) \psi(-\tr(Tb)) \, db.
\end{equation*}  

Fix basis elements $\alpha \in \wedge^{2n^2}(V^n)^*$ and $\beta \in \wedge^{n^2}(\Herm_n)^*$. By \eqref{not:gauge form}, we have a form $\omega$ of degree $n^2$ on $V^n_\reg$ along with two measures and constants:  
\begin{equation*}
    d_{\alpha}x = c(\alpha, \psi) dx, \quad  
    d_{\beta}b = c(\beta, \psi) db.
\end{equation*}  

Let $T \in \Herm_n(F)$. For any Schwartz function $\phi \in \CS(V^n)$, we define the \emph{local orbital integral}  
\begin{equation*}
    \Orb_T(\phi) := \int_{V_T} \phi(x) \, d_\omega x,
\end{equation*}  
where $d_\omega x$ is the measure on $V_T$ determined by the restriction of $\omega$ to $V_T$.  

\subsection{Non-archimedean local Siegel--Weil formula}

We recall the non-archimedean local Siegel--Weil formula, following the formalism of \cite[p. 39]{KR14}. While \textit{loc. cit.} focuses on the case $F_0 = \BQ_p$, all results extend directly to general $p$-adic fields $F_0$.  

\begin{theorem}[Non-archimedean local Siegel--Weil formula]\label{thm:nonarchi local SW}  
Suppose $F/F_0$ is a quadratic extension of $p$-adic fields.  
Let $T \in \Herm_n^{\reg}(F)$ be a hermitian matrix, and let $\phi \in \CS(V^n)$ be a Schwartz function. Then, we have the equality  
\begin{equation*}
    W_T(g, 0, \phi) = C(V, \alpha, \beta, \psi) \cdot \Orb_T(\omega(g) \phi),
\end{equation*}  
where $C(V, \alpha, \beta, \psi)$ is a constant independent of $T$, given by  
\begin{equation*}
    C(V, \alpha, \beta, \psi) = \frac{\gamma_V^n c(\beta, \psi)}{c(\alpha, \psi)}.
\end{equation*}  
\end{theorem}  

\begin{proof}  
See \cite[p. 39]{KR14}.  
\end{proof}

We also relate the local Whittaker function to weighted local density functions.  

\begin{proposition}[\protect{\cite[Proposition 10.1]{KR14}}]\label{prop:witt-den}  
Suppose $F/F_0$ is a quadratic extension of $p$-adic fields.  
Let $T \in \Herm_n^{\reg}(F)$ be a nonsingular hermitian matrix, and let $\phi \in \CS(V^n)$ be any Schwartz function. For any integer $r \geq 0$, we have  
\begin{equation*}
    W_T(1, r, \phi) = \gamma_V^n \alpha_T(r, \phi).
\end{equation*}  
In particular,  we have
\begin{equation*}
    W_T(1, 0, \phi) = \gamma_V^n \alpha_T(0, \phi) \quad \text{and} \quad  
    W'_T(1, 0, \phi) = \gamma_V^n \alpha'_T(0, \phi) \log q^2.
\end{equation*}  
Here, $q$ denotes the cardinality of the residue field of $O_{F_0}$. \qed  
\end{proposition}

\section{Symmetric spaces and special cycles}\label{sec:symmetric space}
In this section, we review the archimedean theory of hermitian symmetric spaces and their special cycles.  
We then introduce Kudla's Green currents and recall the local arithmetic Siegel--Weil formula.

\subsection{Hermitian symmetric spaces}\label{sec:sym-space}
Let $F/F_0 = \BC/\BR$, and let $n \geq 2$ be a positive integer.  
Let $(V, h)$ be a non-degenerate hermitian space over $\BR$ of dimension $n$, with signature $(n-1,1)$.  
Let $\U(V)$ denote the unitary group.  

We fix an orthogonal decomposition $V = V^+ \oplus V^-$,  
where $V^+$ and $V^-$ are positive and negative definite, respectively.  
Then
\begin{equation*}
    K = \U(V^+) \times \U(V^-)
\end{equation*}  
is a maximal compact subgroup of $\U(V)$.

Let $\BP(V)$ denote the space of $1$-dimensional complex subspaces of $V$.  
Define the \emph{hermitian symmetric space}
\begin{equation} \label{eq:def_unitary_D}
    \CD = \{ z \in \BP(V) : h|_z < 0 \}
\end{equation}  
as the open subset of $\BP(V)$ consisting of negative definite lines.  
It is clear that $\CD$ carries a $\U(V)$-invariant complex structure, and we have  
$\CD \simeq \U(V) / K$.

There is another symmetric space associated with the analytic side.  
Let $W$ be the standard split $\BC/\BR$-skew-hermitian space of dimension $2n$.  
The hermitian symmetric domain for $G_n = \U(W)$ is the \emph{hermitian upper half-space}, given by  
\begin{align}
    \BH_n(\BC) ={}& \left\{ \btau \in \mathrm{Mat}_n(\BC) : \frac{1}{2i} (\btau - \prescript{t}{}{\ov{\btau}}) > 0 \right\} \notag \\
    ={}& \left\{ \btau = \bx + i \by : \bx \in \Herm_n(\BC), \by \in \Herm_n^+(\BC) \right\}. \label{equ:upper half plane}
\end{align}  

For any partition $\sum_{i=1}^{m} n_i = n$, we define an embedding  
\begin{equation*}
    \prod_{i=1}^{m} \BH_{n_i}(\BC) \to \BH_n(\BC), \quad  
    (\tau_i) \mapsto \btau =  
    \begin{pmatrix}
        \tau_1 & & & \\  
        & \tau_2 & & \\  
        & & \ddots & \\  
        & & & \tau_m  
    \end{pmatrix}.
\end{equation*}

\subsection{Special cycles and Green currents}\label{sec:archi-cycle}
Let $u \in V$ be \emph{any nonzero vector}.  
For any $z \in \CD$, we decompose $u$ orthogonally with respect to $z$ as  
\begin{equation*}
    u = u_z + u_{z^\perp},  
\end{equation*}  
where $u_z \in z$ and $u_{z^\perp} \perp z$.  
Define the function  
\begin{equation*}
    R(u, z) = -(u_z, u_z).
\end{equation*}  
The special cycle (often denoted by $\CD_u$) is then defined as  
\begin{equation*}
    \CZ(u) := \{ z \in \CD : z \perp u \} = \{ z \in \CD : R(u, z) = 0 \}.
\end{equation*}  
The cycle $\CZ(u)$ is nonempty if and only if $(u, u) > 0$, in which case it forms an analytic divisor on $\CD$.  

For $\tau = x + i y \in \BH_1(\BC)$, \emph{Kudla's Green function} is defined as  
\begin{equation*}
    \CG^\bK(u, y, z) = -\Ei(-4\pi y R(u, z)),\quad\text{where}\quad \Ei(x) = -\int_{1}^\infty \frac{e^{xt}}{t} dt.
\end{equation*}  
The function $\CG^\bK(u, y, -)$ is smooth on $\CD \setminus \CZ(u)$ and has a logarithmic singularity along $\CZ(u)$.  If $\CZ(u)$ is empty, then the function is smooth on $\CD$.
By \cite[Proposition 4.9]{Liu11-I}, it satisfies the $(1,1)$-current equation for $\CZ(u)$:  
\begin{equation}\label{equ:green-current-cycle}
    \rd \rd^c [\CG^\bK(u, y)] + \delta_{\CZ(u)} = [\omega_{\KM}(u, y)]\footnote{Here, we recall the standard differential operators:  
\begin{equation*}
    \rd = \partial + \bar\partial, \quad  
    \rd^c = \frac{1}{4\pi i} (\partial - \bar\partial), \quad  
    \rd \rd^c = -\frac{1}{2\pi i} \partial \bar\partial.
\end{equation*}  },
\end{equation}  
where $\omega_{\KM}(u, -) = e^{2\pi (u, u)} \varphi_\mathrm{KM}(u, -)$, and  
$\varphi_\mathrm{KM}(-,-) \in \CS(V) \otimes (A^{1,1}(\CD))^{\U(V)}$  
is the \emph{Kudla--Millson Schwartz form} (\cite{Kudla-Millson86}).  
We emphasize that the relation \eqref{equ:green-current-cycle} holds for \emph{any} nonzero vector $u\in V$, see \cite[Proposition 11.1]{Kudla97}.
When $u=0$, we set $\CG^{\bK}(0,y,z)=-\log |y|$.

Recall the following fact. Let $Z_1$ and $Z_2$ be analytic cycles in a smooth manifold $X$ of codimensions $d_1$ and $d_2$, respectively, that intersect properly.  
Let $G_1$ and $G_2$ be Green currents on $X$ satisfying  
\begin{equation*}
    \rd \rd^c G_i + \delta_{Z_i} = [\omega_i], \quad i = 1,2.
\end{equation*}  

When $G_2 = [g_2]$ is the current defined by a smooth form on $X\setminus Z_2$, the \emph{star product} is given by  
\begin{equation}\label{equ:star-prod}
    G_1 \star G_2 = \delta_{Z_1} \wedge G_2 + G_1 \wedge \omega_2.
\end{equation}  

Let $\bu = (u_1, \dots, u_r) \in V^r$ be a set of linearly independent vectors.  
Define the intersection of special cycles as  
\begin{equation*}
    \CZ(\bu) = \CZ(u_1) \cap \dots \cap \CZ(u_r).
\end{equation*}  

Let $\btau = \mathrm{diag}(\tau_1, \dots, \tau_r) = \bx + i\by \in \BH_r(\BC)$.  
\emph{Kudla's Green current} is defined via the star product:  
\begin{equation*}
    \CG^\bK(\bu, \by) \coloneqq [\CG^\bK(u_1, y_1)] \star \cdots \star [\CG^\bK(u_r, y_r)].
\end{equation*}  

It satisfies the $(r,r)$-current equation for $\CZ(\bu)$ (see also \cite[Theorem 2.7.1]{Garcia-Sankaran}):  
\begin{equation*}
    \rd \rd^c \CG^\bK(\bu, \by) + \delta_{\CZ(\bu)} =  
    [\omega_{\KM}(u_1, y_1) \wedge \cdots \wedge \omega_{\KM}(u_r, y_r)].
\end{equation*}  

\begin{remark}  
In \cite{Garcia-Sankaran}, Garcia and Sankaran generalize Kudla's Green currents to all $\btau \in \BH_n(\BC)$.
However, this level of generality is not needed for our purposes.  
\end{remark}  

\subsection{Archimedean local Siegel--Weil formula}
We now relate local Whittaker functions to their classical counterparts and recall the Siegel--Weil formula.  
Let $\omega$ denote the Weil representation, and let $W_T(g,s,\phi)$ be the local Whittaker function defined in \S \ref{sec:local weil}.
The hermitian symmetric space for $G_n = \U(W)$ is the hermitian upper half-space given in \eqref{equ:upper half plane}.  

For any $\btau=\bx+i\by\in\BH_n(\BC)$, the \emph{classical Whittaker function} is defined as  
\begin{equation*}
    W_T(\btau, s, \phi) := \chi_V(\det \ba)^{-1} \det(\by)^{-n/2} \cdot W_T(g_\btau, s, \phi),  
    \quad g_\btau := n(\bx) m(\ba) \in G_n(\BR),
\end{equation*}  
where $\ba \in \GL_n(\BC)$ satisfies $\by = \ba^t \cdot \ov{\ba}$.

For the remainder of this paper, we fix $\phi\in\CS(\BV^n)$ (at the archimedean place) to be the \emph{standard Gaussian function}:
\begin{equation}\label{equ:standard-gaussian}
    \phi(\bu) = e^{-2\pi \tr (\bu, \bu)} \in \CS(\BV^n),
\end{equation}  
where $\BV$ is a positive definite hermitian space of dimension $n$.  
The archimedean local (arithmetic) Siegel--Weil formula is as follows:  

\begin{theorem}\label{thm:archi Siegel Weil}
Let $T \in \Herm_n^{\reg}(\BC)$ be a nonsingular hermitian matrix, and define $q^T := \psi(\tr T\btau)$.  
\begin{enumerate}  
\item For $\mathrm{sign}(T) = (a, b)$, we have  
\begin{equation*}
    \ord_{s=0} W_T(\btau, s, \phi) \geq b.
\end{equation*}  

\item When $T$ is positive definite, i.e., $\mathrm{sign}(T)=(n,0)$, we have
\begin{equation*}
    W_T(\btau, 0, \phi) = \gamma_V^n \frac{(2\pi)^{n^2}}{\Gamma_n(n)} q^T.
\end{equation*}  

\item When $\mathrm{sign}(T) = (n-1,1)$, we have $W_T(\btau, 0, \phi)\equiv 0$, and  
\begin{equation*}
    W'_T(\btau, 0, \phi) = \Int(T, \by) \cdot \gamma_V^n \frac{(2\pi)^{n^2}}{\Gamma_n(n)} q^T.
\end{equation*}  
Here, the integral
$$
\Int(T, \by):=\int_{\CD} \CG^\bK(\bu, \by)
$$ 
is the \emph{local height function}, where $\bu \in V^n$ is a basis with $T = (\bu, \bu)$.  
By \cite[Proposition 4.10]{Liu11-I}, this local height function is independent of the choice of $\bu$.  
\end{enumerate}  
\end{theorem}  
\begin{proof}  
Parts (1) and (2) follow from \cite[Proposition 4.5]{Liu11-I}, while part (3) is given in \cite[Theorem 4.17]{Liu11-I}. Garcia and Sankaran give a different proof in \cite{Garcia-Sankaran}.
\end{proof}

\part{Global theory}\label{part:global}

\section{Shimura varieties and their integral models}\label{sec:RSZ}
We now turn to the global theory.
In this section, we review the RSZ Shimura varieties introduced in \cite{KR14,RSZ,RSZ-AGGP} and generalized in \cite{LMZ25}, along with their archimedean and non-archimedean uniformizations.

\subsection{Group-theoretic data}\label{sec:gp data}
Let $F/F_0 $ be a totally imaginary quadratic extension of a totally real number field with non-trivial Galois involution $x \mapsto \ov{x}$. Assume $F_0\neq \BQ$.
We write $\BA$, $\BA_0$ for the adele rings of $F$ and $F_0$, resp. We systematically use a subscript $f$ for the ring of finite adeles, and a superscript $p$ for the adeles away from the prime number $p$.

We fix a CM type $\Phi$ for $F$ with a distinguished element $\varphi_0\in\Phi$, and let $F_\infty=F\otimes_{\BQ}\BR\simeq \BC^\Phi$.
We fix a hermitian $F$-vector space $(V,h)$ of dimension $n$ such that, for $\varphi\in\Phi$,
\begin{equation*}
    \mathrm{sign}(V_\varphi)=\left\{\begin{array}{ll}(n-1,1)&\text{if }\varphi=\varphi_0,\\
    (n,0)&\text{if }\varphi\neq \varphi_0.\end{array}\right.
\end{equation*}
Let $G=\mathrm{U}(V)$ be the unitary group of $V$, viewed as an algebraic group over $F_0$. Following \cite[\S 3]{RSZ}, we consider the following reductive groups over $\mathbb Q$, where $c(g)$ is the similitude factor of $g$:
\begin{align*}
    Z^\BQ:={}&\{z\in\Res_{F/\BQ}\BG_m\mid \Nm_{F/F_0}(z)\in\BG_m\},\\
    G^\BQ:={}&\{g\in\Res_{F_0/\BQ}\mathrm{GU}(V)\mid c(g)\in\BG_m\},\\
    \wt{G}:={}&Z^\BQ\times_{\BG_m}G^\BQ\xrightarrow{\sim}Z^\BQ\times \Res_{F_0/\BQ}G.
\end{align*}
The last isomorphism is given by $(z,g)\mapsto (z, z^{-1}g)$. Consider the RSZ Shimura datum (\cite[\S3.2]{RSZ}),
$$
( \wt{G}, \{ h_{\wt{G}} \} ) = ( Z^{\BQ}, h_{\Phi}) \times ( \Res_{F_0/\BQ}G, \{ h_{\Res_{F_0/\BQ}G} \}). 
$$
For any neat compact open subgroup $\wt K=K_{Z^{\BQ}} \times K_{G} \leq \widetilde{G}(\mathbb A_f)$, we consider the \emph{RSZ unitary Shimura variety} associated to $( \widetilde{G}, \{ h_{\widetilde{G}} \} )$ with level $\wt{K}$
\[
\wt M:=\mathrm{Sh}_{\wt{K}}( \widetilde{G}, \{ h_{\widetilde{G}} \} )\to \Spec E,
\]
where $E$ is the reflex field. By \cite[(3.4)]{RSZ}, $E$ can be described through
\begin{align*}
\Gal(\ov{\BQ}/E)={}&\{\sigma\in \Gal(\ov{\BQ}/\BQ)\mid  \sigma\circ \varphi \in \Phi \text{ and }\mathrm{sign}(V_{\sigma\circ \varphi}) = \mathrm{sign}(V_\varphi)\text{ for all $\varphi\in \Phi$}\}.
\end{align*}
The goal of this section is to recall the definition of the (punctured) integral model $\CM\to \Spec O_E$ and its uniformization.

\subsection{Integral models of toric Shimura varieties}
We first recall the integral model for $\mathrm{Sh}_{K_{Z^{\BQ}}}(Z^{\BQ},h_\Phi)$, see \cite[\S 3.1]{RSZ}. 
Let $\fkd \in \BZ_{>1}$ be a product of primes such that the neat level $K_{Z^\BQ}$ factors as
\begin{equation*}
K_{Z^\BQ}=K_{Z^\BQ,\fkd}\times K_{Z^\BQ}^{\fkd}
\end{equation*}
where $K_{Z^\BQ}^{\fkd}$ is the maximal compact subgroup in $Z^\BQ(\BA^\fkd_f)$.

As an auxiliary datum, we also fix a hermitian $F_\fkd/F_{0,\fkd}$-module $\BW_\fkd$ of rank $1$. Following \cite[\S 3.2]{RSZ-AGGP} and \cite[Definition 3.6, \S 4.1]{RSZ}, we fix an ideal $\fka$ of $O_{F_0}$ and define 
$$
\CM^{\fka,}_0{}' \lr \Spec O_{E_\Phi}[\fkd^{-1}]
$$
as the functor sending a scheme $S$ over $O_{E_\Phi}[\fkd^{-1}]$ to the set of isomorphism classes of tuples $(\CA_0, \iota_0, \lambda_0, \eta_0)$ where 
\begin{enumerate}
 \item $\CA_0$ is an abelian scheme over $S$ of dimension $[F_0 : \mathbb Q]$.
\item $\iota_0: O_{F}[\fkd^{-1}]\to \End(\CA_0)[\fkd^{-1}]$ is an $O_F[\fkd^{-1}]$-action that is \emph{$(\Phi, \Phi)$-strict}, see \cite[\S 6.2]{LMZ25}.
\item $\lambda_0 \in \Hom(\CA_0, \CA_0^\vee)[\fkd^{-1}]$ is a \emph{prime-to-$\fkd$ polarization}, meaning it is a quasi-isogeny such that locally on $S$ some $\fkd^\BZ$-multiple is a polarization.
\item We require that $\iota_0$ and $\lambda_0$ are compatible in the sense that for all $a \in O_F^\fkd$, we have
$$ \lambda^{-1}_0 \circ \iota_0(a)^{\vee} \circ \lambda_0 = \iota_0(\overline{a}).$$
\item We require that $\ker(\fkd^k \lambda_0)[\fkd^{-1}] = \CA_0[\fka][\fkd^{-1}]$, where $k\geq 0$ is chosen (locally on $S$) such that $\fkd^k\lambda_0\in \Hom(\CA_0, \CA_0^\vee)$ (and the space is independent of the choice of $k$).
\item $\ov \eta_0$ is a $K_{Z^{\mathbb Q}_\fkd}$-level structure in the sense of \cite[C.3]{Liu-ARTF} using $\BW_0$.
\end{enumerate}
One may always choose $\fka$ and $\BW_0$ such that $\CM_0^{\fka,}{}'$ is non-empty, and we fix such choices. Then $\CM_0^{\fka,}{}'$ is representable by a finite \'etale $O^\fkd_{E_\Phi}$-scheme.
As in \cite[\S 3.4]{RSZ}, its generic fiber $(\CM_0^{\fka,}{}')_\BQ$ is a disjoint union of copies of the Shimura variety $\mathrm{Sh}_{K_{Z^{\BQ}}}(Z^{\BQ},h_{\Phi})$. 
We fix one copy in this union and let $\CM_0^{\fka}\subseteq \CM_0^{\fka,\prime}$ be its closure.
This is an open and closed subscheme of $\CM_0^{\fka,}{}'$ which provides the finite, \'etale integral model of $\mathrm{Sh}_{K_{Z^{\BQ}}}(Z^{\BQ},h_{\Phi})$.

\subsection{Punctured integral model}\label{sec:RSZ-integral}
Now we define the punctured integral model of $\mathrm{Sh}_{\wt K}(\wt G, \{h_{\wt G}\})$. Let $L$ be an $O_F$-lattice in $V$ and let $\fkd \in \BZ_{>1}$ be chosen such that:
\begin{itemize}
    \item If $v$ is a finite place of $F_0$ that is ramified in $F$, then $v\mid \fkd$.
    \item For any finite non-split place $v\nmid \fkd$ of $F_0$, the localization $L_v$ in $V_v=V \otimes_{F_0} F_{0, v}$ is a vertex lattice, i.e. $\varpi_v L_v^{\sharp}\subseteq L_v \subseteq L_v^\sharp$, where $\varpi_v\in F_{0,v}$ is a uniformizer.
    \item The level $K_G=K$ factors as 
    \begin{equation*}
        K=K_\fkd\times K^\fkd
    \end{equation*}
    where $K_\fkd\subset G(F_{0,\fkd})$ is any open compact subgroup and where $K^\fkd=\mathrm{Stab}(\wh{L}^\fkd)$ is the stabilizer of the completion $\wh{L}$ in $G(\BA_{0,f}^\fkd)$.
    \item The level $K_{Z^\BQ}$ factors similarly as
    \begin{equation*}
        K_{Z^\BQ}=K_{Z^\BQ,\fkd}\times K_{Z^\BQ}^{\fkd}
    \end{equation*}
    where $K_{Z^\BQ,\fkd}\subset Z^\BQ(\BQ_\fkd)$ is any open compact subgroup and where $K_{Z^\BQ}^{\fkd}$  is the maximal compact subgroup in $Z^\BQ(\BA^\fkd_f)$.
\end{itemize}

\begin{definition}
The \emph{punctured integral model} of the RSZ Shimura variety
$$
\tCM \lr \Spec O_E[\fkd^{-1}]
$$ 
is the functor that takes an $O_E[\fkd^{-1}]$-scheme $S$ to the set of isomorphism classes of tuples
$$(\CA_0, \iota_0, \lambda_0,\ov{\eta}_0; \CA, \iota, \lambda,\ov{\eta})$$
that are of the following form:
\begin{enumerate}
\item  $(\CA_0, \iota_0, \lambda_0, \ov{\eta}_0) \in \CM_0^\fka(S)$.
\item $\CA$ is an abelian scheme over $S$ of dimension $n[F_0:\mathbb Q]$.
\item $\iota: O_{F}[\fkd^{-1}] \to \End(\CA)[\fkd^{-1}]$ is an $O_F[\fkd^{-1}]$-action that is of signature $(n-1,1)$ in the sense of \cite[Definition 6.2]{LMZ25}.

\item $\lambda\in\Hom(\CA,\CA^\vee)[\fkd^{-1}]$ is a prime-to-$\fkd$ polarization.
\item We require $\iota$ and $\lambda$ to be compatible in the sense that for $a \in O_{F}[\fkd^{-1}]$ we have $ \lambda^{-1} \circ \iota(a)^{\vee} \circ \lambda= \iota(\overline{a})$. 
We also require that, for every place $v\nmid \fkd$ of $F$, corresponding to a prime ideal $\fkp_v\subset O_{F}[\fkd^{-1}]$, the group $\ker(\lambda)[\fkp_v^\infty]$ is annihilated by $\fkp_v$ and has order $|L^\vee_v/L_v|$.

\item $\ov{\eta}$ is a $K_{G,\fkd}$-orbit of isometries of hermitian $\prod_{p\mid\fkd}\BQ_p$-sheaves over $S$: 
\begin{equation*}
    \eta:V_{\fkd}(\CA_0,\CA)\overset{\sim}{\lr}V(F_{0,\fkd}).
\end{equation*}
Here, the notation is as follows. First,
\begin{equation*}
V_{\fkd}(\CA_0,\CA):=\prod_{p\mid\fkd}V_p(\CA_0,\CA)\quad\text{and}\quad
V_p(\CA_0,\CA)=\Hom_{F\otimes_{\BQ}\BQ_p}(V_p(\CA_0),V_p(\CA)),
\end{equation*}
where $V_p(\CA_0)$ and $V_p(\CA)$ denote the rational Tate modules at $p$; these are $\mathbb Q_p$-sheaves on $S$. 
Here, $V_\fkd(A_0, A)$ is made into a hermitian $F_\fkd$-vector space by
$$(x,y) := \lambda_0^{-1}\circ x^\vee\circ \lambda \circ y \in \End(V_\fkd(\CA_0)) = F_\fkd.$$
Second, we consider
\begin{equation*}
V(F_{0,\fkd}):=\prod_{p\mid\fkd}\BQ_p\otimes_{\BQ}V=\prod_{v\mid\fkd}F_{0,v}\otimes_{F_0}V
\end{equation*}
as a locally constant $\prod_{p\mid\fkd}\BQ_p$ sheaf on $S$.
\end{enumerate}
An isomorphism between two such tuples $(\CA_0, \iota_0, \lambda_0,\ov{\eta}_0; \CA, \iota, \lambda,\ov{\eta}) \cong (\CA_0', \iota_0', \lambda_0',\ov{\eta}_0'; \CA', \iota', \lambda',\ov{\eta}')$ is given by an isomorphism $(\CA_0, \iota_0, \lambda_0,\ov{\eta}_0) \cong (\CA_0', \iota_0', \lambda_0', \ov{\eta}'_0)$, and a prime-to-$\fkd$ isogeny $\CA \to \CA'$ which induces isomorphisms $\CA[p^\infty] \cong \CA'[p^\infty]$ $(p \nmid \fkd)$ compatible with $\iota$ and $\iota'$, $\lambda$ and $\lambda'$, $\overline{\eta}$ and $\ov{\eta}'$.
\end{definition}

\begin{proposition}[\protect{\cite[Proposition 6.2]{LMZ25}}]
The functor $\wt \CM$ is representable by a flat projective $O_E[\fkd^{-1}]$-scheme (recall that $F_0\neq \BQ$).\qed
\end{proposition}

\subsection{Basic uniformization}  

We now recall the $p$-adic uniformization of $\wt \CM$ along the basic locus.  

Let $\nu \nmid \fkd$ be a non-archimedean place of $E$, whose restriction to $F$ (resp. $F_0$) is denoted by $v$ (resp. $v_0$).  
Assume that $v_0$ is \emph{inert}. By \cite[Theorem 7.3]{LMZ25}, we have the following non-archimedean uniformization along the basic locus:  
\begin{equation}\label{equ:shimura-p-unif-comp}
    \tCM^{\wedge}_{O_{\breve{E}_\nu}} = \wt{G}^{(v_0)}(\BQ_p) \backslash  
    \left[\CN^{[h]}_{n, O_{\breve{E}_\nu}} \times \wt{G}(\BA^{v_0}_{0,f}) \slash K^{v_0}_{\wt{G}}\right].
\end{equation}  

Here, the left-hand side denotes the formal completion of the base change  
$\wt{\CM} \otimes O_{\breve{E}_\nu}$ along the basic locus of the geometric special fiber $\wt \CM_{\ov{\BF}_\nu}$.  
The basic locus consists of points $(A_0, A, \ov{\eta})$ such that $A[v^\infty]$ is a supersingular $p$-divisible group.  

On the right-hand side, $\wt{G}^{(v_0)}$ is the analogue of $\wt{G}$ for the nearby hermitian space $V^{(v_0)}$ (see \eqref{not:nearby herm}),  
where $V$ is fixed in \S \ref{sec:gp data}.  
The formal scheme $\CN_n^{[h]}$ is the (relative) Rapoport--Zink space for the quadratic extension $F_v/F_{0,v_0}$,  
where $h$ denotes the type of the vertex lattice $L_v$.

The uniformization \eqref{equ:shimura-p-unif-comp} induces a projection onto a discrete set:  
\begin{equation}\label{equ:fiberwise projection}
    \tCM^\wedge_{O_{\breve{E}_\nu}} \to Z^{\BQ}(\BQ) \backslash (Z^{\BQ}(\BA_{\BQ,f}) \slash K_{Z^{\BQ}}).
\end{equation}  
This provides a partition of the formal scheme $\tCM^\wedge_{O_{\breve{E}_\nu}}$, where each fiber is naturally isomorphic to  
\begin{equation}\label{equ:RZ-comp}
    \CM^{\wedge}_{O_{\breve{E}_\nu}} := G^{(v_0)}(F_0) \backslash  
    \left[\CN^{[h]}_{O_{\breve{E}_\nu}} \times G(\BA^{v_0}_{0,f}) \slash K_G^{v_0}\right].
\end{equation}

\subsection{Complex uniformization}\label{sec:complex-uniformization}
Let $\nu: E \hookrightarrow \BC$ be a complex place of the reflex field $E$.  
Its restriction to $F$ (resp. $F_0$) is denoted by $v$ (resp. $v_0$).  
Define the complex manifold  $\tCM_{\nu}:=\tCM\otimes_{O_E[\fkd^{-1}],\nu}\BC$ via $\nu$.  

Let $V^{(v_0)}$ be the ``nearby'' hermitian space defined in \eqref{not:nearby herm},  
i.e., the unique space that is positive definite at all archimedean places except at $v_0$,  
where it has signature $(n-1,1)$, and is locally isomorphic to $V$ at all non-archimedean places,  
where $V$ is fixed in \S \ref{sec:gp data}.  

Let $G^{(v_0)}$ (resp. $\wt{G}^{(v_0)}$) be the unitary group (resp. unitary similitude group), viewed as an $F_0$-algebraic group (resp. $\BQ$-algebraic group), associated with $V^{(v_0)}$.
Let $\CD^{(v_0)}$ be the hermitian symmetric space for $V^{(v_0)}(F_{v})$ defined in \eqref{eq:def_unitary_D},
which is the Grassmannian of negative definite $\BC$-lines in $V^{(v_0)} \otimes_{F,v} \BC$.  

Then, we have the following complex uniformization:  
\begin{equation*}
    \wt \CM_{\nu} = \wt{G}^{(v_0)}(\BQ) \backslash  
    \left[ \CD^{(v_0)} \times \wt{G}(\BA_{0,f}) \slash K_{\wt{G}} \right].
\end{equation*}  

We obtain a partition via the projection  
\begin{equation}\label{equ:complex projection}
    \wt \CM_{\nu} \to Z^{\BQ}(\BQ) \backslash (Z^{\BQ}(\BA_{\BQ,f}) \slash K_{Z^{\BQ}}),
\end{equation}  
where each fiber is naturally isomorphic to  
\begin{equation*}
    \CM_{\nu} := G^{(v_0)}(F_0) \backslash  
    \left[ \CD^{(v_0)} \times G(\BA_{0,f}) \slash K_G \right].
\end{equation*}  
Here, we fix an isomorphism $G^{(v_0)}(\BA_{0,f}) \simeq G(\BA_{0,f})$.

\section{Kudla--Rapoport special cycles}\label{sec:KR cycle in global}
In this section, we define weighted Kudla--Rapoport special cycles and their derived variants.  
We establish their complex and $p$-adic uniformizations and conclude by introducing their Green currents in the complex fibers.  

\subsection{Weighted special cycles}\label{sec:global KR cycle}
We keep the notation from \S \ref{sec:RSZ}.  
For any point $(A_0, A, \eta) \in \wt \CM(S)$, consider the hermitian form $h'(-,-)$ on  
$\Hom_{O_{F}}(A_0, A)[\fkd^{-1}]$ given by  
\begin{equation*}
    h'(u_1, u_2) = \lambda_0^{-1} \circ u_2^\vee \circ \lambda \circ u_1  
    \in \End_{O_{F}}(A_0)[\fkd^{-1}] \simeq O_{F}[\fkd^{-1}].
\end{equation*}  
By \cite[Lemma 2.7]{KR14}, this is totally positive definite. The polarization $\lambda: A \to A^\vee$ induces an injection  
\begin{equation*}
    \Hom_{O_F}(A_0, A)[\fkd^{-1}] \hookrightarrow  
    \Hom_{O_F}(A_0, A^\vee)[\fkd^{-1}].
\end{equation*}  

We first define global admissible Schwartz functions and associated special cycles, extending the local definition given in Definition \ref{def:local special cycle}.  
\begin{definition}  
Recall from \S \ref{sec:RSZ-integral} that we have fixed an integral lattice $L \subset V$.  
For a weight vector $\bw \in \{\CZ, \CY\}^m$, we construct the Schwartz function  
$\phi^\fkd(\bw) \in \CS(V(\BA_{0,f}^{\fkd})^m)$ such that  
\begin{equation*}
    \phi^{\fkd}(\bw) = \bigotimes_i \phi^{\fkd}(w_i),\quad\text{where}\quad\phi^{\fkd}(w_i) =  
    \begin{cases}  
        \charfun_{\wh L^{\fkd}}, & \text{if } w_i = \CZ; \\  
        \charfun_{\wh L^{\sharp,\fkd}}, & \text{if } w_i = \CY.  
    \end{cases}  
\end{equation*}  

A Schwartz function $\phi^{\fkd} \in \CS(V(\BA_{0,f}^{\fkd})^m)$ is called \emph{admissible}  
if there exists a weight vector $\bw$ such that $\phi^{\fkd} = \phi^{\fkd}(\bw)$,  
cf. Definition \ref{def:local special cycle}.  
\end{definition}  

\begin{definition}\label{defn:special-cycle}  
Fix an integer $0 < m \leq n$.  
Let $\phi^\fkd \in \CS(V(\BA_{0,f}^{\fkd})^m)$ be an admissible Schwartz function.  
Let $T \in \Herm_m(F)$ be \emph{any} hermitian matrix, and let  
$\mu \in V(F_{0,\fkd})^m / K_{\fkd}$ be a right $K_{\fkd}$-orbit.  
The \emph{weighted Kudla--Rapoport cycle} $\wt\GZ(T, \mu, \phi^\fkd)$  
associates to each scheme $S$ over $O_E[\fkd^{-1}]$ the groupoid of tuples $(A_0, \iota_0, \lambda_0, \ov{\eta}_0, A, \iota, \lambda, \ov{\eta}, \bu)$, where:  
\begin{itemize}  
\item $(A_0, \iota_0, \lambda_0, \ov{\eta}_0, A, \iota, \lambda, \ov{\eta}) \in \wt \CM(S)$.  
\item $\bu \in \Hom_{O_F}(A_0, A)^m \otimes_{\BZ} \BQ$ satisfies $h'(\bu, \bu) = T$  
      and $\ov{\eta}(\bu) \in \mu \subset V(F_{0,\fkd})^m$.  
\item For each $1 \leq i \leq m$, we impose the following conditions:  
    \begin{itemize}  
        \item If $\phi_{f,i}^{\fkd} = \charfun_{L^{\fkd}}$, then $u_i \in \Hom_{O_{F}}(A_0, A)[\fkd^{-1}]$.  
        \item If $\phi_{f,i}^{\fkd} = \charfun_{L^{\sharp, \fkd}}$, then $u_i \in \Hom_{O_{F}}(A_0, A^\vee)[\fkd^{-1}]$.  
    \end{itemize}  
\end{itemize}  

A morphism between two objects  
$(A_0, \iota_0, \lambda_0, \ov{\eta}_0, A, \iota, \lambda, \ov{\eta}, \bu)$  
and  
$(A'_0, \iota'_0, \lambda'_0, \ov{\eta}'_0, A', \iota', \lambda', \ov{\eta}', \bu')$  
consists of an isomorphism  
$
(A_0, \iota_0, \lambda_0, \ov{\eta}_0) \xrightarrow{\sim} (A_0', \iota_0', \lambda_0', \ov{\eta}_0')
$
in $\CM_{0}^{\fka}(S)$ and an $O_F$-linear prime-to-$\fkd$ isogeny  
$\varphi: A \to A'$, compatible with $\lambda$ and $\lambda'$,  
with $\ov{\eta}$ and $\ov{\eta}'$, and satisfying $\bu' = \varphi\circ \bu$.  
\end{definition}

We recall some basic properties of special cycles.  

\begin{proposition}\label{prop:special-cycle-decomp}  
\begin{enumerate}  
\item The morphism $i: \tCZ(T, \mu, \phi^\fkd) \to \wt \CM$ is representable, finite, and unramified.  

\item When $m=1$ and $T=t>0$ is totally positive, the morphism $i$ defines a Cartier divisor \'etale locally.

\item Let $T_i \in \Herm_{m_i}(F)$ for $i = 1,2$.  
Decompose $\mu_1 \times \mu_2 = \coprod_i \mu_{[i]} \subset V(F_{0,\fkd})^{m_1 + m_2}$  
into $K_{\fkd}$-orbits.  
Then, the scheme-theoretic fiber product decomposes as  
\begin{equation*}
    \tCZ(T_1, \mu_1, \phi^\fkd_1) \times_{\tCM} \tCZ(T_2, \mu_2, \phi^\fkd_2)  
    = \coprod_{\substack{T = {\begin{psmallmatrix} T_1 & * \\ * & T_2 \end{psmallmatrix}} \\  
    T \in \Herm_m(F)}} \coprod_i \tCZ(T, \mu_{[i]}, \phi^\fkd),
\end{equation*}  
where $m = m_1 + m_2$ and  
$\phi^\fkd = \phi^\fkd_1 \otimes \phi^\fkd_2 \in \CS(V(\BA_{0,f}^{\fkd})^m)$.  
\end{enumerate}    
\end{proposition}  

\begin{proof}  
By the argument of \cite[Lemma A.2.2]{RChenIII} (see also \cite[Proposition 2.7.2]{AHM17}),  
the morphism $i$ is quasi-compact.
Then, part (1) follows from \cite[Proposition 2.9]{KR14},  
see also \cite[Proposition 5.22(1)]{Liu-ARTF}.  

For part (2), the key idea originates from \cite[Proposition 3.5]{KR11},  
while a complete proof using the line bundle of modular forms can be found in  
\cite[Proposition 5.22(3)]{Liu-ARTF}.  
In our setting, these line bundles are constructed in \cite[\S 6.2]{LMZ25}  
via the Eisenstein condition (see also equation \eqref{equ:line bundle of modular form}).  
We note that proving $\tCZ(t, \mu, \charfun_{L^{\sharp,\fkd}})$  
is a Cartier divisor requires the line bundle of modular forms for $\omega_\CY$,  
see \cite[\S 5]{He-Luo-Shi-regular}.  

Part (3) follows directly from the moduli description,  
note that all but finitely many terms on the right hand side are empty.
\end{proof}  

\subsection{Derived special cycle}\label{sec:derived special cycle}
In this subsection, we define derived special cycles over the RSZ integral model.  
Since the integral model $\wt\CM \to \Spec O_E[\fkd^{-1}]$ is generally not regular, the definition requires extra care.
Instead of working with the $K$-group of perfect complexes, we consider the $K$-group of coherent sheaves (which we refer to as the $G$-group).  

An alternative approach to constructing derived special cycles is through the higher Gysin map in the Chow group,  
as described in \cite{Fulton-Intersection} and \cite[\href{https://stacks.math.columbia.edu/tag/0FBI}{0FBI}]{stacks-project}.  
For further discussion, see \cite[Remark 6.16]{LMZ25}. 
To prove the local Kudla--Rapoport conjecture, we may carefully choose the global embedding so that $\wt{\CM}$ is regular, see Remark \ref{rmk:globalization-unramified}.

Let $X$ be a noetherian scheme, and let $G_0(X)$ denote the $K$-group of coherent sheaves on $X$.  
For each $r \geq 0$, define $G_0(X)_r \subseteq G_0(X)$ as the subgroup generated by coherent $\CO_X$-modules $\CF$  
whose support has dimension at most $r$.  

Given a finite morphism $Y \to X$, let $G_0^Y(X)$ be the subgroup generated by coherent $\CO_X$-modules  
whose support is contained in the image of $Y$.  
For a comprehensive treatment of the intersection pairing in the setting of finite morphisms, we refer the reader to \cite[Appendix A]{Howard-Madapusi-22}; see also \cite[Appendix B]{Zhang21}.

\begin{definition}\label{def:derived-special-cycle}  
Let $\Herm_m^{\heartsuit}(F)$ denote the set of all hermitian matrices with totally positive diagonal entries.
\begin{enumerate}  
\item Let $0 < m \leq n$ be an integer, and consider $T \in \Herm^{\heartsuit}_m(F)$,  
an admissible Schwartz function $\phi^{\fkd}\in\CS(V(\BA^{\fkd}_{0,f})^m)$, and $\mu \in V(F_{0,\fkd})^m / K_{\fkd}$.  
Suppose $T$ has main diagonal entries $t_1, \dots, t_m$, and choose a vector  
$(v_1, \dots, v_m) \in \mu$.  
We define the \emph{derived Kudla--Rapoport cycle}
$\tDCZ(T, \mu, \phi^{\fkd}) \in G_0(\tCZ(T, \mu, \phi^{\fkd}))$ as  
\begin{equation*}
    \CO_{\tCZ(t_1, v_1 K_{\fkd}, \phi^{\fkd}_1)}  
    \Dotimes_{\CO_{\tCM}} \cdots  
    \Dotimes_{\CO_{\tCM}}  
    \CO_{\tCZ(t_m, v_m K_{\fkd}, \phi^{\fkd}_m)}  
    \big|_{\tCZ(T, \mu, \phi^{\fkd})}  
    \in G_0(\tCZ(T, \mu, \phi^{\fkd})).
\end{equation*}  
By Proposition \ref{prop:special-cycle-decomp} and \cite[Proposition 23.58]{GW-AGII},  
together with the fact that all morphisms are finite and unramified,  
the derived tensor product is well-defined and lies in $G_0(\wt\CM)_{n-m}$.  
Furthermore, this derived tensor product naturally carries a module structure over  
\begin{equation*}
    \CO_{\tCZ(t_1, v_1 K_{\fkd}, \phi^{\fkd}_1)}  
    \otimes_{\CO_{\tCM}} \cdots  
    \otimes_{\CO_{\tCM}}  
    \CO_{\tCZ(t_m, v_m K_{\fkd}, \phi^{\fkd}_m)}.
\end{equation*}  
We apply the idempotent corresponding to $\CO_{\tCZ(T, \mu, \phi^{\fkd})}$  
(see Proposition \ref{prop:special-cycle-decomp}(3))  
to obtain the restriction of the complex to $\tCZ(T, \mu, \phi^{\fkd})$.  Hence we see that $\tDCZ(T, \mu, \phi^{\fkd}) \in G_0(\tCZ(T, \mu, \phi^{\fkd}))_{n-m}$.

\item A Schwartz function $\phi_f \in \CS(V(\BA_{0,f})^m)^{K_G}$ is called \emph{admissible}  
if $\phi_f^{\fkd}$ is admissible, i.e.,  
$\phi_f^{\fkd} = \phi_f^{\fkd}(\bw)$ for some weight vector $\bw \in \{\CZ, \CY\}^m$.
For any $T \in \Herm^{\heartsuit}_{m}(F)$,  
we define the \emph{derived averaged weight special cycle} as
\begin{equation}\label{equ:defn derived special cycles}
    \tDCZ(T, \phi_f)  
    = \sum_{\mu \in V_T(F_{0,\fkd}) / K_{\fkd}}  
    \phi_{\fkd}(\mu) \tDCZ(T, \mu, \phi^{\fkd})  
    \in \bigoplus_{\substack{\mu \in V(F_{0,\fkd})^m / K_{\fkd} \\ \phi^{\fkd} \text{ admissible}}}  
    G_0(\tCZ(T, \mu, \phi^{\fkd}))_{n-m}.
\end{equation}  
Since the Schwartz function $\phi_{\fkd}$ is $K_{\fkd}$-invariant and has compact support,  
the summation in \eqref{equ:defn derived special cycles} is finite,  
so the right-hand side of \eqref{equ:defn derived special cycles}  
can be replaced by a finite direct sum.  
\end{enumerate}  
\end{definition}  

\begin{proposition}\label{prop:derived-special-cycle-decomp}  
Let $T_i \in \Herm^{\heartsuit}_{m_i}(F)$ be hermitian matrices for $i \in \{1,2\}$,  
and set $m = m_1 + m_2$.  
For any cosets $\mu_i \in V(F_{0,\fkd})^{m_i} / K_{\fkd}$ for $i \in \{1,2\}$,  
the \emph{derived intersection}  
\begin{equation*}  
    \tDCZ(T_1, \mu_1, \phi^{\fkd}_1) \cap^{\BL} \tDCZ(T_2, \mu_2, \phi^{\fkd}_2)  
\end{equation*}  
is defined as the restriction of the derived tensor product  
\begin{equation*}  
    \tDCZ(T_1, \mu_1, \phi^{\fkd}_1) \Dotimes_{\CO_{\tCM}} \tDCZ(T_2, \mu_2, \phi^{\fkd}_2)  
\end{equation*}  
from $G_0(\tCM)$ to  $G_0(\tCZ(T_1, \mu_1, \phi^{\fkd}_1) \times_{\tCM} \tCZ(T_2, \mu_2, \phi^{\fkd}_2))$.
Alternatively, it is the restriction of  
\begin{equation*}
    \CO_{\tCZ(t_1, v_1 K_{\fkd}, \phi^{\fkd}_1)}  
    \Dotimes_{\CO_{\tCM}} \cdots  
    \Dotimes_{\CO_{\tCM}}  
    \CO_{\tCZ(t_m, v_m K_{\fkd}, \phi^{\fkd}_m)}
\end{equation*}  
to $\tCZ(T_1, \mu_1, \phi^{\fkd}_1) \times_{\tCM} \tCZ(T_2, \mu_2, \phi^{\fkd}_2)$,  
where $t_i$ are the diagonal entries of $T_1$ and $T_2$. These two constructions agree in the derived category, and we pass to the $G$-group.
In particular, the derived intersection lies in the $m$-th step of the filtration.
\begin{enumerate}  
\item Decomposing  
    $\mu_1 \times \mu_2 = \coprod_i \mu_{[i]} \subset V(F_{0,\fkd})^{m}$  
    into $K_{\fkd}$-orbits, we obtain the decomposition:
    \begin{equation*}
        \tDCZ(T_1, \mu_1, \phi^{\fkd}_1) \cap^{\BL} \tDCZ(T_2, \mu_2, \phi^{\fkd}_2)  
        = \sum_{\substack{T =  
        \begin{psmallmatrix} T_1 & * \\ * & T_2 \end{psmallmatrix} \\  
        T \in \Herm_{m}^{\heartsuit}(F)}}  
        \sum_i \tDCZ(T, \mu_{[i]}, \phi^{\fkd}).
    \end{equation*}  

\item For $i \in \{1,2\}$, let $\phi_{i} \in \CS(V(\BA_{0,f})^{m_i})$ be an admissible Schwartz function.  
    Then, we have  
    \begin{equation*}
        \tDCZ(T_1, \phi_{1}) \cap^{\BL} \tDCZ(T_2, \phi_{2})  
        = \sum_{\substack{T =  
        \begin{psmallmatrix} T_1 & * \\ * & T_2 \end{psmallmatrix} \\  
        T \in \Herm^{\heartsuit}_{m}(F)}}  
        \tDCZ(T, \phi),
    \end{equation*}  
where $\phi=\phi_1\otimes\phi_2\in \CS(V(\BA_{0,f})^m)$.
\end{enumerate}  
\end{proposition}  
\begin{proof}  
Part (1) follows directly from the definition by applying Proposition \ref{prop:special-cycle-decomp}(3)  
repeatedly, along with \cite[Proposition 23.58]{GW-AGII}  
(or equivalently, using \cite[Definition 23.52 and Definition 23.57]{GW-AGII}).  
Moreover, all indices in the summation belong to $\Herm_m^{\heartsuit}(F)$  
since $T_i \in \Herm_{m_i}^{\heartsuit}(F)$.  
Part (2) follows immediately from part (1).  
\end{proof}  

\begin{remark}  
For any Schwartz function $\phi_{\fkd} \in \CS(V(F_{0,\fkd}))$, we define  
\begin{equation*}
    \tCZ(0, \phi_{\fkd} \charfun_{\wh L^{\fkd}}) := -\phi_{\fkd}(0) \omega_{\CZ},  
    \quad \text{and} \quad  
    \tCZ(0, \phi_{\fkd} \charfun_{\wh L^{\fkd,\sharp}}) := -\phi_{\fkd}(0) \omega_{\CY},
\end{equation*}  
where  
\begin{equation}\label{equ:line bundle of modular form}  
    \omega_\CZ = (\Lie A)_0 \otimes (\Lie A^\vee)_1,  
    \quad \text{and} \quad  
    \omega_\CY = (\Lie A^\vee)_0 \otimes (\Lie A)_1  
\end{equation}  
are line bundles of modular forms.  
Using these definitions, we can construct derived special cycles  
for hermitian matrices whose diagonal entries are either zero  
or totally positive (for all other hermitian matrices, the special cycles are empty).

It remains an open problem to establish linear invariance relations  
analogous to those in \cite[\S 5]{Howard-linear}, see also \cite[\S 3.6]{RChenI}.
\end{remark}

\subsection{Special cycles in the formal neighborhood of the basic locus}\label{sec:p adic unitofmization of special cycle}
We consider the restriction of the special cycles to the formal completion of $\tCM$ along the basic locus.  
Let $\nu \nmid \fkd$ be a non-archimedean place of $E$,  
whose restriction to $F$ (resp. $F_0$) is denoted by $v$ (resp. $v_0$).  
Assume that $v_0$ is inert.  

Let $\CN_n^{[h]} = \CN^{[h]}_{n, F_{v} / F_{0,v_0}} \to \Spf O_{\breve{F}_{v}}$  
be the relative Rapoport--Zink space, where $h$ denotes the type of the vertex lattice  
$L_{v_0} \subset V(F_{0,v_0})$.  
Consider its base change  
\begin{equation*}
    \CN^{[h]}_{n,O_{\breve{E}_\nu}} = \CN^{[h]}_n \wh{\otimes}_{O_{\breve{F}_{v}}} O_{\breve{E}_\nu}.
\end{equation*}  
By \eqref{equ:shimura-p-unif-comp}, we have the uniformization:
\begin{equation*}
    \CM^{\wedge}_{O_{\breve{E}_\nu}}  
    = G^{(v_0)}(F_0) \backslash  
    \left[ \CN^{[h]}_{n,O_{\breve{E}_\nu}} \times G(\BA^{v_0}_{0,f}) / K^{v_0}_G \right].
\end{equation*}  

From \S \ref{sec:non-archi-cycle},  
for each $\bu \in (V^{(v_0)}(F_{0,v_0}))^m$,  
we have the local weighted special cycle $\CZ(\bu, \phi_{v_0})$  
and the local derived special cycle $\DCZ(\bu, \phi_{v_0})$  
on $\CN_n^{[h]}$.  

For a pair $(\bu, g) \in ((V^{(v_0)}(F_0))^m \times G(\BA_{0,f}^{v_0})) / K_G^{v_0}$  
with $h'(\bu, \bu) \in \Herm^{\reg}_{m}(F)$,  
we define the closed formal subscheme of
$\CN^{[h]}_{n,O_{\breve{E}_\nu}} \times G(\BA^{v_0}_{0,f}) / K^{v_0}_G$ by  
\begin{equation}\label{equ:KR-unif}
    \CZ(\bu, g, \phi_{v_0})_{K_G^{v_0}} = \CZ(\bu, \phi_{v_0}) \times \charfun_{gK^{v_0}_G}.
\end{equation}  
We then consider the sum  
\begin{equation*}
    \sum \CZ(\bu', g', \phi_{v_0})_{K_G^{v_0}},
\end{equation*}  
where the sum runs over $(\bu', g')$ in the $G^{(v_0)}(F_0)$-orbit  
of the pair $(\bu, g)$ under the diagonal action of $G^{(v_0)}(F_0)$ on  
\begin{equation*}
    \left[(V^{(v_0)}(F_0))^m \times G(\BA^{v_0}_{0,f})\right] / K^{v_0}_G.
\end{equation*}  
Since this sum is $G^{(v_0)}(F_0)$-invariant,  
it descends to a cycle on $\CM^{\wedge}_{O_{\breve{E}_\nu}}$ in \eqref{equ:RZ-comp},  
which we denote by $[\CZ(\bu, g, \phi_{v_0})]_{K^{v_0}_G}$.
Similarly, we define $[\DCZ(\bu, g, \phi_{v_0})]_{K^{v_0}_G}$ in the $G$-group of $[\CZ(\bu, g, \phi_{v_0})]_{K^{v_0}_G}$.

\begin{proposition}\label{prop:KR-p-unif}  
Let $T \in \Herm_{m}^{+}(F)$.  
The restriction of the special cycle $\tCZ(T, \phi)$  
to each fiber $\CM^{\wedge}_{O_{\breve{E}_\nu}}$  
of the projection \eqref{equ:fiberwise projection}  
is given by the finite sum  
\begin{equation}  
    \sum_{(\bu, g) \in G^{(v_0)}(F_0) \backslash  
    \left[V^{(v_0)}_T(F_0) \times G(\BA^{v_0}_{0,f})\right] / K_G^{v_0}}  
    \phi^{v_0}(g^{-1} \bu) \cdot [\CZ(\bu, g, \phi_{v_0})]_{K^{v_0}_G},
\end{equation}  
viewed as a cycle on \eqref{equ:shimura-p-unif-comp}.  

Similarly, the restriction of the derived special cycles $\tDCZ(T, \phi)$  
is given by  
\begin{equation}  
    \sum_{(\bu, g) \in G^{(v_0)}(F_0) \backslash  
    \left[V^{(v_0)}_T(F_0) \times G(\BA^{v_0}_{0,f})\right] / K_G^{v_0}}  
    \phi^{v_0}(g^{-1} \bu) \cdot [\DCZ(\bu, g, \phi_{v_0})]_{K^{v_0}_G}.
\end{equation}  
\end{proposition}  

\begin{proof}  
This follows by the same reasoning as the basic uniformization  
of special divisors in \cite[Proposition 7.4]{Zhang21}.
For the derived special cycles, see also Proposition \ref{prop:inert-local-int}.  
\end{proof}

\subsection{Special cycles in the complex uniformization}\label{sec:special cycle complex uniformization}
Let $\nu: E \to \BC$ be a complex place of the reflex field,  
and let $V^{(v_0)}$ be the nearby hermitian space.  
Recall the notations for the complex uniformization from \S \ref{sec:complex-uniformization}.

For each $\bu \in V^{(v_0),m}(F_0)$ with totally positive norm,  
we define the special cycle $\CZ^{(v_0)}(\bu) \subset \CD^{(v_0)}$  
as the space of negative definite $\BC$-lines perpendicular to $\bu$  
(see \S \ref{sec:archi-cycle}).  

For a pair $(\bu, g) \in V^{(v_0),m}(F_0) \times G(\BA_{0,f}) / K_G$, we define the analytic submanifold
\begin{equation*}
    \CZ(\bu, g)_{K_G} = \CZ^{(v_0)}(\bu) \times \charfun_{g K_G} \subset \CD^{(v_0)} \times G(\BA_{0,f}).
\end{equation*}
We then consider the formal sum  
\begin{equation*}
    \sum \CZ(\bu', g')_{K_G}
\end{equation*}  
taken over $(\bu', g')$ in the $G^{(v_0)}(F_0)$-orbit of the pair $(\bu, g)$  
under the diagonal action of $G^{(v_0)}(F_0)$ on  
$V^{(v_0),m}(F_0) \times G(\BA_{0,f}) / K_G$.  
Since this sum is $G^{(v_0)}(F_0)$-invariant,  
it descends to a cycle on the quotient $\CM_\nu$, which we denote by $[\CZ(\bu, g)]_{K_G}$.  

\begin{proposition}\label{prop:complex uniformization of special cycle}  
Let $T \in \Herm^{+}_{m}(F)$.  
The restriction of the special cycle $\tCZ_{\nu}(T, \phi):=\tCZ(T, \phi)\otimes_{E, \nu} \BC$  
to each fiber $\CM_{\nu}$ of the projection \eqref{equ:complex projection} is given by  
\begin{equation*}\pushQED{\qed}  
    \sum_{(\bu, g) \in G^{(v_0)}(F_0) \backslash (V^{(v_0)}_T(F_0) \times G(\BA_{0,f}) / K_G)}  
    \phi(g^{-1} \bu) \cdot [\CZ(\bu, g)]_{K_G}.  
\qedhere\popQED
\end{equation*}  
\end{proposition}  

\subsection{Green currents of special cycles}\label{sec:green of cycles}
Let $\nu$ be an archimedean place of $E$, with its restriction to $F$ (resp. $F_0$)  
denoted by $v$ (resp. $v_0$).  
Since all constructions in this subsection depend only on $v_0$,  
in \S \ref{sec:local interection of KR cycles}, we will freely use the notation $\CG^?_{v_0}$ rather than $\CG^?_{\nu}$ for $? \in \{\bK, \bB\}$.  

For a pair $(\bu, g) \in V^{(v_0),m}(F_0) \times G(\BA_{0,f}) / K_G$, we define  
\begin{equation*}
    \CG_\nu^{\bK}(\bu, \by_\infty, z, g)_{K_G}  
    = \CG^{\bK}(\bu, \by_{v_0}, z) \times \charfun_{g K_G}, \quad z \in \CD^{(v_0)},
\end{equation*}
where $\btau_{\infty} = \mathrm{diag}(\tau_1, \dots, \tau_m) = \bx_{\infty} +i \by_{\infty} \in \BH_m(F_{\infty})$,  
with $\tau_i \in \BH_1(F_{\infty})$.  
Here, $\CG^{\bK}(\bu, \by_{v_0})$ is Kudla’s Green function  
for the analytic special cycle $\CZ(\bu)$, as defined in \S \ref{sec:archi-cycle}.

Consider the sum  
\begin{equation*}
    \sum \CG_\nu^{\bK, (v_0)}(\bu', \by_{\infty}, z, g')_{K_G},
\end{equation*}
taken over $(\bu', g')$ in the $G^{(v_0)}(F_0)$-orbit of the pair $(\bu, g)$  
under the diagonal action of $G^{(v_0)}(F_0)$  
on $V^{(v_0),m}(F_0) \times G(\BA_{0,f}) / K_G$.  
Since this sum is $G^{(v_0)}(F_0)$-invariant,  
it descends to a current on the quotient,  
which we denote by $[\CG^{\bK}(\bu, \btau_{v_0}, g)]_{K_G}$.  
A parallel construction holds for the automorphic Green function $\CG^{\bB}(\bu)$,  
which we will discuss in \S \ref{sec:generating}.  

For each $T \in \Herm_m^{\reg}(F)$ and  $\phi = \phi^{\fkd}(\bw) \otimes \phi_{\fkd} \in \CS(V(\BA_{0,f})^m)^{K_G}$,  where $\phi_\fkd \in \CS(V(F_{0,\fkd})^m)^{K_{G,\fkd}}$,  we define a current ${\tCG}_\nu^{\bK}(T, \by_\infty, z, \phi)$ on $\tCM_{\nu}$  such that its restriction to each fiber $\CM_{\nu}$  of the projection \eqref{equ:complex projection} is  
\begin{equation}\label{equ:KR_complex-unif}
    \CG_\nu^{\bK}(T, \by_\infty, z, \phi):=\sum_{(\bu, g) \in G^{(v_0)}(F_0) \backslash (V^{(v_0)}_T(F_0) \times G(\BA_{0,f}) / K_G)}  
    \phi(g^{-1} \bu) \cdot [\CG_\nu^{\bK}(\bu, \by_{\infty}, z, g)]_{K_G}.
\end{equation}

For any linearly independent $\bu \in V(\BR)^r$,  recall from \S \ref{sec:archi-cycle} that we have  the Kudla--Millson Schwartz form
$$
\omega_{\KM}(\bu, \by) \in (\CS(V(\BR))^r \otimes A^{(r,r)}(\CD^{(v_0)}))^{\U(V)(\BR)}.
$$
Similarly, we construct $[\omega_{\KM,\nu}(\bu, \by_\infty, g)]_{K_G}$  
and define a form $\wt\omega_{\KM,\nu}(T, \by_{\infty}, \phi)$  on $\wt\CM_{\nu}$ whose restriction to each fiber $\CM_\nu$ of the projection is  
\begin{equation*}
    \omega_{\KM,\nu}(T, \by_{\infty}, \phi):=\sum_{(\bu, g) \in G^{(v_0)}(F_0) \backslash (V^{(v_0)}_T(F_0) \times G(\BA_{0,f}) / K_G)}  
    \phi(g^{-1} \bu) \cdot [\omega_{\KM,\nu}(\bu, \by_{\infty}, g)]_{K_G}.
\end{equation*}

From the construction, we obtain  
\begin{equation*}\label{equ:green function for special cycles}
    \rd\rd^c \CG_\nu^{\bK}(T, \by_{\infty}, \phi)  
    + \delta_{\GZ(T, \phi)_{\nu}}  
    = \omega_{\KM,\nu}(T, \by_{\infty}, \phi),
\end{equation*}
showing that $\CG_\nu^{\bK}(T, \by_{\infty}, \phi)$  is a Green current for the special cycle $\CZ_\nu(T, \phi)$.  The initial construction of this current is due to Kudla \cite{Kudla97}.  

Note that the form $\omega_{\KM,\nu}(T, \by_\infty, \phi)$  is the $T$-coefficient of the theta function  $\Theta_{\KM,\nu}(\btau_\infty, \phi)$ considered in \cite{Kudla-Millson-90}.  For any $\phi \in \CS(V(\BA_f)^r)^{K_G}$, it is defined by  
\begin{equation*}
    \Theta_{\KM,\nu}(\btau_\infty, \phi)  := \sum_{\bu \in V^{(v_0)}}  \phi_{\KM,\nu}(\btau_\infty, \bu) \phi(\bu)  \in A^{(r,r)}(\CM_{\nu}),
\end{equation*}
where $\phi_{\KM,\nu} = \bigotimes_{w_0 \in \Sigma_{F_0,\infty}} \phi_{w_0}$ with $\phi_{v_0}$ the Kudla--Millson Schwartz form and $\phi_{w_0}$ the standard Gaussian functions defined in \eqref{equ:standard-gaussian} for $w_0 \neq v_0$.
It is a weight-$n$ automorphic function on $\bH(\BA_0)$ with the level determined by $\phi$.  

\begin{theorem}[{\cite[Theorem 4.10]{Garcia-Sankaran}}]\label{thm:star-prod-cycles}
Let $\CG_\nu^{\bK}(T_1,\by_1,\phi_{1})$ and $\CG_\nu^{\bK}(T_2,\by_2,\phi_{2})$  
be Green currents associated with the special cycles $\GZ_\nu(T_1,\phi_{1})$  
and $\GZ_\nu(T_2,\phi_{2})$.  
Let $\phi=\phi_{1}\otimes\phi_{2}$ be a non-singular Schwartz function.  
Assume that $T_1$ and $T_2$ are non-degenerate  
and that $\CZ_\nu(T_1,\phi_{1})$ and $\CZ_\nu(T_2,\phi_{2})$ intersect properly.  
Recalling \eqref{equ:star-prod}, we define the star product as follows:
\begin{equation*}
\CG_\nu^{\bK}(T_1,\by_1,\phi_{1})\star\CG_\nu^{\bK}(T_2,\by_2,\phi_{2})
:=\CG_\nu^{\bK}(T_1,\by_1,\phi_{1})\wedge\delta_{\GZ(T_2,\phi_{2})}
+\CG_\nu^{\bK}(T_2,\by_2,\phi_{2})\wedge \omega_{\KM,\nu}(T_1,\by_1,\phi_{1}).
\end{equation*}
Under these assumptions, we obtain
\begin{equation*}
\CG_\nu^{\bK}(T_1,\by_1,\phi_{1})\star\CG_\nu^{\bK}(T_2,\by_2,\phi_{2})
\equiv\sum_{\substack{T=
    \begin{psmallmatrix}T_1&*\\ *& T_{2}\end{psmallmatrix}}}
    \CG_\nu^{\bK}\left(T,\begin{pmatrix}\by_1&\\&\by_2\end{pmatrix},\phi\right) 
    \mod \mathrm{im}\,\partial +\mathrm{im}\,\ov{\partial}.
\end{equation*}
Here, the summation ranges over all $T\in\Herm_{m_1+m_2}(F)$.
\end{theorem}

In our applications, we will always assume that $T_i$ is positive definite for $i\in\{1,2\}$  
and that $\phi$ is regular at two places over $\fkd$ (see Theorem \ref{thm:sing-vanish}).  
Under these conditions, the sum is taken over all $T\in\Herm_n^{\reg}(F)$,  
and Theorem \ref{thm:star-prod-cycles} follows directly from the definitions.

\begin{proposition}\label{prop:omega series}
\begin{enumerate}
\item For any positive definite hermitian matrix $T_\flat\in\Herm_{n-1}^+(F)$, the integral
\begin{equation*}
\int_{\CM_{\nu,\BC}} \CG_\nu^{\bK}(T_\flat,\by_\flat,\phi_\flat)\wedge \omega_{\KM,\nu}(\xi, y,\phi_1)
\end{equation*}
converges absolutely.

\item Let $\phi_1\in\CS(V(\BA_{0,f}))^{K_G}$ be invariant under $K_{\bH} \subset \bH(\BA_{0,f})$  
with respect to the Weil representation.  
Then, for any $\tau_\flat\in\BH_{n-1}(F_{\infty})$ and $h_f\in\bH(\BA_{0,f})$,  
the generating series
\begin{equation*}
\omega_{T_\flat,\nu}(\tau,h_f,\phi_1,\phi_\flat)
=\sum_{\xi\in F_0}
\int_{\CM_{\nu,\BC}} \CG_\nu^{\bK}(T_\flat,\by_\flat,\phi_\flat)\wedge \omega_{\KM,\nu}(\xi,y,\omega(h_f)\phi_1) q^\xi
\end{equation*}
lies in $\CA_{\infty}(\bH(\BA_0),K_{\bH},n)$.
\end{enumerate}
\end{proposition}

\begin{proof}
The series is well-defined since $\CZ_\nu(T_\flat,\phi_\flat)$  consists of finitely many points in the complex fiber and Kudla's Green currents have only logarithmic singularities at these points.  Moreover, the generating series converges because $\omega_{\KM,\nu}$  is defined as the product of the Kudla--Millson series  with a function that has exponential decay in $\xi$ (see \eqref{equ:green-current-cycle}).  The modularity of the generating series follows directly  from the modularity of the Kudla--Millson series.
\end{proof}

\section{Eisenstein series of unitary groups}\label{sec:Eisenstein series}
In this section, we review the Siegel Eisenstein series and their Fourier--Jacobi expansion. Let $\BV$ be the \emph{incoherent} $\BA/\BA_0$-Hermitian space associated with $V$, as fixed in \S \ref{sec:gp data}. This means that $\BV$ is totally positive definite and satisfies $\BV_{v_0} \simeq V_{v_0}$ for all finite places $v_0 \in \Sigma_{F_0}$.

\subsection{Siegel Eisenstein series}
Let $W$ be the standard split $F/F_0$-skew-hermitian space of dimension $2n$.
Let $G_n=\U(W)$ and let $P_n(\BA_{F_0})=M_n(\BA_{F_0})N_n(\BA_{F_0})$ be the standard Siegel parabolic subgroup of $G_n(\BA_{F_0})$, where
\begin{align*}
    M_n(\BA_{F_0})={}&\left\{
    m(\ba)=\begin{pmatrix}\ba&0\\0& \prescript{t}{}{\ov{\ba}}^{-1}\end{pmatrix}: \ba\in\GL_n(\BA)\right\},\\
    N_n(\BA_{F_0})={}&\left\{\begin{pmatrix}1_n&\bb\\0&1_n\end{pmatrix}: \bb\in\Herm_n(\BA)\right\}.
\end{align*}
We define 
\begin{equation*}
    w_n=\begin{pmatrix}
        0&1_n\\ -1_n&0
    \end{pmatrix}\in G_n.
\end{equation*}

Let \(\eta: \BA_{F_0}^{\times}/F_0^\times \to \BC^\times\) be the quadratic character associated with the extension \(F/F_0\).  
Fix a character \(\chi: \BA^\times_{F} \to \BC^\times\) such that \(\chi|_{\BA_{F_0}^\times} = \eta^n\).  

We may regard \(\chi\) as a character on \(M_n(\BA_{F_0})\) via \(\chi(m(\ba)) = \chi(\det(\ba))\) and extend it to \(P_n(\BA_{F_0})\) by setting it to be trivial on \(N_n(\BA_{F_0})\).  
The \emph{degenerate principal series} is then defined as the unnormalized smooth induction  
$$
I_n(s,\chi) := \mathrm{Ind}^{G_n(\BA_{F_0})}_{P_n(\BA_{F_0})} (\chi \cdot |-|_F^{s+n/2}), \quad s \in \BC.
$$

For a standard section \(\Phi(-,s) \in I_n(s,\chi)\) (i.e., a section whose restriction to the standard maximal compact subgroup of \(G_n(\BA_{F_0})\) is independent of \(s\)), we define the associated \emph{Siegel Eisenstein series} by  
\[
\Eis(g,s,\Phi) := \sum_{\gamma \in P_n(F_0) \backslash G_n(F_0)} \Phi(\gamma g, s), \quad g \in G_n(\BA_{F_0}).
\]
This series converges for \(\Re(s) \gg 0\), admits a meromorphic continuation to the entire complex plane, and is holomorphic at \(s=0\).  
Note that \(\Eis(g,s,\Phi)\) depends on the choice of \(\chi\).  

For $T \in \Herm_n(F)$ and $g \in G_n(\BA_{F_0})$, define  
\begin{equation*}
    \Eis_T(g, s, \Phi) = \int_{N_n(F_0) \backslash N_n(\BA_{F_0})} \Eis(n(\bb)g, s, \Phi) \psi(-\tr(T\bb)) \, dn(\bb).
\end{equation*}
Then, we have the decomposition  
\begin{equation} \label{equ:Eis-Fourier}
    \Eis(g, s, \Phi) = \sum_{T \in \Herm_n(F)} \Eis_T(g, s, \Phi).
\end{equation}

For $\Re(s) \gg 0$, the integral may be unfolded in the usual way. If $\det T \neq 0$ and $\Phi(s) = \otimes_{v_0} \Phi_{v_0}(s)$ is factorizable, the Fourier coefficient admits a factorization into a product:
\begin{equation*}
    \Eis_T(g, s, \Phi) = \prod_{v_0 \in \Sigma_{F_0}} W_{T, v_0}(g_{v_0}, s, \Phi_{v_0}),
\end{equation*}
where
\begin{equation*}
    W_{T, v_0}(g_{v_0}, s, \Phi_{v_0}) = \int_{\Herm_n(F_{v_0})} \Phi_{v_0}(w_n^{-1} n(\bb) g_{v_0}, s) \psi_{v_0}(-\tr(T\bb)) \, dn(\bb)
\end{equation*}
is the \emph{local Whittaker function}, as defined in \S\ref{sec:local weil}. Furthermore, we have the equivariance property
\begin{equation*}
    W_{T, v_0}(g_{v_0}, s, \Phi_{v_0}) = W_{T, v_0}(e, s, r(g_{v_0})\Phi_{v_0}),
\end{equation*}
where $r(g_{v_0})$ denotes the action of right translation by $g_{v_0} \in G_n(F_{0, v_0})$ in the local induced representation $I_{n, v_0}(s, \chi_{v_0})$.

Recall that, for $T$ nonsingular, these local integrals have an entire analytic continuation. In other words, we have the following result:

\begin{proposition}[\cite{Kudla97}]\label{prop:Fourier-coeff-non-sing}
For $T \in \Herm_n^{\reg}(F)$, the function $\Eis_T(g, s, \Phi)$ admits a meromorphic continuation and is holomorphic at $s = 0$. In particular, the following expression holds:
\begin{equation*}
    \Eis_T(g, s, \Phi) = \prod_{v_0} W_{T,v_0}(g_{v_0}, s, \Phi_{v_0})
\end{equation*}
for $s$ near $0$. The vanishing of the nonsingular coefficients at $s = 0$ is determined by the vanishing of the local factors $W_{T, v_0}(1, s, \Phi_{v_0})$.\qed
\end{proposition}

Next, we study the central derivative of the Siegel Eisenstein series. The main argument follows \cite[\S 6]{Kudla97}. The Fourier expansion in \eqref{equ:Eis-Fourier} is well-defined near $s = 0$, and from now on, we will restrict our attention to this region.

\begin{theorem}\label{thm:Eis-cent-deriv}
The central derivative $\PEis(g, 0, \Phi)$ is well-defined and is an automorphic form on $G_n(\BA_0)$. Moreover, we have
\begin{equation*}
    \PEis(g, 0, \Phi) = \sum_T  \frac{\rd}{\rd s} \Big|_{s=0} \Eis_T(g, s, \Phi),
\end{equation*}
which is absolutely convergent. Furthermore, for nonsingular $T$, we have 
\begin{equation*}
    \frac{\rd}{\rd s}\Big|_{s=0} \Eis_T(g, s, \Phi)  = \PEis_T(g, 0, \Phi),
\end{equation*}
i.e., the Fourier coefficients are preserved under differentiation.
\end{theorem}

\begin{proof}
Since the function $b \mapsto \Eis(n(b) g, s, \Phi)$ is smooth on $\Herm_n(\BA)/\Herm_n(F)$ and holomorphic in a neighborhood of $s = 0$, the Fourier expansion convergent absolutely and uniform for $s$ in a closed disk around $s = 0$ (see \cite[\S 6]{Kudla97}).
Hence, it can be differentiated term by term with respect to $s$, and the resulting series for the derivative remains absolutely convergent.
By differentiating term by term with respect to $s$, evaluating at $s = 0$, and collecting the terms for nonsingular $T$, we obtain the first assertion. The second assertion follows from Fubini's theorem.
\end{proof}

\subsection{Weil representation}\label{sec:weil representation}
The Eisenstein series of interest to us arise from the Weil representations. By class field theory, we may, and will, choose an additive character $\psi: \BA_{F_0}/F_0 \to \BC^\times$ such that $\psi$ is unramified outside of $\mathrm{Spl}(F/F_0)$, the set of finite places of $F_0$ that split in $F$.

Let $\CV$ be an $\BA_F/\BA_{F_0}$-hermitian space of rank $n$. 
We say that the space $\CV$ (resp., the Schwartz function $\phi \in \CS(\CV)$ and the Eisenstein series $\Eis(g, s, \phi)$) is \emph{coherent} if $\CV = V \otimes_{F_0} \BA_0$ for some $F/F_0$-Hermitian space $V$, and \emph{incoherent} otherwise.

Let $\CS(\CV^n)$ be the space of Schwartz functions on $\CV^n$. The fixed choice of $\chi$ and $\psi$ gives a \emph{Weil representation} $\omega=\omega_{\chi,\psi}$ of $G_n(\BA_F)\times \U(\CV)$ on $\CS(\CV^n)$. Explicitly, for $\phi\in\CS(\CV^n)$ and $\bu\in\CV^n$,
\begin{align*}
\omega(m(\ba))\phi(\bu)&=\chi(m(\ba))|\det \ba|_F^{n/2}\phi(\bu\cdot a),&m(\ba)=\begin{psmallmatrix}
\ba&\\&\prescript{t}{}{\ov{\ba}^{-1}}
\end{psmallmatrix}
\in M_n(\BA_F),\\
\omega(n(\bb))\phi(\bu)&=\psi(\tr \bb(\bu,\bu))\phi(\bu),&n(\bb)=\begin{psmallmatrix}
1_n& \bb\\ &1_n
\end{psmallmatrix}
\in N_n(\mathbb{A}_F),\\
\omega_\chi(w_n)\phi(\bu)&=\gamma_{\CV}^n\cdot\widehat \phi(\bu),&w_n=\begin{psmallmatrix}
  & 1_n\\
  -1_n & \\
\end{psmallmatrix},\\
\omega(h)\phi(\bu)&=\phi(h^{-1}\cdot\bu),& h\in \U(\CV).
\end{align*}
Here $\gamma_{\CV}$ is the Weil constant (see \cite[(10.3)]{KR14}), and $\widehat\phi$ is the Fourier transform of $\phi$ using the self-dual Haar measure on $\CV^n$ with respect to $\psi\circ\tr_{F/F_0}$.

For $\phi \in \CS(\CV^n)$, define the function  
\begin{equation*}
    \Phi_\phi(g) := \omega(g) \phi(0), \quad g \in G_n(\BA_F).
\end{equation*}
Then, $\Phi_\phi \in I_n(0, \chi)$. Let $\Phi_\phi(-, s) \in I_n(s, \chi)$ be the associated standard section, known as the \emph{standard Siegel--Weil section} associated to $\phi$.

For $\phi \in \CS(\CV^n)$, we define the Eisenstein series and its Fourier coefficients as follows:
\begin{equation*}
    \Eis(g, s, \phi) := \Eis(g, s, \Phi_{\phi}),
    \quad \Eis_T(g, s, \phi) := \Eis_T(g, s, \Phi_{\phi}).
\end{equation*}
Similarly, we write  
\begin{equation*}
    W_{T, v_0}(g_{v_0}, s, \phi_{v_0})
\end{equation*}
for the local Whittaker function.

For a finite place $v_0$, we define the space of \emph{regular test functions} as
\begin{equation*}
    \CS(\CV_{v_0}^{n})_{\reg} = \{\phi_{v_0} \in \CS(\CV_{v_0}^n) \mid \phi_{v_0}(\bu) = 0 \text{ if } \det (\bu, \bu) = 0\}.
\end{equation*}

Fix a finite subset $S \subset \Sigma_{F_0,f}$ with $|S| = k > 0$ and define  
\[
    \CS(\CV_S^{n})_{\reg} = \bigotimes_{v \in S} \CS(\CV_v^n)_{\reg}.
\]  
A Schwartz function $\phi\in\CS(\CV^n)$ is \emph{regular} at $k$ places if there exists a finite subset $S\subset \Sigma_f$ such that $|S|=k$ and that $\phi_S\in\CS(\CV_S^n)_{\reg}$.
The following result is due to \cite[\S 6]{YZZ12} and \cite[2.B]{Liu11-II}:

\begin{theorem}\label{thm:sing-vanish}
Suppose $\phi \in \CS(\CV^n)$ is regular at $k$ places, and that $\CV_\infty$ is positive definite, and that $\phi_\infty\in\CS(\CV_\infty^n)$ is the product of standard Gaussian functions as in \eqref{equ:standard-gaussian}.

Then, for $T$ singular and $g \in P_n(\BA_{F_0, S}) G_n(\BA_{F_0}^S)$, we have  
\begin{equation*}
    \ord_{s=0} \Eis_T(g,s,\phi) \geq k.
\end{equation*}
\end{theorem}

\begin{proof}
The proof in \cite[2.B]{Liu11-II} applies directly, although it is originally stated for the case when $\CV$ is even-dimensional. See also \cite[(12.3.0.2)]{LZ21}.
\end{proof}

\begin{remark}\label{rem:reg-fun-construct}
We can construct a regular test function $\phi_{v_0}$ of the form  
\begin{equation*}
    \phi_{v_0} = \bigotimes_i \phi_{i, v_0}, \quad
    \phi_{i, v_0} = \charfun_{D_i} \in \CS(\CV_{v_0}),
\end{equation*}
where $D_i \subset \CV_{v_0}$ are open subsets. This follows from the fact that any open subset in a non-archimedean metric space can be expressed as a union of boxes.
Furthermore, we can choose $\phi_{v_0}$ to be a $\mathbb{Q}$-valued function.
\end{remark}

\subsection{Classical Eisenstein series associated to the Shimura datum}
Recall from \S\ref{sec:sym-space} that we defined the Hermitian upper half-space for $G_n = \U(W)$, with $F_\infty = F \otimes_{\BQ} \BR \simeq \BC^\Phi$.

\begin{notation}\label{not:eisenstein}
For the remaining part of the paper, we assume that the $F_\infty/F_{0,\infty}$-hermitian space $\CV_\infty$ is positive definite, and we fix $\phi_\infty\in \CS(\CV_\infty^n)$ to be the product of the standard Gaussian functions \eqref{equ:standard-gaussian}.
\end{notation}

For any $\phi\in\CS(\CV_f^n)$, the \emph{classical Eisenstein series} for $\btau_\infty=\bx_\infty+i\by_\infty\in\BH_n(F_\infty)$ is defined as  
\begin{equation*}
    \Eis(\btau_{\infty}, s, \phi) := \chi_\infty(\det(\ba))^{-1} \det(\by_\infty)^{-n/2} \cdot \Eis(g_{\btau_\infty}, s, \phi\otimes \phi_\infty),
    \quad g_{\btau_\infty} := n(\bx) m(\ba) \in G_n(\BA_0),
\end{equation*}
where $\ba \in \GL_n(F_\infty)$ is chosen such that $\by = \ba^t \cdot \overline{\ba}$.  
See \S \ref{sec:weil representation} for the notation.
Notice that $\Eis(\btau_\infty, s, \phi)$ does not depend on the choice of $\chi$, and that $g_{\btau, f} = 1 \in G_n(\BA_{0, f})$.  

The Fourier expansion \eqref{equ:Eis-Fourier} can then be rewritten as  
\begin{equation*}
    \Eis(\btau_\infty, s, \phi) = \sum_T \Eis_T(\btau_\infty, s, \phi) = \sum_T C_T(\by_\infty, s, \phi) q^T, 
    \quad q^T = \psi_\infty(\tr(T \tau)).
\end{equation*}

We write the central derivative as  
\begin{equation*}
    \PEis(\btau_\infty, \phi) := \PEis(\btau_\infty, 0, \phi), \quad
    \PEis_T(\btau_\infty, \phi) := \PEis_T(\btau_\infty, 0, \phi).
\end{equation*}
A similar notation will be used for Whittaker functions.

\begin{proposition}\label{prop:fourier-expansion-of-eis}
Assume that $\phi\in \CS(\CV^n_f)$ is regular at two or more finite places. Then:
\begin{enumerate}
\item The Eisenstein series admits a Fourier expansion:
\begin{equation*}
    \Eis(\btau_\infty, s, \phi) = \sum_{T \in \Herm_n(F)} \Eis_T(\btau_\infty, s, \phi)
    = \sum_{T \in \Herm_n(F)} C_T(\by_\infty, s, \phi) q^T,
\end{equation*}
where $C_T(\by_\infty,s,\phi):=\Eis_T(\btau_\infty,s,\phi)/q^T$ is the $T$-th Fourier coefficient in the classical sense.
\item The central derivative of the Eisenstein series also has a Fourier expansion:
\begin{equation*}
    \PEis(\btau_\infty, \phi) = \sum_{T \in \Herm_n^{\reg}(F)} \PEis_T(\btau_\infty, \phi)
    = \sum_{T \in \Herm_n^{\reg}(F)} \PC_T(\by_\infty, \phi) q^T,
\end{equation*}
where $\PC_T(\by_\infty, \phi):=\PEis_T(\btau_\infty, \phi)/q^T$.
For each term, we have the expression
\begin{equation*}
    \PEis_T(\btau_\infty, \phi) = \sum_{v_0 \in \Sigma_{F_0}} \PEis_{T, v_0}(\btau_\infty, \phi), \quad
    \PEis_{T, v_0}(\btau_\infty, \phi) = W'_{T, v_0}(\btau_\infty, \phi_{v_0}) \cdot W_T(\btau_\infty, \phi^{v_0}),
\end{equation*}
where $W_{T, v_0}(\btau_\infty, \phi_{v_0})$ depends on $\btau_\infty$ only when $v_0 \mid \infty$.
\end{enumerate}
\end{proposition}

\begin{proof}
Part (1) follows from Theorem \ref{thm:sing-vanish} and the fact that $g_{\btau, f} = 1 \in G_n(\BA_{F_0, f})$.
Part (2) follows from Theorem \ref{thm:Eis-cent-deriv} and Proposition \ref{prop:Fourier-coeff-non-sing}.
\end{proof}

Let $T \in \Herm_n^{\reg}(F)$ either be totally positive definite ($T \in \Herm_n^{+}(F)$) or have signature $(n-1,1)$ at one place while being positive definite at all other places ($T \in \Herm_n^{\dagger}(F)$).
By Theorem \ref{thm:archi Siegel Weil}(1), only these two cases yield non-vanishing Fourier coefficients for $\Eis(\btau, \phi)$ and $\PEis(\btau, \phi)$.

\subsection{Fourier--Jacobi series}\label{sec:FJ series}
In this subsection, we recall the Fourier--Jacobi type expansion for Hermitian Siegel modular forms.
For more detailed definitions and references, we refer the reader to \cite{Braun-I}.

Let $f(\btau): \BH_n(F_\infty) \to \BC$ be a classical Siegel modular form (not necessarily holomorphic) with a Fourier expansion that is absolutely convergent:
\begin{equation*}
    f(\btau) = \sum_{T \in \Herm_n(F)} f_T(\by) q^T, \quad q^T = \psi_\infty(\tr(T \btau)).
\end{equation*}
We write
\begin{equation*}
    \btau = \begin{pmatrix}
        \tau & a \\ b & \tau_\flat
    \end{pmatrix},
    \quad \text{and} \quad
    T = \begin{pmatrix}
        \xi & t_{12} \\ \prescript{t}{}{\ov{t}_{12}} & T_\flat
    \end{pmatrix} = T(\xi, t_{12}, T_\flat).
\end{equation*}
Then the modular form $f(\btau)$ has a Fourier--Jacobi type expansion:
\begin{equation*}
    f(\btau) = \sum_{T_\flat \in \Herm_{n-1}(F)} f^{\FJ}_{T_\flat}(\tau, a, b) q^{T_\flat}, \quad 
    q^{T_\flat} = \psi_{\infty}(\tr(T_\flat \tau_\flat)),
\end{equation*}
where
\begin{equation*}
    f^{\FJ}_{T_\flat}(\tau, a, b) = \sum_{T = \begin{psmallmatrix}
        \xi & t_{12} \\ t_{21} & T_\flat
    \end{psmallmatrix}} f_T(\by) \cdot \psi_\infty(\xi \tau + t_{12} b) \psi_{\infty}(\tr(\prescript{t}{}{\ov{t}_{12}} a)).
\end{equation*}

If we set $a = b = 0$, we obtain the function with the Fourier expansion:
\begin{equation*}
    f^{\FJ}_{T_\flat}(\tau) := f^{\FJ}_{T_\flat}(\tau, 0, 0) = \sum_{\xi \in F_0} \sum_{T = \begin{psmallmatrix}
        \xi & * \\ * & T_\flat
    \end{psmallmatrix}} f_T(\by) q^\xi, \quad q^\xi = \psi_\infty(\xi \tau).
\end{equation*}

Note that this function also depends on $\by_\flat$, the imaginary part of $\tau_\flat$. We will fix $\tau_\flat = \bi$ from now on.  
Since $\phi_\infty$ is a pure tensor of standard Gaussian functions, the Fourier--Jacobi expansion shows that $f^{\FJ}_{T_\flat}(\tau, a, b)$ is a weight-$n$ modular form, where $g \in G_1(F_0)$ acts by
\begin{equation*}
    g \cdot f^{\FJ}_{T_\flat}(\tau, a, b) = f^{\FJ}_{T_\flat}(g \tau, g a, b g).
\end{equation*}
In particular, the automorphic function associated with $f^{\FJ}_{T_\flat}(\tau)$ is $\bH(F_0)$-invariant.

Recall from Proposition \ref{prop:fourier-expansion-of-eis} that the Eisenstein series and its central derivative admit Fourier expansions, which are absolutely convergent by Theorem \ref{thm:Eis-cent-deriv}:
\begin{equation*}
    \Eis(\btau, \phi) = \sum_{T \in \Herm_n^{\reg}(F)} C_T(\by, \phi) q^T, 
    \quad \text{and} \quad 
    \PEis(\btau, \phi) = \sum_{T \in \Herm_n^{\reg}(F)} \PC_T(\by, \phi) q^T.
\end{equation*}
We consider the functions
\begin{equation*}
    \FJ_{T_\flat}(\tau, \phi) = \sum_{\xi \in F_0} \sum_{T = \begin{psmallmatrix} \xi & * \\ * & T_\flat \end{psmallmatrix}} C_T(\by, \phi) q^\xi, 
    \quad \text{and} \quad 
    \PFJ_{T_\flat}(\tau, \phi) = \sum_{\xi \in F_0} \sum_{T = \begin{psmallmatrix} \xi & * \\ * & T_\flat \end{psmallmatrix}} \PC_T(\by, \phi) q^\xi.
\end{equation*}

Consider the Weil representation $\omega_1$ acting on $\CS(\BV^n)$ in the first factor. By \cite[A.2]{Kudla97}, this action is compatible with the Weil representation $\omega$. In particular, for any $h \in \bH(\BA_0)$, we have
\begin{equation*}
    \omega(h) \cdot \FJ_{T_\flat}(\tau_\infty, \phi) = \FJ_{T_\flat}(h_\infty \tau_\infty, \omega(h_f) \phi), 
    \quad \text{and} \quad 
    \omega_1(h) \cdot \PFJ_{T_\flat}(\tau_\infty, \phi) = \PFJ_{T_\flat}(h_\infty \tau_\infty, \omega_1(h_f) \phi).
\end{equation*}

We conclude the following:

\begin{proposition}\label{prop:Fourier--Jacobi series are modular forms}
Suppose that $\phi_1 \in \CS(\CV)$ is invariant under $K_{\bH} \subset \bH(\BA_{0, f})$ via the Weil representation, and that $\phi=\phi_1\otimes \phi_\flat$ is regular at two places.
Then the generating series $\FJ_{T_\flat}(\tau, \phi)$ and $\PFJ_{T_\flat}(\tau, \phi)$ lie in $\CA_{\infty}(\bH(\BA_0), K_{\bH}, n)$. \qed
\end{proposition}

\part{Proof of the Kudla--Rapoport conjecture}\label{part:proof}
\section{Arithmetic intersection of special cycles}\label{sec:generating}
In this section, we review the concept of almost intersection of admissible cycles, as established in \cite{Mihatsch-Zhang} and further generalized in \cite{ZZhang24,LMZ25}.
Due to the non-regularity of the integral model $\tCM$, we must work with the arithmetic Picard group instead of arithmetic Chow groups.

\subsection{Arithmetic Picard group and admissible line bundles}
Let $E$ be a number field and let $\fkd \in \BZ_{>1}$. Let $\CX$ be a flat, proper scheme of finite type over $O_E[\fkd^{-1}]$ with smooth generic fiber $X$.

Following \cite[\S 2.1.2]{BGS94-JAMES}, a \emph{hermitian vector bundle} on $\CX$ is a pair $\wh{\CE} = (\CE, (\lVert-\rVert_{\nu})_{\nu\in\Hom(E,\BC)})$, where $\CE$ is a locally free coherent $\CO_\CX$-module, and for each complex embedding $\nu \in \Hom(E,\BC)$, the holomorphic vector bundle $\CE_\nu$ over the complex fiber $\CX_\nu$ is equipped with a hermitian metric $\lVert-\rVert_\nu$.

Let $\wh{\Pic}(\CX)=\wh{\Pic}(\CX)_\BQ$ be the rational space of isomorphism classes of hermitian line bundles on $\CX$. 
The tensor product endows $\wh{\Pic}(\CX)$ with the structure of an abelian group. The identity element in $\wh{\Pic}(\CX)$ is the structure sheaf $\CO_\CX$ equipped with the trivial metric, and the inverse of the class of $\wh{\CL}$ is given by its dual bundle. For a detailed treatment, see \cite[\S 1 and 2]{GS90-Ann-I}.

Fix a Kähler metric $h_\nu$ on the complex fiber $\CX_\nu$ for every $\nu \in \Hom(E,\BC)$ that is compatible with complex conjugation. This metric induces a Laplace operator on differential forms on $X_\nu$. Let $A^{1,1}(X_{\nu},\BR)$ denote the $\BR$-vector space of real differential forms of type $(1,1)$ on $X_\nu$.
A hermitian line bundle $\wh{\CL} = (\CL, (\lVert-\rVert_\nu)_{\nu\in \Hom(E,\BC)})$ is called \emph{admissible} if its first Chern form $c_1\bigl(\CL_{\nu},\lVert-\rVert_\nu\bigr) \in A^{1,1}(X_\nu,\BR)$, as defined in \cite[\S 2.3]{GS90-Ann-I}, is harmonic for each $\nu\in \Hom(E,\BC)$. 

Let $\wh{\Pic}^{\rm{adm}}(\CX)$ be the subgroup consisting of classes of admissible hermitian line bundles on $\CX$. According to \cite[\S 2.6]{GS90-Ann-I}, the admissible metric $\lVert-\rVert_\nu$ is unique up to a constant scalar on each factor.\footnote{In \cite[\S 2.6]{GS90-Ann-I} and its references, this constant is fixed by the choice of K\"ahler form; see \cite[\S 1, equation ($*$)]{Arakelov74}.}

Note that while \cite{GS90-Ann-I} assumes that the scheme $\CX$ is regular with $\fkd = \emptyset$, the existence and uniqueness of the admissible metric depend only on the complex fiber. Therefore, these assumptions are not relevant in this context, see \cite[\S 2]{GS92-Invent} for the most general treatment.

\begin{definition}
\begin{enumerate}
\item For any finite place $\nu_0$ of $E$, let $\wh{\rm{Pic}}_{\nu_0}(\CX)\subset \wh{\rm{Pic}}(\CX)$ be the subgroup generated by hermitian line bundles $(\CO(-\CX_{\nu_0,i}), (h_{\nu})_{\nu\in \Hom(E,\BC)})$, where $\CX_{\nu_0,i}$ are effective Cartier divisors of $\CX$ whose underlying topological space is supported on the special fiber $\CX\otimes{\BF_{\nu_0}}$ and where $h_\nu$ is the fixed Kähler metric on the complex fiber $\CO(-\CX_{\nu_0,i})_{\nu}\cong\CO_{\CX_\nu}$, for each $\nu\in \Hom(E,\BC)$. 

By definition, we have an isomorphism $$\wh{\Pic}_{\nu_0}(\CX)\cong \Pic_{\nu_0}(\CX),$$ 
where the right hand side is the subgroup of $\Pic(\CX)$ generated by the same classes but without the hermitian metric.

\item Let $\wh{\Pic}_{\infty}(\CX)\subset  \wh{\rm{Pic}}(\CX)$ be the subgroup generated by $(\CO_{\CX},(c_\nu h_\nu)_{\nu\in \Hom(E,\BC)})$ where $c_\nu$ is a constant and $h_\nu$ is the Kähler metric we fixed on $\CX$ for each $\nu\in \Hom(E,\BC)$.
\end{enumerate}
\end{definition}
We define  $\wh{\mathrm{Pic}}_{\mathrm{vert}}(\CX)$ as the subgroup of $\wh{\mathrm{Pic}}(\CX)$ generated by $\wh{\rm{Pic}}_{\nu}(\CX)$ for all finite places $\nu$ of $E$ together with the subgroup $\wh{\rm{Pic}}_\infty(\CX)$.

\begin{lemma}\label{lem:picard pair deg zero}
\begin{enumerate}
\item The projection $\wh{\rm{Pic}}^{\rm{adm}}(\CX)\to \Pic(\CX_\eta)$ given by $(\CL,(\lVert-\rVert_\nu)_{\nu\in \Hom(E,\BC)})\mapsto \CL_\eta$ induces a left exact sequence
$$0\lr \wh{\Pic}_{\mathrm{vert}}(\CX)\lr 
\wh{\Pic}^{\rm{adm}}(\CX)\lr \Pic(\CX_\eta),$$
where $\CX_\eta$ is the generic fiber of the integral model and $\CL_\eta$ is the pull-back of the line bundle $\CL$ to the generic fiber.
\item For any finite place $\nu$ such that $\CX_\nu\otimes O_{E,\nu}$ is smooth over $O_{E,\nu}$, we have $\wh{\Pic}_{\nu}(\CX)\subset \wh{\Pic}_{\infty}(\CX)$.
\end{enumerate}
\end{lemma}
\begin{proof}
See \cite[Lemma 8.2]{LMZ25}.
\end{proof}

Let $\wt{\CZ}_1(\CX)$ be the ($\BQ$-valued) group of $1$-cycles on $\CX$ and let $\CZ_1(\CX)$ be the quotient of $\wt{\CZ}_1(\CX)$ by the subgroup generated by $1$-cycles that are rationally trivial on a closed fiber of $\CX$.
Recall from \cite[Proposition 2.3.1]{BGS94-JAMES} that we have a well-defined intersection pairing:
\begin{equation}\label{eq: Truncated Arithmetic int pairing}
    \wh{\Pic}(\CX)\times \CZ_1(\CX)\to \BR_{\fkd}:=\BR/\mathrm{Span}_{\BQ}\{\log p:  p\mid \fkd\},
\end{equation}
see also remarks on the p.941 of \cite{BGS94-JAMES}.

Let $\wt{\CZ}_{1}(\CX)_{\deg=0}$ be the subgroup of $1$-cycles that have zero degree on each connected component of the generic fiber of $\CX$. This subgroup is the orthogonal complement of $\wh{\Pic}_{\infty}(\CX)$.  Let $\wt{\CZ}_{1}(\CX)^{\perp} \subseteq \wt{\CZ}_{1}(\CX)$ be the orthogonal complement of $\wh{\Pic}_{\text{vert}}(\CX)$. 
Let $\Pic^{\rm{adm}}(\CX_\eta)$ be the image of the map $\wh{\Pic}^{\rm{adm}}(\CX)\to \Pic(\CX_\eta)$ in Lemma \ref{lem:picard pair deg zero}.
By Lemma \ref{lem:picard pair deg zero}, the truncated arithmetic intersection pairing (\ref{eq: Truncated Arithmetic int pairing}) induces a bilinear pairing:
\begin{equation}\label{equ:almost intersection}
(-,-): \, \Pic(\CX_\eta)^\text{adm}\times \wt{\CZ}_{1}(\CX)^{\perp}  \to \BR_\fkd.
\end{equation}

\subsection{Arithmetic Kudla--Rapoport special divisors}
Let $\CZ\subset \CX$ be an effective Cartier divisor with associated line bundle $\CO(-\CZ)$. On a complex fiber, let $\CG_\nu$ be a Green current associated to the complex subvariety $\CZ_\nu\subset \CX_\nu$. Then the pair $(\CZ_\nu,\CG_\nu)$ determines a hermitian line bundle $(\CO(-\CZ)_\nu,\lVert-\rVert_\nu)$ over the complex fiber (see \cite[\S 2]{GS90-Ann-I}).
By abuse of notation, let $(\CZ,(\CG_\nu)_{\nu\in\Phi})$ be the corresponding hermitian line bundle in the arithmetic Picard group.

We return to the global notation in \S \ref{sec:RSZ} and \S \ref{sec:KR cycle in global}.
Let $\phi\in\CS(V(\BA_{0,f}))^{K_G}$ be an admissible Schwartz function and let $K_\bH \subset \bH(\BA_{0,f})$ be any compact open subgroup under which $\phi$ is invariant via the Weil representation. For $\xi\in F_0$, we have the Kudla--Rapoport divisor $\tCZ(\xi,\phi)\to\tCM$.

For $\nu\in \Hom(E,\BC)$ and $\xi\in F_0^+$, we have the \emph{automorphic Green function} $\CG_\nu^{\bf B}(\xi, y_{\infty}, \phi)=\CG^{\bf B}(\xi,\phi)$ introduced in \cite{Bruinier12}. By the work of Bruinier \cite[Corollary 5.16]{Bruinier12}, this automorphic Green function is admissible. While its definition is similar to that of Kudla's Green function, the precise details are not essential for this paper.

For $\xi \in F_{0}^{+}$, define \emph{arithmetic Kudla--Rapoport divisor}
\begin{equation*}\label{eq: Fourier coeff of arithmetic generating series}
 \tCZ^{\bB} (\xi, \phi)= (\tCZ (\xi, \phi), ( \tCG^{\bB}_{\nu}(\xi,\phi))_{\nu} ) \in \wh{\Pic}(\tCM).  
\end{equation*}

\begin{definition}\label{def:bruinier generating series}
\begin{enumerate}
\item We take $\omega_\CZ$ and  $\omega_\CY$ as in \cite[Definition 6.7]{LMZ25} as extensions of the automorphic line bundle $\omega$ to $\tCM$. Define the arithmetic line bundle 
$$\tCZ^{\bB} (0, \phi_{\fkd}\charfun_{L^{\fkd}}):=-\phi_{\fkd}(0)(\omega_\CZ, \lVert  -  \rVert_{\mathrm{Pet}}),\quad\text{and}\quad \tCZ^{\bB} (0, \phi_{\fkd}\charfun_{L^{\sharp,\fkd}}):=-\phi_{\fkd}(0)(\omega_\CY, \lVert  -  \rVert_{\mathrm{Pet}})$$ 
where $\lVert  -  \rVert_{\mathrm{Pet}}$ denotes the natural Petersson metric. 
\item For any admissible Schwartz function $\phi$ in $\CS(V(\BA_{0,f}))$, we have the generating series 
\begin{equation*}
\tCZ^{\bB}(\tau,\phi)=\tCZ^\bB(0,\phi)+\sum_{\xi\in F_0^+}\tCZ^{\bB}(\xi,y,\phi)q^\xi,\quad q^\xi=\psi_\infty(\xi\tau).
\end{equation*}
It is defined as a formal series in the arithmetic Picard group $\wh{\Pic}(\tCM)[\![q]\!]$.
\end{enumerate}
\end{definition}

\begin{proposition}\label{prop:modularity-bruinier}
Let $\phi_1\in \CS(V(\BA_{F_0,f}))^{K_G}$ be $K_{\bH}$-invariant under the Weil representation and
let $\CC\in \wt{\CZ}_1(\wt{\CM})^{\perp}$ be any $1$-cycle. The arithmetic intersection product 
\begin{equation*}
    (\tCZ^{\bB}(\tau,\phi_1),\CC)\in \BR_{\fkd}[\![q]\!]
\end{equation*}
is well-defined and is a weight-$n$ holomorphic Hilbert modular form up to a constant term (for which we use the same notation). 
More precisely, the generating function $\tau \in\BH_1(F_\infty)\mapsto   (\tCZ^{\bB}(\tau,\phi_1),\CC)$ lies in $\CA_{\hol}(\bH(\BA_0),K_{\bH},n)_{\ov{\BQ}}\otimes_{\ov{\BQ}}\BR_{\fkd,\ov{\BQ}}$.
\end{proposition}
\begin{proof}
Since $\tCZ^{\bB}(\tau,\phi)\in \wh{\Pic}^{\mathrm{adm}}(\tCM)[\![q]\!]$, the result follows from \eqref{equ:almost intersection} and \cite[Theorem 2.4]{Mihatsch-Zhang}.
\end{proof}

Let $Z\to \tCM$ be any finite map with one-dimensional image and zero dimensional complex fibers.
The composition $G_0(Z)_{1}\to G_0^Z(\tCM)_1\to \mathrm{Gr}_1G_0^Z(\tCM)\cong \Ch_1^Z(\tCM)$ induces a natural map:
\begin{equation}\label{equ:K-gorup to 1-cycle}
    G_0(Z)_{1}\to \CZ_1(\tCM),
\end{equation}
where the last isomorphism is due to \cite[\href{https://stacks.math.columbia.edu/tag/0FEW}{0FEW}]{stacks-project}.
In particular, let $\phi_\flat\in\CS(V(\BA_{0,f})^{n-1})$ be an admissible Schwartz function.
We define a $1$-cycle $\tDCZ(T_\flat,\phi_\flat)\in \CZ_1(\tCM)$ via Definition \ref{def:derived-special-cycle}. Hence we obtain the arithmetic intersection pairing
\begin{equation}\label{equ:intersection product}
(\tCZ^{\bB}(\tau,\phi_1),\tDCZ(T_\flat,\phi_\flat))\in \BR_{\fkd}[\![q]\!].
\end{equation}

\section{Local intersection numbers}\label{sec:local-int}
In this section, we calculate the non-archimedean local contribution of the generating series in Definition \ref{def:bruinier generating series} to the arithmetic intersection product given in equation \eqref{equ:intersection product}. We also compute the archimedean local contribution for its variant, the ``Kudla generating series''.

\subsection{Intersection of Kudla--Rapoport cycles}\label{sec:local interection of KR cycles}
Recall the group data from \S \ref{sec:gp data}. In particular, we have a global CM field $F/F_0$ with a fixed CM type $\Phi$, along with an $F/F_0$-hermitian space of dimension $n$ and an integral lattice $L \subset V$.

For the remainder of the paper, we impose the following assumptions on Schwartz functions $\phi=\phi_1\otimes\phi_\flat\in\CS(V(\BA_{0,f})^n)^{K_G}$:
\begin{flalign}\label{equ:assumption on schwartz function}
\begin{array}{l}
\qquad\bullet\quad\text{$\phi_{1}=\phi_{1,\fkd}\otimes \phi_{1}^{\fkd}\in\CS(V(\BA_{0,f}))$ such that $\phi_{1}^{\fkd}$ is admissible.}\\
\qquad\bullet\quad\text{$\phi_{\flat}=\phi_{\flat,\fkd}\otimes \phi_{\flat}^{\fkd}\in\CS(V(\BA_{0,f})^{n-1})$ such that $\phi_{\flat}^{\fkd}$ is admissible.}\\
\qquad\bullet\quad\text{The function $\phi=\phi_{1}\otimes\phi_{\flat}$ is regular in at two places over $\fkd$.}
\end{array}&&\tag{$\bigstar$}
\end{flalign}

Under this assumption, for any singular $T \in \Herm_n(F)$, the derived averaged weighted special cycle $\tDCZ(T,\phi)$ vanishes by Definition \ref{def:derived-special-cycle}(2). Consequently, for any positive element $\xi \in F_0^+$ and any positive definite Hermitian matrix $T_\flat \in \Herm_{n-1}^+(F)$, the derived intersection can be expressed as
\begin{equation}\label{equ:derived intersection}
\tDCZ(\xi,\phi_{1})\cap^{\BL}\tDCZ(T_\flat,\phi_{\flat})=\sum_{\substack{T={\begin{psmallmatrix}\xi&*\\ *& T_{\flat}\end{psmallmatrix}}\\
    T\in\Herm^{\reg}_{n}(F)}}\tDCZ(T,\phi).
\end{equation}
By \cite[Lemma 2.7]{KR14}, any nonsingular matrix $T$ appearing in the summation is positive semi-definite and therefore must be positive definite.

For each $\xi\in F_0$, equation \eqref{eq: Truncated Arithmetic int pairing} defines an intersection number of $\tCZ^{\bB}(\xi,\phi_1)$ and $\tDCZ(T_\flat,\phi_\flat)$ valued in $\BR_\fkd$. For $\xi\neq 0$, the intersection is proper over the complex fiber by \cite[Lemma 2.20]{KR14}, and we can define a more specific intersection number
\begin{equation}\label{equ:lift int number}
    \Int_{T_\flat}^\bB(\xi,y,\phi_1,\phi_{\flat}):=\frac{1}{\tau(Z^{\BQ})[E:F]}\left(\tCZ^\bB(\xi,\tau,\phi_{1}),\tDCZ(T_{\flat},\phi_{\flat})\right)\in\BR,
\end{equation}
which lifts the $\BR_{\fkd}$-valued intersection number (this depends on the choice of the section $s$ of the line bundle).
Here
\begin{equation}\label{equ:tau ZQ}
    \tau(Z^{\BQ})=\#(Z^{\BQ}(\BQ)\backslash Z^{\BQ}(\BA_{\BQ,f})\slash K_{Z^{\BQ}}).
\end{equation}

Moreover, the intersection number \eqref{equ:lift int number} is defined as the sum of an archimedean and a non-archimedean part, such that:
\begin{itemize}
    \item The \emph{archimedean contribution} is defined in the usual way by evaluating the Green function at the complex point of the Kudla--Rapoport cycle:
$$
\sum_{\nu\in\Sigma_{E,\infty}}\langle \tCG_\nu^\bB(\xi, y, \phi_1), \tDCZ_\nu(T_\flat, \phi_{\flat})\rangle_\nu \log q_\nu,
$$
where by definition $\log q_\nu=2$ for complex places $\nu$ (and $1$ if $\nu$ were a real place).
    \item The \emph{non-archimedean contribution} is given by the arithmetic Euler--Poincar\'e characteristic:
\begin{equation*}
    \chi(\tCM, \tDCZ(\xi, \phi_{1}) \cap^{\BL} \tDCZ(T_\flat, \phi_{\flat})):=\sum_{\nu\in \Sigma_{E,f}} \langle \tDCZ(\xi, \phi_{1}),\tDCZ(T_\flat, \phi_{\flat})\rangle_\nu \log q_\nu,
\end{equation*}
where the local intersection number $\langle-,-\rangle_\nu$ is defined for a non-archimedean place $\nu$ through the Euler--Poincar\'e characteristic of a derived tensor product on $\tCM\otimes_{O_E} O_{E_\nu}$.
\end{itemize}

With this definition of $\Int_{T_\flat}^\bB(\xi, \phi_1, \phi_{\flat})$, we obtain a well-defined place-by-place decomposition:
\begin{equation*}
\Int_{T_\flat}^\bB(\xi, \phi_1, \phi_{\flat}) = \sum_{v_0 \in \Sigma_{F_0,\infty}} \Int_{T_\flat,v_0}^\bB(\xi,\phi_1, \phi_{\flat})+\sum_{v_0 \in \Sigma^{\fkd}_{F_0,f}} \Int_{T_\flat,v_0}(\xi,\phi_1, \phi_{\flat}),
\end{equation*}
where $\Int_{T_\flat,v_0}^\bB(\xi, \phi_1, \phi_{\flat})$ is the sum of the local contributions from all $\nu \in \Sigma_{E,\infty}$ lying above $v_0$ and $\Int_{T_\flat,v_0}(\xi, \phi_1, \phi_{\flat})$ is the sum of the local contributions from all $\nu \in \Sigma_{E,f}$ lying above $v_0$.

For a finite place $v_0 \in \Sigma^{\fkd}_{F_0,f}$, by the same proof as in \cite[Lemma 2.21]{KR14}, the local contribution $\Int_{T_\flat}(\xi, \tau_\infty, \phi_1, \phi_{\flat})$ is nonzero only if the quadratic extension $F/F_0$ is non-split at $v_0$, hence 
\begin{equation*}
\Int_{T_\flat}^\bB(\xi, \phi_1, \phi_{\flat}) = \sum_{v_0 \in \Sigma_{F_0,\infty}} \Int_{T_\flat,v_0}^\bB(\xi,\phi_1, \phi_{\flat})+\sum_{v_0 \in \Sigma^{\fkd}_{F_0,\mathrm{inert}}} \Int_{T_\flat,v_0}(\xi,\phi_1, \phi_{\flat}),
\end{equation*}
where $\Sigma^{\fkd}_{F_0,\mathrm{inert}}$ denotes the set of all finite places $v_0\nmid \fkd$ of $F_0$ such that $F/F_0$ is inert.

Recall from \S \ref{sec:green of cycles} that we have Kudla's Green functions $\CG_\nu^{\bf K}(\xi,y_\infty,\phi_1)$ associated to the special divisor $\GZ_\nu(\xi,\phi_1)$ on $\CM_\nu$, where $\tau_\infty=x_\infty+i y_\infty \in \BH_1(F_{\infty})$ is a variable. For our application, we also consider 
\begin{equation*}
\Int_{T_\flat}^\bK(\xi, \tau_\infty, \phi_1, \phi_{\flat}) = \sum_{v_0 \in \Sigma_{F_0,\infty}} \Int_{T_\flat,v_0}^\bK(\xi, \tau_\infty, \phi_1, \phi_{\flat})+\sum_{v_0 \in \Sigma^{\fkd}_{F_0,\mathrm{inert}}} \Int_{T_\flat,v_0}(\xi,\phi_1, \phi_{\flat}),
\end{equation*}
where the non-archimedean contribution is still given by the arithmetic Euler--Poincar\'e characteristic, and the archimedean contribution is defined by evaluating Kudla's Green function $\tCG^{\bK}_\nu(\xi,y_\infty,\phi_1)$ on $\tDCZ_\nu(T_\flat,\phi_\flat)$ rather than the automorphic Green function, the latter is not holomorphic in general.

Next, consider the central value of the Eisenstein series at a nonsingular $T$ and its first derivative. By the archimedean Siegel--Weil formula (Theorem \ref{thm:archi Siegel Weil}(1)), only two classes of $T$ contribute nonzero terms: either $T$ is totally positive definite ($T \in \Herm_n^{+}(F)$), or $T$ has signature $(n-1,1)$ at some place and is positive definite at all other places ($T \in \Herm_n^{\dagger}(F)$).

Based on this, we define $\Diff(T, \BV)=\Diff(T)$ as follows (recall from \S \ref{sec:Eisenstein series} that $\BV$ is the incoherent $\BA_F/\BA_{F_0}$-hermitian space associated to $V$):
\begin{itemize}
    \item When $T \in \Herm_n^{+}(F)$, we define $\Diff(T, \BV)$ to be the set of finite places $v_0 \in \Sigma_{F_0,f}$ such that $V_{v_0}$ does not represent $T$. This set is nonempty since $V$ is not totally positive definite.
    \item When $T \in \Herm_n^{\dagger}(F)$, we define $\Diff(T, \BV)$ to be the set of finite places $v_0 \in \Sigma_{F_0,f}$ such that $V_{v_0}$ does not represent $T$, together with the set of archimedean places $v_0 \in \Sigma_{F_0,\infty}$ where the base change of $T$ to the nearby hermitian space $V_{v_0}$ is positive definite.
\end{itemize}

For a more general definition using representation theory, see \cite[\S 5]{Kudla97}. By the same proof as in \cite[Lemma 9.1]{KR14}, the derivative of the local Whittaker function at a finite place $v_0$ is nonzero only if the quadratic extension $F/F_0$ remains non-split at $v_0$.

\subsection{Local intersection: inert non-archimedean places}
Now let $v_0 \in \Sigma^{\fkd}_{F_0,\mathrm{inert}}$ be a finite place of $F_0$ that is inert in $F$, and let $v$ be the unique place of $F$ lying above $v_0$.

\begin{proposition}\label{prop:inert-local-int}
For any Schwartz function $\phi=\phi_1\otimes\phi_\flat\in\CS(V(\BA_{0,f})^n)^{K_G}$ satisfying \eqref{equ:assumption on schwartz function}, we have
\begin{align*}
\Int_{T_\flat,v_0}(\xi,\phi_1, \phi_{\flat})={}& \frac{\log q^2_{v_0}}{\vol(K_G^{v_0})} 
\sum_{\bu \in G^{(v_0)}(F_0) \backslash (V^{(v_0)}_\xi \times V^{(v_0)}_{T_{\flat}})(F_0)}
\Int_{v_0}(\bu, \phi_{v_0}) \cdot \int_{G(\BA^{v_0}_{0,f})} \phi^{v_0}(h^{-1} \bu) \, dh,\\
={} & \frac{\log q^2_{v_0}}{\vol(K_G^{v_0})} 
\sum_{\substack{T = {\begin{psmallmatrix} \xi & * \\ * & T_{\flat} \end{psmallmatrix}} \in \Herm_{n}^+(F) \\ \Diff(T,\BV) = \{v_0\}}}
\Int_{v_0}(T, \phi_{v_0}) \cdot \Orb^{v_0}_T(\phi^{v_0}).
\end{align*}
Here, $\Int_{v_0}(\bu, \phi_{v_0})$ is the weighted arithmetic intersection number defined in Definition \ref{def:local special cycle}. It is independent of the choice of representative elements $\bu$ in $V_T$, and we set $\Int_{v_0}(T, \phi_{v_0}) := \Int_{v_0}(\bu, \phi_{v_0})$ for some $\bu \in V_T$.
\end{proposition}
\begin{proof}
As in the proof of \cite[Proposition 2.22]{KR14}, the terms in \eqref{equ:derived intersection} contribute to the intersection number $\Int_{v_0}$ only when $T$ is positive definite, $\Diff(T) = \{v_0\}$, and $v_0 \nmid \fkd$. Moreover, by the argument in \cite[Lemma 2.21]{KR14}, the intersection in this case is supported in the basic locus. Consequently, we may reduce our computation to the formal completion along the basic locus using \S \ref{sec:p adic unitofmization of special cycle}. For any two $\bu_1,\bu_2\in V_T$ where $T$ is nonsingular, we can find $g\in \U(V_{v_0})$ such that $g\bu_1=\bu_2$, hence the induced automorphism on the RZ space induces an identification between $\Int_{v_0}(\bu_1,\phi_{v_0})$ and $\Int_{v_0}(\bu_2,\phi_{v_0})$.
Hence $\Int_{v_0}(T,\phi_{v_0})$ does not depend on the choice of $\bu\in V_T$\footnote{This is not related to the linear invariance of the special cycle.}.

Fix a place $\nu$ of $E$ above $v_0$. Then, by Proposition \ref{prop:KR-p-unif}, we can compute the intersection number fiberwise using the projection \eqref{equ:fiberwise projection} and then multiply the result by the factor $\tau(Z^{\BQ})$ defined in \eqref{equ:tau ZQ}. Therefore, we focus on the intersection over the fiber $\CM^{\wedge}_{O_{\breve{E}_\nu}}$ defined in \eqref{equ:RZ-comp}. 

By Proposition \ref{prop:KR-p-unif}, the restrictions of the special cycles $\tCZ(\xi, \phi_1)$ and $\tDCZ(T_{\flat}, \phi_{\flat})$ to $\CM^{\wedge}_{O_{\breve{E}_\nu}}$ take the forms
\begin{equation*}
    \sum_{(u_1, g) \in G^{(v_0)}(F_0) \backslash \left[V^{(v_0)}_\xi(F_0) \times G(\BA^{v_0}_{0,f})\right]/K_G^{v_0}}
    \phi_1^{v_0}(g^{-1} u_1) \cdot [\CZ(u_1, g, \phi_{1,v_0})]_{K^{v_0}_G},
\end{equation*}
and
\begin{equation*}\label{equ:derived summation}
    \sum_{(u_\flat, h) \in G^{(v_0)}(F_0) \backslash \left[V^{(v_0)}_{T_{\flat}}(F_0) \times G(\BA^{v_0}_{0,f})\right]/K_G^{v_0}}
    \phi_{\flat}^{v_0}(h^{-1} u_\flat) \cdot [\DCZ(u_\flat, h, \phi_{\flat,v_0})]_{K^{v_0}_G},
\end{equation*}
respectively. 
We compute the intersection number by pulling back to the covering formal scheme $\CN^{[h]}_{n,O_{\breve{E}_\nu}} \times G(\BA^{v_0}_{0,f})/K^{v_0}_G$ of the $p$-uniformization setting, as described in \S \ref{sec:p adic unitofmization of special cycle}.

The intersection number of 
$\tCZ(\xi, \phi_1) \cap^\BL \tDCZ(T_{\flat}, \phi_{\flat})$ (restricted to $\CM^{\wedge}_{O_{\breve{E}_\nu}}$) is given by the sum of the terms
\begin{equation*}
\phi_1^{v_0}(g^{-1} u_1) \phi_{\flat}^{v_0}(h^{-1} u_\flat) \cdot \chi(
\CZ(u_1, g, \phi_{1,v_0}) \cap^\BL \CZ^\BL(u_\flat, h, \phi_{\flat,v_0})),
\end{equation*}
taken over the $G^{(v_0)}(F_0)$-orbits (under the diagonal action) of tuples $(u_1, u_\flat, g, h)$ satisfying
\begin{equation*}
    (u_1, g) \in V^{(v_0)}_\xi(F_0) \times G(\BA^{v_0}_{0,f})/K^{v_0}_G
    \quad\text{and}\quad
    (u_\flat, h) \in V^{(v_0)}_{T_\flat}(F_0) \times G(\BA^{v_0}_{0,f})/K^{v_0}_G.
\end{equation*}
Here, $\chi$ denotes the Euler--Poincaré characteristic of the derived tensor product of the corresponding complexes in the $K$-group.
It is well defined by the same reasoning as in Proposition \ref{prop:derived-special-cycle-decomp}.

Let $\bu=(u_1,u_\flat)$. By \eqref{equ:KR-unif}, we obtain
\begin{equation}\label{equ:derived cup for single term}
\chi(\CZ(u_1,g,\phi_{1,v_0})\cap^{\BL}\DCZ(u_\flat,h,\phi_{\flat,v_0}))=\chi(\DCZ(\bu,\phi_{v_0})_{O_{\breve{E}_{\nu}}})\cdot \charfun_{K_G^{v_0}}(g^{-1}h).
\end{equation}
The first factor of \eqref{equ:derived cup for single term} is equal to
\begin{equation*}
    [E_\nu:F_{v}]\cdot\Int_{v_0}(\bu,\phi_{v_0}).
\end{equation*}
In particular, it is invariant under the (diagonal) action of $G^{(v_0)}(F_0)$ on the product $(V^{(v_0)}_\xi\times V^{(v_0)}_{T_{\flat}})(F_0)$.

The second factor of \eqref{equ:derived cup for single term} as a function 
$$(g,h)\in (G(\BA^{v_0}_{0,f})/K_G^{v_0})^2\mapsto \charfun_{K_G^{v_0}}(g^{-1}h)$$ 
is also invariant under the (diagonal) $G^{(v_0)}(F_0)$-action. For a fixed pair $\bu=(u_1,u_\flat)$, we obtain
\begin{equation*}\label{equ:relating to orbit integral}
\sum_{(g,h)\in (G(\BA^{v_0}_{0,f})/K_G^{v_0})}\phi^{v_0}(g^{-1}u_1,h^{-1}u_{\flat})\cdot \charfun_{K^{v_0}_G}(g^{-1}h)=\frac{1}{\vol(K_G^{v_0})}\int_{G(\BA^{v_0}_{0,f})}\phi^{v_0}(h^{-1}\bu)dh.
\end{equation*}
This proves the first line of the assertion. For the second line, we decompose
\begin{equation*}
    V_\xi^{(v_0)}\times V_{T_\flat}^{(v_0)}=\coprod_{ T=\begin{psmallmatrix}\xi& *\\ * &T_\flat\end{psmallmatrix}} V^{(v_0)}_T,
\end{equation*}
where the index ranges over all $T\in\Herm_n(F)$, and rewrite the summation as:
$$
\frac{\log q^2_{v_0}}{\vol(K^{v_0}_G)}\sum_{T={\begin{psmallmatrix}\xi&*\\ *&T_{\flat}\end{psmallmatrix}}}\sum_{\bu\in V^{(v_0)}_T(F_0)/G^{(v_0)}(F_0)}\Int_{v_0}(\bu,\phi_{v_0})\cdot \int_{G(\BA^{v_0}_{0,f})}\phi^{v_0}(h^{-1}\bu)dh.
$$
By the regularity assumption, if $T$ is singular, then for any $\bu\in V_T^{(v_0)}$, the local orbital integral $\int_{G(F_{0,w_0})}\phi_{w_0}(h^{-1}\bu)dh$ vanishes at a place $w_0\mid \fkd$. Hence the global orbital integral $\int_{G(\BA^{v_0}_{0,f})}\phi^{v_0}(h^{-1}\bu)dh$ vanishes for all $\bu\in V_T^{(v_0)}$ with singular $T$. By \cite[Lemma 2.7]{KR14}, we further conclude that $T\in\Herm_n^+(F)$, hence:
$$
\Int_{v_0}(\xi,\phi)=\frac{\log q^2_{v_0}}{\vol(K^{v_0}_G)}\sum_{\substack{T=\begin{psmallmatrix}\xi&*\\ *&T_{\flat}\end{psmallmatrix}\in\Herm^+_{n}(F)\\ \Diff(T,\BV)=\{v_0\}}}
\Int_{v_0}(T,\phi_{v_0})\cdot \Orb^{v_0}_T(\phi^{v_0}),
$$
where $\Int_{v_0}(T,\phi_{v_0})=\Int_{v_0}(\bu,\phi_{v_0})$ is well-defined, as discussed at the beginning of the proof.
\end{proof}

\subsection{Local intersection: archimedean places}
In this subsection, we compute the archimedean contribution to the local intersection numbers.
Let $v_0\in \Sigma_{F_0}$ be an archimedean place of $F_0$ and let $v$ be the unique place in $\Phi$ above $v_0$.

\begin{theorem}\label{thm:archi-int}
For any Schwartz function $\phi=\phi_1\otimes\phi_\flat\in\CS(V(\BA_{0,f})^n)^{K_G}$ satisfying \eqref{equ:assumption on schwartz function}, and $\xi\in F_0^\times$, we have
\begin{multline*}
\Int^{\bK}_{v_0}(\xi,\by_\infty,\phi)=\frac{\log q_{v_0}^2}{\vol(K_G)}\sum_{\substack{T={\begin{psmallmatrix}\xi&*\\ *&T_{\flat}\end{psmallmatrix}}\in\Herm_{n}^\dagger(F)\\ \Diff(T,\BV)=\{v_0\}}}\Int_{v_0}(T,\by_{v_0})\Orb_{T,f}(\phi)\\
-\int_{\CM_{v_0}}\CG_{v_0}^\bK(T_\flat,y_{\flat,v_0},\phi_\flat)\wedge \omega_{\KM,v_0}(\xi,y_{v_0},\phi_1).
\end{multline*}
Here $\btau_\infty=\begin{psmallmatrix}
	\tau & \\ & \tau_\flat
\end{psmallmatrix}=\bx_\infty+i\by_\infty\in \BH_n(F_{\infty})$, and the functions in the second line of equations are defined in Proposition \ref{prop:omega series}.
Also $\Int_{v_0}(T,\by_{v_0})=\Int_{v_0}(\bu,\by_{v_0})$ is the local height function on $\CD^{(v_0)}$ (Theorem \ref{thm:archi Siegel Weil}), which is independent of the choice of $\bu$ with $(\bu,\bu)=T$.
\end{theorem}
Note that the archimedean uniformization is a priori taken for a place $\nu\mid v\in \Phi$ of $E$, but the complex uniformization of the RSZ Shimura variety (\S \ref{sec:complex-uniformization}) and the Kudla--Rapoport divisors and their Green functions (\S \ref{sec:KR cycle in global}) depend only on the restriction $v_0$ of $\nu$ to $F_0$ (and on the fixed CM type $\Phi$).
Hence we index the functions in the main statement by $v_0$ for simplicity.
\begin{proof}
As in the computation in Proposition \ref{prop:inert-local-int}, we have
\begin{align*}
   \Int^\bK_{v_0}(\xi,\by_\infty,\phi)=\frac{1}{\vol(K_G)}\sum_{ \substack{T=\begin{psmallmatrix}\xi&*\\ * &T_\flat\end{psmallmatrix}\\ \in \Herm_n^{\reg}(F)}}\sum_{\bu\in V_T^{(v_0)}(F_0)/G^{(v_0)}(F_0)}\CG_{v_0}^{\bK}(u_1,\by_\infty)(\CZ^{(v_0)}(u_\flat))\cdot \int_{G(\BA_{0,f})}\phi(h^{-1}\bu)dh.
\end{align*}
While the summation initially ranges over all Hermitian matrices $T$ of the form ${\begin{psmallmatrix}\xi&*\\ * & T_\flat\end{psmallmatrix}}$,  the regularity assumption on $\phi$ ensures that only nonsingular terms contribute. Moreover, the set
 $V_T^{(v_0)}$ is non-empty if and only if $T$ represents the space $V^{(v_0)}$, hence we may restrict to $T\in\Herm_n^\dagger(F)$ satisfying $\Diff(T,\BV)=\{v_0\}$.

By Proposition \ref{prop:complex uniformization of special cycle} and our regularity assumption \eqref{equ:assumption on schwartz function} on $\phi$, the complex fibers $\GZ_{v_0}(u_1,\phi_1)$ and $\GZ_{v_0}(u_\flat,\phi_\flat)$ intersect properly (including the case where one of them is empty), see also \cite[Lemma 6.4]{Li-Liu-Annals}. We have
\begin{equation*}
    \CG_{v_0}^\bK(u_1,\by_\infty)(\CZ^{(v_0)}(u_\flat))=\int_{\CD^{(v_0)}}\CG_{v_0}^{\bK}(u_1,\by_{v_0})\wedge \delta_{\CZ^{(v_0)}(u_\flat)}.
\end{equation*}
Using the definition of the star product \eqref{equ:star-prod} and Green currents for special cycles \eqref{equ:green-current-cycle}, we have
\begin{multline}\label{equ:archimedean value}
\int_{\CD^{(v_0)}}\CG_{v_0}^{\bK}(u_1,\by_{v_0})\wedge \delta_{\CZ^{(v_0)}(u_\flat)}=\int_{\CD^{(v_0)}}\CG_{v_0}^{\bK}(u_1,y_{v_0})\star\CG_{v_0}^{\bK}(u_\flat,y_{\flat,v_0})\\
-\int_{\CD^{(v_0)}}\omega_{\KM,v_0}(u_1,y_{1,v_0})\wedge \CG_{v_0}^{\bK}(u_\flat,y_{\flat,v_0}).
\end{multline}

The summation of the first term in \eqref{equ:archimedean value} gives
\begin{equation*}
\frac{1}{\vol(K_G)}\sum_{ \substack{T=\begin{psmallmatrix}\xi&*\\ * &T_\flat\end{psmallmatrix}\\ \in \Herm_n^{\dagger}(F)}}\sum_{\bu\in V_T^{(v_0)}(F_0)/G^{(v_0)}(F_0)}\int_{\CD^{(v_0)}}\CG_{v_0}^{\bK}\left(\bu,{\begin{pmatrix}
	\tau_{v_0} & \\ & \tau_{\flat,v_0}
\end{pmatrix}}\right)\cdot \int_{G(\BA_{0,f})}\phi(h^{-1}\bu)dh.
\end{equation*}
Recall from the definition of the local height function (Theorem \ref{thm:archi Siegel Weil}) that $\int_{\CD^{(v_0)}}\CG^{\bK}(\bu,\by_{v_0})=
	\Int_{v_0}(T,\by_{v_0})$.
Taking the sum, we get
$$
\frac{1}{\vol(K_G)}\sum_{T={\begin{psmallmatrix}\xi&*\\ *&T_{\flat}\end{psmallmatrix}}\in\Herm_{n}^\dagger(F)}\Int_{v_0}(T,\by_{v_0})\Orb_{T,f}(\phi).
$$
Taking the summation over the second terms of \eqref{equ:archimedean value} and applying the definitions of $\CG^{\bK}(T,\phi)$ from \eqref{equ:KR_complex-unif} and of $\omega_{\KM}(T,y_{1,v_0},\phi)$ after \eqref{equ:KR_complex-unif}, we obtain the desired result.
\end{proof}

\section{Generating series and modularity}\label{sec:generating modularity}
In this section, we combine the results in \S \ref{sec:Eisenstein series}, \S \ref{sec:generating} and \S \ref{sec:local-int} to construct the generating series $\CE_{T_\flat}(\tau,\phi_1,\phi_\flat)$ introduced in \eqref{equ:KR proof generating series}. 
We will also prove that it defines a holomorphic modular form valued in $\CA_{\mathrm{hol}}(\bH(\BA_0),K_{\bH},n)\otimes_{\ov{\BQ}}\BR_{\fkd,\ov{\BQ}}$, where $\BR_{\fkd,\ov{\BQ}}:=\BR_{\fkd}\otimes_{\BQ}\ov{\BQ}$.
The proof is based on the modification method introduced in \cite{Mihatsch-Zhang} for hyperspecial level and generalized in \cite{ZZhang24} to general maximal parahoric level.

\subsection{Modified Fourier--Jacobi series}\label{sec:local-factors}
In this subsection, we modify the Fourier--Jacobi series introduced in \S \ref{sec:FJ series} to incorporate error terms. We then compare the local factors with the generating series on the geometric side.

Fix basis elements $\alpha \in \wedge^{2n^2}(V^n)^*$ and $\beta \in \wedge^{n^2}(\Herm_n)^*$. 
Recall from \eqref{not:gauge form} that there exists a form $\omega$ of degree $n^2$ on $V^n_{\reg}$ such that its pullback, $\omega_{\U(V)} = i_\bu^*\omega$, defines a gauge form on $\U(V)$, which is independent of the choice of $\bu$. This form determines a Tamagawa measure on $\U(V)(\BA_0)$ via a choice of convergence factors $\lambda_{v_0}$. 
Furthermore, recall from \eqref{not:gauge form 2} and \eqref{not:gauge form 3} that we have the constants defined by
\begin{equation*}
    d_{\alpha,v_0}x = c_{v_0}(\alpha, \psi) dx, \quad 
    d_{\beta,v_0}b = c_{v_0}(\beta, \psi) db,
\end{equation*}
which compare the measures $d_{\alpha,v_0}x$ and $d_{\beta,v_0}b$ with the self-dual Haar measures.

Let $\phi = \phi_\fkd\otimes \phi^{\fkd}\in\CS(V(\BA_{0,f})^n)^{K_G}$ be an admissible Schwartz function of weight $\bw$ satisfying \eqref{equ:assumption on schwartz function}, and let $\Sigma^{\fkd}_{\max}$ be the finite set of finite places $v_0 \nmid \fkd$ of $F_0$ where $L_{v_0}$ is not self-dual.

Recall from \eqref{not:nearby herm} that $V^{(v_0)}$ denotes the $v_0$-nearby hermitian space of $V$. Define
\begin{equation*}
    \phi_{\corr} := \sum_{v_0 \in \Sigma^{\fkd}_{\max}} \phi_{\corr}(v_0) \cdot \log q_{v_0}^2
    \in \bigoplus_{v_0 \in \Sigma^{\fkd}_{\max}} \CS(V^{(v_0)}(\BA_{0,f})^n) \cdot \log q_{v_0}^2,
\end{equation*}
where $\phi_{\corr}(v_0) = \phi^{v_0} \otimes \phi_{v_0,\corr}(v_0)$ is a $\BQ$-valued Schwartz function in $\CS(V^{(v_0)}(\BA_{0,f})^n)$ such that
\begin{equation}\label{equ:corr term def}
    \phi_{v_0,\corr}(v_0) = -\frac{\gamma(V^{(v_0)}_{v_0})^n}{\gamma(V_{v_0})^n}
    \sum_{\substack{0\leq t< h\text{ and}\\ t\equiv h-1\bmod{2}}} \beta_{\phi}^{[h_{v_0}]}(t) \phi_{\bw,v_0}^{[t]} 
    \in \CS(V^{(v_0)}(F_{0,v_0})^n),
\end{equation}
where $h_{v_0}$ is the type of the vertex lattice $L$ at the place $v_0$, and where $\phi_{\bw,v_0}^{[t]}$ are admissible Schwartz functions at $v_0$ of type $t$ that are matched with $\phi_{v_0}$.

For each $v_0$, the tensor product $\phi_{\corr}(v_0)\otimes\phi_{\infty}$ defines a Schwartz function on the coherent hermitian space $V^{(v_0)}(\BA_0)$, and gives rise to the (coherent) Eisenstein series (recall Notation \ref{not:eisenstein} for our convention):
\begin{equation*}
    \Eis(\btau, \phi_{\corr}) = \sum_{v_0 \in \Sigma_{\max}} \log q_{v_0}^2 \Eis(\btau, \phi_{\corr}(v_0)).
\end{equation*}
Let $C_{T}(\btau, \phi_{\corr}(v_0))$ denote the $T$-th Fourier coefficient of $\Eis(\btau, \phi_{\corr}(v_0))$.

\begin{lemma}
We have
$C_{T}(\by_\infty, \phi_{\corr}(v_0)) = 0$ unless $\Diff(T, \BV) = \{v_0\}$.
\end{lemma}

\begin{proof}
By construction, $\phi_{\corr}(v_0)$ is a Schwartz function on $V^{(v_0)}(\BA_{0,f})^n$.  The space $V^{(v_0)}$ is positive definite at the archimedean places and isomorphic to $V$ at all finite places except $v_0$. 
In particular, a matrix $T$ represents this space if and only if it represents $\BV^{v_0}$ but does not represent $\BV_{v_0}$.
\end{proof}

Define the \emph{modified Eisenstein series} as
\begin{equation}\label{equ:modify eisenstein series}
    \PEis^{\modi}(\btau, \phi) := \PEis(\btau, \phi) + \Eis(\btau, \phi_{\corr}).
\end{equation}
We denote its $T$-th Fourier coefficient by
\begin{equation*}
    \partial C^{\modi}_{T}(\by_\infty, \phi) = \partial C_{T}(\by_\infty, \phi) + C_{T}(\by_\infty, \phi_{\corr}).
\end{equation*}

\begin{theorem}\label{thm:local-const}
There exists a constant $c_V$, depending only on $V$ (see \eqref{equ:ASWF-constant}), such that for any nonsingular $T \in \Herm_n(F)$ with $\Diff(T, \BV) = \{v_0\}$, we have
\begin{equation*}
    c_V \frac{\PC^{\modi}_{T}(\by_\infty, \phi)}{\log q_{v_0}^2} =
    \begin{cases}
        \vol(K_{G,v_0}) \PDen_{v_0}(T, \phi_{v_0}) \cdot \Orb_{T}(\phi^{v_0}), & v_0 \in \Sigma_f, \\
        \Int_{v_0}(T, \by_{v_0}) \Orb_{T,f}(\phi), & v_0 \in \Sigma_\infty.
    \end{cases}
\end{equation*}
Here, $q_{v_0}$ denotes the cardinality of the residue field of $O_{F_0}$ at $v_0$ when $v_0$ is non-archimedean and equals $e$ when $v_0$ is a real embedding. 
Furthermore, $\PC^{\modi}_T = 0$ for all remaining values of $T$.
\end{theorem}
\begin{proof}
The last assertion follows from assumption \eqref{equ:assumption on schwartz function}. 
For any nonsingular $T \in \Herm_n^{\reg}(F)$ satisfying $\Diff(T, \BV) = \{v_0\}$ with $v_0 \in \Sigma^{\fkd}_{F_0,\infty}$, by the non-archimedean Siegel--Weil formula (Theorem \ref{thm:nonarchi local SW}) and the archimedean (arithmetic) Siegel--Weil formula (Theorem \ref{thm:archi Siegel Weil}), we have
\begin{align*}
\PEis_T(\btau_{v_0}, \phi) ={}& W_{T_{v_0}}'(\btau_{w_0},0,\phi_{v_0}) 
\prod_{w_0 \in \Sigma_\infty \setminus \{v_0\}} W_{T_{w_0}}(\btau,0,\phi_{w_0}) 
\prod_{w_0 \in \Sigma_f} W_{T_{w_0}}(1,0,\phi_{w_0}),\\
={}& \Int(T_{v_0}, \by_{v_0}) \gamma_{\BV_{v_0}} \frac{(2\pi)^{n^2}}{\Gamma_n(n)} q^{T_{v_0}} 
\prod_{w_0 \in \Sigma_\infty \setminus \{v_0\}} \gamma_{\BV_{w_0}} \frac{(2\pi)^{n^2}}{\Gamma_n(n)} q^{T_{w_0}} 
\prod_{w_0 \in \Sigma_f} C(V, \alpha, \beta, \psi)_{v_0} \Orb_{T_{w_0}}(\phi_{w_0}).
\end{align*}
Define the constant 
\begin{equation}\label{equ:ASWF-constant}
    c_V = \frac{C(V, \alpha, \beta, \psi)_\infty}{\gamma_{\BV_\infty} \left( \frac{(2\pi)^{n^2}}{\Gamma_n(n)} \right)^{[F_0:\BQ]}}.
\end{equation}
Then we have
\begin{equation*}
    c_V \partial \Eis_T(\tau, \phi) q^{-T} = \Int_{v_0}(T, \tau) \Orb_{T,f}(\phi).
\end{equation*}

Now suppose $\Diff(T, \BV) = \{v_0\}$ for $v_0 \in \Sigma^{\fkd}_{F_0,f}$.
If $v_0 \notin \Sigma^{\fkd}_{\max}$, by the local Siegel--Weil formula for both non-archimedean and archimedean places, we obtain
\begin{align*}
\PEis_T(\btau, \phi) ={}& 
\prod_{w_0 \in \Sigma_{F_0,\infty}} W_{T_{w_0}}(\btau_{w_0},0,\phi_{w_0}) 
W'_{T_{v_0}}(1,0,\phi_{v_0}) 
\prod_{w_0 \in \Sigma_{F_0,f} \setminus \{v_0\}} W_{T_{w_0}}(1,0,\phi_{w_0}),\\
={}& \gamma_{\BV_\infty} \left( \frac{(2\pi)^{n^2}}{\Gamma_n(n)} \right)^{[F_0:\BQ]} q^T 
W'_{T_{v_0}}(1,0,\phi_{v_0}) 
\prod_{w_0 \in \Sigma_{F_0,f} \setminus \{v_0\}} C(V, \alpha, \beta, \psi)_{w_0} \Orb_{T_{w_0}}(\phi_{w_0}),\\
={}& c_V^{-1} \frac{W'_{T_{v_0}}(1,0,\phi_{v_0})}{C_{v_0}(V, \alpha, \beta, \psi)_{v_0}} \Orb_T(\phi^{v_0}).
\end{align*}

By Theorem \ref{thm:nonarchi local SW}, Proposition \ref{prop:witt-den}, and the definition of $\PDen(T, \phi^{[h]})$ in Theorem \ref{thm:KR-conj}, we obtain
\begin{equation*}
    \frac{W_T'(1,0,\phi_{v_0})}{C_{v_0}(V, \alpha, \beta, \psi)} = 
    \frac{W_T'(1,0,\phi_{v_0})}{\gamma_{V_{v_0}}^n} 
    \cdot \frac{c_{v_0}(\beta, \psi)}{c_{v_0}(\alpha, \psi)} 
    = \partial\alpha_T(\phi_{v_0}) \cdot \frac{c_{v_0}(\beta, \psi)}{c_{v_0}(\alpha, \psi)} \vol(K_{v_0}, dg_{v_0}^{\rm{self}}),
\end{equation*}
where $dg_{v_0}^{\rm{self}}$ is the Haar measure induced by self-dual Haar measures at the place $v_0$ (see \S \ref{sec:def weighted local density}). Changing to the Tamagawa measure, we conclude that
\begin{equation*}
    c_V \PEis_T(\tau, \phi) q^{-T} = \vol(K_{G,v_0}) \PDen_{v_0}(T, \phi_{v_0}) \Orb_T(\phi^{v_0}).
\end{equation*}

On the other hand, if $v_0 \in \Sigma^{\fkd}_{\max}$, we have
\begin{align*}
\Eis(\btau, \phi_{\corr}(v_0))={}&c_V^{-1} \frac{W_{T_{v_0}}(1,0,\phi_{v_0,\corr}(v_0) )}{C_{v_0}(V^{(v_0)}, \alpha, \beta, \psi)_{v_0}} \Orb_T(\phi^{v_0}).
\end{align*}
Furthermore,
$$
\frac{W_{T_{v_0}}(1,0,\phi_{v_0,\corr}(v_0) )}{C_{v_0}(V^{(v_0)}, \alpha, \beta, \psi)_{v_0}}=\frac{W_{T,v_0}(1,0,\phi_{v_0,\corr}(v_0))}{\gamma^n_{V_{v_0}^{(v_0)}}}\cdot \frac{c_{v_0}(\beta,\psi)}{c_{v_0}(\alpha,\psi)}.
$$
This contributes the error terms of $\PDen$ in Theorem \ref{thm:KR-conj}, by the definition of the correction term in \eqref{equ:corr term def}.
\end{proof}

\subsection{Geometric modularity of the difference Green function}
We recall the geometric modularity of the difference Green function, proved in \cite{Stephan-Siddarth} for $F_0 = \BQ$ and generalized in \cite{Mihatsch-Zhang} to totally real $F_0$.

Let $\nu\mid v\in\Phi$ be any place of $E$. Consider the difference of two Green functions:
\begin{equation*}
    \tCG_\nu^{\bf K-\bf B}(\xi, y_\infty, \phi) = 
    \begin{cases}
        \tCG_\nu^{\bf K}(\xi,y_\infty,\phi) - \tCG_\nu^{\bf B}(\xi,\phi), & \text{for } \xi \in F_{0}^+, \\
        \tCG_\nu^{\bf K}(\xi,y_\infty,\phi), & \text{otherwise.}
    \end{cases}
\end{equation*}

We define the generating series of Green functions as the archimedean part of the generating series of arithmetic special divisors ($h_f \in \bH(\BA_{0,f})$):
\begin{equation*}\label{eq: generating series Green}
    \tCG_\nu^{?}(\tau_\infty,h_f, \phi) := \sum_{\xi \in F_0} \tCG_\nu^{?}(\xi, y_\infty, \omega(h_f)\phi) q^\xi, 
    \quad ? = \bf K, \bf B, \bf K - \bf B. 
\end{equation*}

\begin{theorem}[\protect{\cite[Theorem 3.13]{Mihatsch-Zhang}}]\label{thm: Green}
Suppose that $\phi_1 \in \CS(V(\BA_{0,f}))^{K_G}$ is invariant under $K_{\bH} \subset \bH(\BA_{0,f})$ via the Weil representation. Then the generating series $\tCG_\nu^{\bf K-\bf B}(\tau_\infty,h_f, \phi_1)$, when paired with a degree-zero $1$-cycle on $\tCM_{\nu}$, lies in $\CA_{\infty}(\bH(\BA_0),K_{\bH},n)$.
\end{theorem}

\subsection{Auxiliary cycles and modularity}\label{sec:auxiliary}
We now construct the auxiliary cycles and prove the modularity of the generating series in our setup. 
As in \cite{ZZhang24}, the main idea is to modify the given $1$-cycle $\CZ\in\wt{\CZ}_1(\tCM)$ by some auxiliary $1$-cycles $\CC^{\mathrm{aux}}\in\wt{\CZ}_1(\tCM)$ whose modularity is already known.
This modification ensures that the resulting cycle $\CZ + \CC^{\mathrm{aux}}$ lies in $\wt{\CZ}_1(\tCM)^\perp$, allowing us to apply the admissible intersection pairing defined in \eqref{equ:almost intersection}.

We choose certain Kudla--Rapoport special cycles as auxiliary cycles at the archimedean places.
We fix an element $\varpi_0\in F_0^\times$ such that:
\begin{itemize}
\item It is positive at archimedean places.
\item For any $v_0\in \Sigma_{\mathrm{max}}^{\fkd}$, $\varpi\in F_{0,v_0}$ is a uniformizer. Recall that $\Sigma^{\fkd}_{\max}$ is the set of finite places $v_0 \nmid \fkd$ of $F_0$ where $L_{v_0}$ is not self-dual
\item For any $v_0 \in \Sigma^{\fkd}_{\mathrm{inert}} \setminus \Sigma^{\fkd}_{\mathrm{max}}$, $\varpi\in F_{0,v_0}$ is a unit.
\item There is no requirement on the split places over $F_0$.
\end{itemize}
The existence of such an element follows from class field theory, see \cite[footnote 4]{LZ21}.
Let $\bw_{\triv} = (\CZ, \dots, \CZ)$ and let 
$$
T_{\mathrm{aux}}:=
\left\{\begin{array}{cc}
\mathrm{diag}(1^{(n-h)},\varpi^{(h-1)}_0)&h\geq 1;\\
1_{n-1}&h=0.
\end{array}\right.
$$
Let $\phi_{\triv} = \phi_{\mathrm{aux},\fkd} \otimes \phi_{\mathrm{aux}}^{\fkd} \in \CS(V(\BA_{0,f})^{n-1})^K$ be such that:
\begin{itemize}
\item The complex fiber $\DCZ_\nu(T_{\triv}, \phi_{\triv})$ has nonzero degree for one archimedean place $\nu$ of $E$.
\item The prime-to-$\fkd$ part $\phi^{\fkd}_{\mathrm{aux}}$ is admissible of weight $\bw_{\mathrm{aux}}$.
\item The function $\phi_{\mathrm{aux}}$ is regular at two split places $v_0 \mid \fkd$.
\end{itemize}
For such $T_{\triv}$ and $\phi_{\mathrm{aux}}$, by Proposition \ref{lem:cancellation law}, we see that for any $v_0\in \Sigma_{\mathrm{inert}}^{\mathfrak{d}}$, the local intersection number $\Int_{v_0}(T_{\triv},\phi_{v_0})$ equals the analytic side $\PDen_{v_0}(T_{\triv},\phi_{v_0})$.

\begin{remark}
In \cite[\S 13]{ZZhang24} and \cite[\S 8]{LMZ25}, the auxiliary cycle at the archimedean place is chosen as a Hecke CM cycle $\mathrm{CM}(\alpha_0, \mu_0)$, where $\alpha_0$ is a maximal order at any $v_0 \in \Sigma^{\fkd}_{\mathrm{inert}}$. 
To ensure the existence of such a cycle, one must enlarge the truncation $\fkd$, as discussed in \cite[Theorem 13.5]{ZZhang24}. On the other hand, the auxiliary cycles we construct here do not require enlarging $\fkd$. This will help improve the almost modularity theorem in \cite{ZZhang24}, see Remark \ref{rmk:almost modularity}.
\end{remark}

Next, we recall the non-archimedean auxiliary cycles from \cite{ZZhang24}.

\begin{theorem}\label{thm:local modification}
For any $1$-cycle $\CZ \in \wt{\CZ}_1(\tCM)$ and any place $\nu\mid v_0 \in \Sigma^{\fkd}_{\mathrm{max}}$ of $E$, there exist a $1$-cycle $\CC_{\nu}^{\mathrm{aux}}$, supported in the basic locus of $\tCM \otimes \BF_\nu$, and a Schwartz function $\phi^{\mathrm{aux}}(\nu) \in \CS(V^{(v_0)}(\BA_{0}))$ such that:
\begin{enumerate}
    \item The pairing of $\CZ - \CC_{\nu}^{\mathrm{aux}}$ with each vertical component of every geometric mod-$\nu$ fiber $\tCM \otimes \BF_\nu$ of $\tCM$ is zero.
    \item The arithmetic intersection pairing 
    \begin{equation*}
        (\tCZ^{\bB}(\tau, h_f, \phi_1), \CC_{\nu}^{\mathrm{aux}}) \in \BR_{\fkd}[\![q]\!]
    \end{equation*}
    is well-defined and is a weight-$n$ holomorphic Hilbert modular form given by 
    $$
    \sum_{u\in V^{(v_0)}(F_0)}\phi^{\mathrm{aux}}(\nu)(u)q^{(u,u)}
    $$
    up to a constant term (for which we use the same notation).
\end{enumerate}
\end{theorem}
\begin{proof}
The modification follows the same argument as in \cite[\S 12.4]{ZZhang24}, using the \emph{very special $1$-cycles} defined in \cite[Definition 12.14]{ZZhang24}. A detailed and corrected exposition is given in \cite[\S 8.3, \textbf{Step 3}]{LMZ25}, we omit the proof.
\end{proof}

\subsection{Modular forms}\label{sec:modular form}
In this subsection, we construct the auxiliary generating series 
\begin{equation*}
    \CE_{T_\flat}(\tau, \phi_1, \phi_\flat) \in \CA_{\hol}(\bH(\BA_0), K_{\bH}, n)_{\ov{\BQ}} \otimes_{\ov{\BQ}} \BR_{\fkd, \ov{\BQ}}.
\end{equation*}

Recall our global notation in \S \ref{sec:gp data}.
Let $F/F_0$ be a global CM extension with a fixed CM type $\Phi$ and a distinguished element $\varphi_0 \in \Phi$. Let $(V, h)$ be a hermitian space of dimension $n$ over $F$ such that the signature of $V$ is $(n-1,1)$ at $\varphi_0$ and $(n,0)$ otherwise. 
In this case, the associated Shimura datum $(\Res_{F_0/\BQ}G,\{h_{\Res_{F_0/\BQ}G}\})$ is of \emph{strict fake Drinfeld type} (see \cite[Example 2.3]{RSZ}). By \cite[Corollary 12.6]{ZZhang24}, the associated canonical model $\mathrm{Sh}_K(\Res_{F_0/\BQ}G,\{h_{\Res_{F_0/\BQ}G}\})$ is connected.
For each $\nu\in \Phi$, the complex fiber $\CM_{\nu}$ defined in \S \ref{sec:complex-uniformization} is identical to the complex fiber of $\mathrm{Sh}_K(\Res_{F_0/\BQ}G,\{h_{\Res_{F_0/\BQ}G}\})$. Fix a distinguished prime $p \in \BZ$.

Fix an ideal $\fkd \subset O_{F}$ divisible by all primes above $2$ (to be enlarged later). The data specified so far can be completed to define a Shimura datum. We find a lattice $L$ and enlarge $\fkd$ so that $L_{v_0}$ is a vertex lattices for any finite place $v_0 \nmid \fkd \cup \{p\}$.
We require that the level structure satisfies the assumptions in \S \ref{sec:RSZ-integral}.

Let $T_\flat \in \Herm^+_{n-1}(F)$ be a totally positive definite hermitian matrix of rank $n-1$.
Let $\nu:E\to \BC$ be any complex embedding with restriction $v\in\Phi$ to $F$ and $v_0$ to $F_0$.
By the complex uniformization of the special cycles (\S \ref{sec:special cycle complex uniformization}), and Hecke actions, we see that the degree (over $\BC$) of $\tDCZ_\nu(T_\flat, \phi_\flat)$ (resp. $\tDCZ_\nu(T_\triv, \phi_\triv)$) is identical on each geometric connected component $\tCM_{\nu}$ of the projection \eqref{equ:complex projection} since $\CM_{\nu}$ is connected. 

Since $\DCZ_\nu(T_\triv, \phi_\triv)$ has nonzero degree on $\CM_{\nu}$, we may choose for a constant $c_\infty$ such that the degree function of
\begin{equation*}
    \tDCZ(T_\flat, \phi_\flat)_{\nu} + c_\infty \tDCZ(T_\triv, \phi_\triv)_{\nu}
\end{equation*}
is identically zero along each fiber of the projection $\tCM_{\nu} \to Z^{\BQ}(\BQ) \backslash Z^{\BQ}(\BA_{\BQ,f})/K_{Z^\BQ}$. By the uniformization, it follows that $ \tDCZ(T_\flat, \phi_\flat)_\nu + c_\infty \tDCZ(T_\triv, \phi_\triv)_\nu$ has zero degree for all $\nu\in\Phi$.

By Theorem \ref{thm:local modification}, for each finite place $\nu$ of $E$ above some $v_0\in\Sigma^{\fkd}_{\max}$ we can find a modification $\CC_\nu^{\mathrm{aux}}$.

To summarize, we define a suitable modification:
\begin{equation}\label{equ:aux special cycle}
    \tDCZ(T_\flat, \phi_\flat)^{\mathrm{aux}} = \tDCZ(T_\flat, \phi_\flat) + c_\infty \tDCZ(T_{\triv}, \phi_{\triv}) + \sum_{\nu\mid v_0 \in \Sigma^{\fkd}_{\mathrm{max}}} c_{\nu} \CC_{\nu}^{\mathrm{aux}},
\end{equation}
which satisfies the following properties:
\begin{enumerate}
    \item (Over $\BC$) It has degree $0$ on each connected component of every complex fiber $\tCM_{\nu}$.
    \item (Over $\BF_{\nu}$) Its pairing with each vertical component of every geometric mod-$\nu$ fiber $\tCM \otimes \BF_\nu$ of $\tCM$ is zero for all finite places $\nu \nmid \fkd$ of $E$.
\end{enumerate}
Therefore, $\tDCZ(T_\flat, \phi_\flat)^{\mathrm{aux}} \in \wt{\CZ}_1(\tCM)^{\perp}$, and we may apply \eqref{equ:almost intersection}. By Proposition \ref{prop:modularity-bruinier}, the arithmetic intersection pairing 
\begin{equation*}
    (\tCZ^{\bB}(\tau, h_f, \phi_1), \tDCZ(T_\flat, \phi_\flat)^{\mathrm{aux}}) \in \BR_{\fkd}[\![q]\!]
\end{equation*}
is well-defined and is a weight-$n$ holomorphic Hilbert modular form up to a constant term (for which we use the same notation). 

\begin{remark}\label{rmk:almost modularity}
The modification procedure applies to any $1$-cycle $\CZ \in \CZ_1(\tCM)$ satisfying the following condition: for each $v \in \Phi$, the degree of the complex fiber $\CZ_{\nu}$ is the same for all places $\nu$ of $E$ above $v$.

This condition holds for Kudla--Rapoport special cycles and CM cycles. 
Let $\CS_K(\Res_{F_0/\BQ_p}(\U(V)), \mu)$ be the integral model of the abelian-type unitary Shimura variety, we have a natural projection $p: \tCM \to \CS_K(\Res_{F_0/\BQ_p}(\U(V)), \mu)$. 
If $\CZ$ is any $1$-cycle on $\CS_K(\Res_{F_0/\BQ_p}(\U(V)), \mu)$, then the pull-back $p^*\CZ$ also satisfies this condition.
\end{remark}

Let $\btau = \begin{psmallmatrix} \tau & \\ & \bi \end{psmallmatrix} \in \BH_n(F_\infty)$. Recall from \S \ref{sec:FJ series} the Fourier--Jacobi coefficients of the Eisenstein series and its derivative:
\begin{equation*}
    \PFJ_{T_\flat}(\tau, \phi) = \sum_{\xi \in F_0} \sum_{T} \PC_T(\btau, \phi) q^\xi, 
    \quad
    \FJ_{T_\flat}(\tau, \phi_{\corr}) = \sum_{\xi \in F_0} \sum_{T} C_T(\btau, \phi_{\corr}) q^\xi.
\end{equation*}
Define the \emph{modified Fourier--Jacobi series} as
\begin{equation*}
    \PFJ^{\modi}_{T_\flat}(\tau, \phi) := \PFJ_{T_\flat}(\tau, \phi) + \FJ_{T_\flat}(\tau, \phi_{\corr}).
\end{equation*}
This is the Fourier--Jacobi series of the modified Eisenstein series introduced in \eqref{equ:modify eisenstein series}.

Consider also (where $c_V$ is the constant defined in \eqref{equ:ASWF-constant})
\begin{multline*}
\Theta_{T_\flat}^{\mathrm{aux}}(\tau):=c_\infty\Bigl(c_V \PFJ^{\modi}_{T_{\triv}}(\tau, \phi_1\otimes \phi_{\triv})-\vol(K_G)\sum_{v_0 \in \Sigma_{F_0, \infty}} \omega_{T_\triv, v_0}(\tau, \phi_1, \phi_\triv)\Bigr)\\
+\frac{\vol(K_G)}{\tau(Z^{\BQ})[E:F]}\sum_{\nu\mid v_0 \in \Sigma^{\fkd}_{\mathrm{max}}}c_\nu\,\bigl(\tCZ^{\bB}(\tau, \phi_1), \CC_{\nu}^{\mathrm{aux}}\bigr),
\end{multline*}
where $\omega_{T_{\mathrm{aux}}}(\tau,\phi_1,\phi_{\mathrm{aux}})$ is defined in Proposition \ref{prop:omega series}.
This expression is related to the intersection numbers of the auxiliary cycles defined on the geometric side in \S \ref{sec:auxiliary}.

Now, consider the generating series (recall the index convention in Theorem \ref{thm:archi-int}, and note that the modifications $\CC_\nu^{\mathrm{aux}}$ in special fibers have empty complex fiber):
\begin{multline*}
J_{T_\flat}(\tau):=c_V\PFJ^{\modi}_{T_\flat}(\tau, \phi)-\vol(K_G)\sum_{v_0 \in \Sigma_{F_0, \infty}} \Bigl[ \CG_{v_0}^{\bK-\bB}(\tau, \phi_1)\bigl(\DGZ_{v_0}(T_\flat, \phi_\flat)^{\mathrm{aux}}\bigr)
    + \omega_{T_\flat, v_0}(\tau, \phi_1, \phi_\flat) \Bigr]\\
+ \Theta_{T_\flat}^{\mathrm{aux}}(\tau).
\end{multline*}
By Theorem \ref{thm: Green}, Proposition \ref{prop:Fourier--Jacobi series are modular forms}, and Proposition \ref{prop:omega series}, we conclude that
\begin{equation*}
    J_{T_\flat}(\tau) \in \CA_{\infty}(\bH(\BA_0), K_{\bH}, n).
\end{equation*}

\begin{proposition}\label{prop:holo proj}
We have
$J_{T_\flat}(\tau) \in \CA_{\hol}(\bH(\BA_0), K_{\bH}, n)$.
\end{proposition}

\begin{proof}
For an automorphic function $F$, the holomorphicity condition is characterized by the vanishing of the archimedean lowering operator at each point $h \in \bH(\BA_0)$. Since the differential operator commutes with the right action of $\bH(\BA_{0,f})$, by the strong approximation theorem for the special linear group, it suffices to verify holomorphicity when $h_f = 1$. This reduces our analysis to classical Hilbert modular forms.

We now show that $J_{T_\flat}(\tau)$ admits a Fourier expansion of the form (noting that $F_0 \neq \BQ$):
\begin{equation}\label{equ:holomorphic expansion}
    J_{T_\flat}(\tau) = \sum_{\xi \in F_0, \xi \geq 0} A_\xi q^\xi, \quad A_\xi \in \BC.
\end{equation}
By definition, the $\xi$-th Fourier coefficient of $J_{T_\flat}$, for $\xi\in F_0^\times$, is
\begin{multline*}
\vol(K_G)\sum_{v_0 \in \Sigma_{F_0, \infty}}\CG_{v_0}^{\bB}(\xi, \phi_1)\bigl(\DGZ_{v_0}(T_\flat, \phi_\flat) + c_\infty \DGZ_{v_0}(T_\triv, \phi_\triv)\bigr)\\
+\sum_{T}c_V\,\PC^{\modi}_T(\by_\infty,\phi)+c_\infty\sum_{T'}c_V\,\PC^{\modi}_{T'}(\by_\infty,\phi_1\otimes\phi_{\triv})\\
+\frac{\vol(K_G)}{\tau(Z^{\BQ})[E:F]}\sum_{\nu}c_\nu\,\bigl(\tCZ^{\bB}(\xi, \phi_1), \CC_{\nu}^{\mathrm{aux}}\bigr),
\end{multline*}
where $T$ (resp.\ $T'$) runs over the matrices $\begin{psmallmatrix}\xi&*\\ *&T_\flat\end{psmallmatrix}$ (resp.\ $\begin{psmallmatrix}\xi&*\\ *&T_\triv\end{psmallmatrix}$) in $\Herm_n^+(F)$; all four terms are zero unless $\xi$ is totally positive. This expression does not depend on $y$.
\end{proof}

Moreover, the proof of Proposition \ref{prop:holo proj} yields the following conclusion:

\begin{corollary}\label{cor:coefficient of J function}
For any $\xi > 0$, we have
\begin{multline*}
    J_{T_\flat,\xi}(\tau)= \sum_{v_0 \in \Sigma_{F_0, \infty}} 
    \CG_{v_0}^{\bB}(\xi, \phi_1)(\DGZ_{v_0}(T_\flat, \phi_\flat) + c_\infty \DGZ_{v_0}(T_\triv, \phi_\triv)) \\
    + \sum_{v_0 \in \Sigma_{F_0, f}} \log q_{v_0} \vol(K_{G,v_0})
    \sum_{\substack{T = \begin{psmallmatrix} \xi & * \\ * & T_\flat \end{psmallmatrix} \in \Herm_n^+(F);\\ \Diff(T) = \{v_0\}}}
    \PDen_{v_0}(T, \phi_{v_0}) \cdot \Orb_{T}(\phi^{v_0}) \\
    + \sum_{v_0 \in \Sigma_{F_0, f}} \log q_{v_0} \vol(K_{G,v_0})
    \sum_{\substack{T = \begin{psmallmatrix} \xi & * \\ * & T_\triv \end{psmallmatrix} \in \Herm_n^+(F);\\ \Diff(T) = \{v_0\}}}
    \PDen_{v_0}(T, \phi_{v_0}) \cdot \Orb_{T}(\phi^{v_0}).
\end{multline*}
\end{corollary}
\begin{proof}
The coefficient is the expression stated above plus the sum
$$
c_\infty\sum_{w_0 \in \Sigma_{F_0, f}^{\fkd}}\vol(K_{G, w_0}) \log q^2_{w_0}\sum_{T'}\bigl( \Int_{w_0}(T', \phi_{w_0}) - \PDen_{w_0}(T', \phi_{w_0}) \bigr) \Orb_{T'}(\phi^{w_0}),
$$
which vanishes by Lemma \ref{lem:cancel}.
\end{proof}

We may project $J_{T_\flat}(\tau)$ from $\CA_{\hol}(\bH(\BA_0), K_{\bH}, n)$ to 
\begin{equation*}
    \CA_{\hol}(\bH(\BA_0), K_{\bH}, n)_{\ov{\BQ}} \otimes_{\ov{\BQ}} \BR_{\fkd, \ov{\BQ}},
\end{equation*}
and we still denote its image by $J_{T_\flat}(\tau)$. We define the function
\begin{equation}\label{equ:define generat series}
\CE_{T_\flat}(\tau, \phi_1, \phi_\flat) := \frac{\vol(K_G)}{\tau(Z^{\BQ})[E:F]} \bigl( \wh{\GZ}^\bB(\tau, \phi_1), \DGZ(T_\flat, \phi_\flat)^{\mathrm{aux}} \bigr) -J_{T_\flat}(\tau) \in \CA_{\hol}(\bH(\BA_0), K_{\bH}, n)_{\ov{\BQ}} \otimes_{\ov{\BQ}} \BR_{\fkd, \ov{\BQ}}.
\end{equation}

\begin{proposition}\label{prop:KR-gent}
Let $\xi \in F_0^\times$. The Fourier coefficients of $\CE_{T_\flat}(\tau, \phi_1, \phi_\flat)$ are given by:
\begin{multline*}
    \CE_{T_\flat, \xi}(\tau, \phi_1, \phi_\flat) := \sum_{w_0 \in \Sigma_{F_0, f}^{\fkd}} \CE_{\xi, w_0} := \\
    \sum_{w_0 \in \Sigma_{F_0, f}^{\fkd}} \vol(K_{G, w_0}) \log q^2_{w_0}
    \sum_{\substack{T = \begin{psmallmatrix} \xi & * \\ * & T_\flat \end{psmallmatrix} \in \Herm_n^+(F);\\ \Diff(T) = \{w_0\}}}
    \bigl( \Int_{w_0}(T, \phi_{w_0}) - \PDen_{w_0}(T, \phi_{w_0}) \bigr) \cdot \Orb_{T}(\phi^{w_0}) 
    \in \BR_{\fkd, \ov{\BQ}}.
\end{multline*}
\end{proposition}

\begin{proof}
As discussed in \S \ref{sec:local interection of KR cycles}, for each $\xi > 0$, there exists an element
\begin{equation*}
    \frac{\vol(K_G)}{\tau(Z^{\BQ})[E:F]} \bigl( \GZ^\bB(\xi, \phi_1), \DGZ(T_\flat, \phi_\flat)^{\mathrm{aux}} \bigr) - J_{T_\flat, \xi}(\tau) \in \BR
\end{equation*}
which lifts the Fourier coefficient $\CE_{T_\flat, \xi}(\tau, \phi_1, \phi_\flat) \in \BR_{\fkd}$. The intersection pairing can then be decomposed into local contributions as described in \S \ref{sec:local interection of KR cycles}.

The archimedean contribution to the geometric side consists of special values of automorphic Green functions
\begin{equation*}
    \CG^{\bB}(\tau, \phi_1)(\DGZ_{v_0}(T_\flat, \phi_\flat) + c_\infty \DGZ_{v_0}(T_\triv, \phi_\mathrm{aux})),
\end{equation*}
which agree with the analytic side as shown in Corollary \ref{cor:coefficient of J function}.

The non-archimedean contributions to the geometric side are addressed in Proposition \ref{prop:inert-local-int}. By our construction of $T_{\triv}$, all terms involving $T_{\triv}$ vanish by Lemma \ref{lem:cancel}, leaving only the terms stated in Proposition \ref{prop:KR-gent}.
\end{proof}

\section{Globalization and the proof of the Kudla--Rapoport conjecture}\label{sec:proof}
In this section, we prove Theorem \ref{thm:KR-conj} by induction. The case $n = 1$ follows from Corollary \ref{cor:KR for n=1}.  
Assuming the theorem holds for all local hermitian spaces of dimension $\leq n-1$ ($p\neq 2$), we now establish the Kudla--Rapoport conjecture for $n$ and for an integer $ 0 \leq h \leq n$, for a prime $p > 2$, and for an admissible Schwartz function of weight $\bw=(w_1,\bw_\flat)$, where $\bw_\flat\in\{\CZ,\CY\}^{n-1}$. We write this case as $\mathrm{KR}(F/F_0,n,h,\bw,\bu)$.

Recall that $V$ is the $F/F_0$-hermitian space that contains a type $h$ vertex lattice, and $\BV$ is the $F/F_0$-hermitian space that does not.
By Proposition \ref{lem:cancellation law}, we have the base case:  
\begin{lemma}\label{lem:cancel}
Assume  $p \neq 2$, and let $\phi \in \CS(V^n)$ be an admissible Schwartz function of type $h$. Let $\bu \in \BV^n$ be a set of linearly independent vectors such that for one of the entries $u_i$, the following holds:
\begin{enumerate}
    \item If $h \neq n$,  then $\val(u_i, u_i) = 0$; if $h = n$, then $\val(u_i, u_i) = 1$.
    \item If  $h \neq 0$, then $\val(u_i, u_i) = -1$; if $h = 0$, then $\val(u_i, u_i) = 0$.
\end{enumerate}
Then, we have  
\begin{equation*}
    \Int(\bu, \phi) = \PDen(\bu, \phi).
\end{equation*}
In other words, Theorem \ref{thm:KR-conj} holds for $\mathrm{KR}(F/F_0,n,h,\bw,\bu)$. \qed
\end{lemma}

Suppose we aim to prove Theorem \ref{thm:KR-conj} for an unramified quadratic extension $F/F_0$ with $p$-adic fields with $p>2$, a vector $\bu \in \BV^n$, and all admissible Schwartz functions $\phi \in \CS(V(F_0)^n)$ of type $h$.

We now change to global notation. We start with a global CM extension $F/F_0$ satisfying the following conditions:  
\begin{enumerate}[label=\textbf{CM\,\arabic*}]
    \item $F_0$ has a place $v_0$ above $p$ that is inert in $F$, and the induced local extension $F_v/F_{0,v_0}$ is the one for which we aim to prove the Kudla--Rapoport conjecture.

    \item \label{CM2} All places $w_0 \neq v_0$ of $F_0$ lying above $p$ are split in $F$.
    \item \label{CM3} If we want to verify the local Kudla--Rapoport conjecture over $\BQ_p$, we choose $F_0$ to be a quadratic extension of $\BQ$ such that $p$ splits in $F_0$, i.e. $F_0\otimes_{\BQ}\BQ_p=\BQ_p\times \BQ_p$.
\end{enumerate}

We fix a CM type $\Phi$ for $F/F_0$. By enlarging $\fkd \subset \BZ_{>1}$ while keeping $p \nmid \fkd$, we further assume that $F_w/F_{0,w_0}$ is unramified for every finite places $w\nmid \fkd$ of $F$.

Now, we choose an $n$-dimensional $F/F_0$-hermitian vector space $V$ satisfying the following conditions:  

\begin{enumerate}[label=\textbf{HM\,\arabic*}]
    \item The signature of $V$ is $(n-1,1)$ at a unique place $\varphi_0: F_0 \to \BR$ and $(n,0)$ at all other places $\varphi\in\Phi\setminus \{\varphi_0\}$.
    \item The localization $V_{v_0}$ is isomorphic to the hermitian space we started with.
    \item \label{HM3} The space $V_p = \prod_{w_0 \mid p} V_{w_0}$ contains a lattice $L_p = \prod_{w_0 \mid p} L_{w_0}$ such that $L_{v_0}$ is a vertex lattice of type $h$. We extend this into a lattice $L^{[h]} \subset V$ (enlarging $\fkd$ if necessary). Similarly, we construct $L^{[t]} \subset V^{(v_0),\fkd}$ such that it is of type $t \equiv h - 1 \bmod{2}$.
\end{enumerate}	

By local constancy, we may replace $\bu\in V^{(v_0)}(F_{0,v_0})^n$ with a vector in $V^n(F_0)$ without changing the intersection number on the geometric side (Theorem \ref{thm:local-constancy}) and the corresponding modified local density function on the analytic side (by continuity). We still denote this vector by $\bu$. Let $\bT\in\Herm_n(F)$ be its moment matrix.
It is non-degenerate and totally positive definite, with $\Diff(\bT)=\{v_0\}$. 
We write
\begin{equation*}
\bT=\begin{pmatrix}
\bt&\bz\\
\prescript{t}{}{\ov{\bz}}&T_\flat
\end{pmatrix}
\end{equation*}
where $T_\flat$ is a matrix of size $(n-1)\times (n-1)$.

We further enlarge $\fkd$ so that the following conditions hold:  

\begin{enumerate}[label=\textbf{LR\,\arabic*}]
\item \label{LR1} By enlarging $\fkd$, we may assume that $L^{[h]}_{w_0}$ and $L_{w_0}^{[t]}$ are self-dual for all inert places $w_0 \in \Sigma_{F_0,\mathrm{inert}}^{\fkd}\setminus\{v_0\}$.
\item \label{LR2} By enlarging $\fkd$, we may assume that for all inert places $w_0 \in \Sigma_{F_0,\mathrm{inert}}^{\fkd}\setminus\{v_0\}$, there exists some $i \in \{2, \dots, n\}$ such that  
$$
\val_{w_0} (u_i, u_i) = 0,
$$
where $u_i$ is an entry of the vector $\bu$.
\end{enumerate}

We now complete the given data to a Shimura datum as in \S \ref{sec:gp data}, ensuring only that $p \nmid \fkd$.
Then, we construct a class of Schwartz functions.  
First, we write
\begin{equation*}
\phi^{\fkd}_{1,\CZ} := \charfun_{\wh{L}^{[h],\fkd}}  
\quad \text{and} \quad  
\phi^{\fkd}_{1,\CY} := \charfun_{\wh{L}^{[h],\sharp,\fkd}}.
\end{equation*}
They are admissible Schwartz functions in $\CS(V(\BA_{0,f}^{\fkd}))^{K^{\fkd}_G}$. Similarly, let $\phi_{\flat, v_0} \in \CS(V(F_{0,\fkd})^{n-1})^{K^{\fkd}_G}$ be an admissible Schwartz function of weight $\bw_\flat$.

We then construct $\mathbb{Q}$-valued Schwartz functions $\phi_{1,\fkd} \in \CS(V(F_{0,\fkd}))^{K_{G,\fkd}}$ and $\phi_{\flat,\fkd} \in \CS(V(\BA_{0,f})^{n-1})^{K_{G,\fkd}}$ such that:  

\begin{enumerate}[label=\textbf{SW\,\arabic*}]
    \item \label{SW1} The Schwartz function $\phi_{1,\fkd} \otimes \phi_{\flat,\fkd}$ is regular at two or more places (see Theorem \ref{thm:sing-vanish}).
    Such a test function always exists, as we can construct a regular test function $\phi_{w_0}$ of the form  
    \begin{equation*}
        \phi_{w_0} = \bigotimes_i \phi_{i, w_0}, \quad
        \phi_{i,w_0} = \charfun_{D_i} \in \CS(V_{w_0}),
    \end{equation*}
    where $\charfun_{D_i}$ is the characteristic function of a small region $D_i$ (see Remark \ref{rem:reg-fun-construct}).  
    Moreover, we can choose $\phi_{v_0}$ to be $\mathbb{Q}$-valued.    

    \item \label{SW2} The orbital integrals $\Orb_{\bT}(\phi_{\CZ}^{v_0})$ and $\Orb_{\bT}(\phi_{\CY}^{v_0})$ are non-vanishing, where
    \[
    \phi_? = (\phi^{\fkd}_{1,?} \phi_{1,\fkd}) \otimes \phi_\flat \in \CS(V(\BA_{0,f})^n), \quad ? \in \{\CZ, \CY\}.
    \]

    \item \label{SW3} For all positive definite Hermitian matrices $T\neq \bT$ of the form
    \[
    T = \begin{pmatrix} \bt & * \\ * & T_\flat \end{pmatrix} \in \Herm_n^+(F),\quad \Diff(T,\BV)=\{v_0\}, 
    \]
    the orbital integrals $\Orb_T(\phi_{\CZ}^{v_0})$ and $\Orb_T(\phi_{\CY}^{v_0})$ vanish.  
    Since there are only finitely many $T$ satisfying these conditions, we can achieve this by shrinking the support of the Schwartz function at a single place so that it meets only $V_{\bT}$.
\end{enumerate}

We choose gauge forms $\alpha\in \wedge^{2n^2}(V^n)^*$ and $\beta\in\wedge^{n^2}(\Herm_n)^*$ such that the volumes  $\vol(K_{w_0})$ are rational numbers for all $w_0\in\Sigma_f^{\fkd}$, where the volumes are computed using the Tamagawa measure induced by $\alpha$ and $\beta$ (see \S \ref{sec:local-factors}) (this is possible, see, e.g. \cite{KR14} or \cite{Bruinier-Yang21}).

By Proposition \ref{prop:KR-gent} we have the generating series
\begin{multline*}\label{equ:YZ generating series}
\CE_{T_\flat}(\tau,\phi_{1?},\phi_\flat)=\sum_{\xi\in F_0}\sum_{w_0\in\Sigma_f^{\fkd}}\vol(K_{G,w_0})\log q^2_{w_0}\\
\sum_{{\substack{T=\begin{psmallmatrix}\xi& *\\ *& T_\flat\end{psmallmatrix}\in\Herm_n^+(F);\\\Diff(T)=\{w_0\}}}}\left(\Int_{w_0}(T,\phi_{1?}\phi_\flat)-\PDen_{w_0}(T,\phi_{1?}\phi_\flat)\right)\cdot \Orb_{T}(\phi_f^{w_0})q^\xi,\quad ?\in\{\CZ,\CY\}.
\end{multline*}

\begin{proposition}\label{prop:dual-relation}
The following dualizing relation holds:
$${ \begin{pmatrix}
0&-1\\1&0
\end{pmatrix}_{ v_0}}.\CE_{T_\flat}(\tau,\phi_{1\CZ},\phi_\flat)=\gamma_{V_{ v_0}}\vol(L_{ v_0})\CE_{T_\flat}(\tau,\phi_{1\CY},\phi_\flat).$$
\end{proposition}
\begin{proof}
Recall from \eqref{equ:define generat series} that $\CE_{T_\flat}(\tau,\phi_1,\phi_\flat)$ is a sum of the following terms (all terms with $T_{\triv}$ vanish):
$$
\CG_{w_0}^{\bK-\bB}(\tau,\phi_{1?}),\quad \PFJ^{\modi}_{T_\flat}(\tau,\phi),\quad \omega_{T_\flat,w_0}(\tau,\phi_1,\phi_\flat),\quad \left(\wh{\GZ}^\bB(\tau,\phi_1),\DGZ(T_\flat,\phi_\flat)\right).$$ 
Each of these terms satisfies the claimed modularity via the action of the Weil representation, as established in Theorem \ref{thm: Green}, Proposition \ref{prop:omega series}, Proposition \ref{prop:Fourier--Jacobi series are modular forms} (combine with Definition \ref{def:error-coeff}(2) for the error terms) and Proposition \ref{prop:modularity-bruinier}, respectively.
Therefore, the assertion follows from the relation in $\CS(V(\BA_{0,f}))$
\begin{equation*}
\pushQED{\qed}
\begin{pmatrix}
0&-1\\
1&0
\end{pmatrix}_{v_0}. \charfun_{L_{v_0}}=\gamma_{V_{v_0}}\vol(L_{v_0})\cdot \charfun_{L^{\sharp}_{v_0}}.
\qedhere
\end{equation*}
\end{proof}

\begin{theorem}\label{thm:vanishing generating series}
We have
$\CE_{T_\flat}(\tau,\phi_{1?},\phi_\flat)=0$ for $?\in\{\CZ,\CY\}$.
\end{theorem}
\begin{proof}
For notational convenience, we write $\CE_?(\tau)$ for $\CE_{T_\flat}(\tau,\phi_{1?},\phi_\flat)$, where $?\in\{\CZ,\CY\}$. Its Fourier expansion is given by
\begin{equation*}
    \CE_?(\tau)=\sum_{\xi>0}\CE_{?,\xi} q^\xi,\quad ?\in\{\CZ,\CY\}.
\end{equation*}
By Proposition \ref{prop:KR-gent}, for $\xi\in F_0^\times$, we have:
\begin{equation*}
\CE_{?,\xi}=\sum_{w_0\in\Sigma_f^{\fkd}}\frac{\log q^2_{w_0}}{\vol(K^{w_0}_G)}
\sum_{{\substack{T=\begin{psmallmatrix}\xi& *\\ *& T_\flat\end{psmallmatrix}\in\Herm_n^+(F);\\\Diff(T)=\{w_0\}}}}\left(\Int_{w_0}(T,\phi_{1?,w_0})-\PDen_{w_0}(T,\phi_{1?,w_0})\right)\cdot \Orb_{T}(\phi_{f?}^{w_0}),\quad ?\in\{\CZ,\CY\}.
\end{equation*}

Let $c_0=\gamma_{V_{v_0}}\vol(L_{v_0})$. We apply the double induction lemma from \cite[Lemma 15.1]{ZZhang24} to the pair $(f_1,f_2)=(\CE_{\CZ}(\tau),c_0\CE_{\CY}(\tau))$ with $B=\{v_0\}$.
The dual relation \cite[Lemma 15.1(1)]{ZZhang24} follows from Proposition \ref{prop:dual-relation}, while the support condition \cite[Lemma 15.1(2)]{ZZhang24} follows from \eqref{equ:define generat series}, where we choose $K_{\bH,v_0}$ as $\Gamma(v_0):=\ker\bigl(\SL_2(O_{F_0,v_0})\to \SL_2(\BF_{v_0})\bigr)$.
It remains to verify the boundary condition \cite[Lemma 15.1(3)]{ZZhang24}, which boils down to the following cases:
\begin{enumerate}
\item When $h=0$, we have $k_{1, v_0}=-c_{ v_0}$ and  $k_{2, v_0}=-c_{ v_0}$. It is sufficient to prove the following two claims: 
\begin{enumerate}
\item For $\xi$ with $ v_0(\xi)=0$, we have $\CE_{\CZ,\xi}=0$;
\item For $\xi$ with $ v_0(\xi)=0$, we have $\CE_{\CY,\xi}=0$.
\end{enumerate}

\item When $0<h<n$, we have $k_{1, v_0}=-c_{ v_0}$ and  $k_{2, v_0}=-c_{ v_0}+1$. It is sufficient to prove the following two claims:
\begin{enumerate}
\item For $\xi$ with $ v_0(\xi)=0$, we have $\CE_{\CZ,\xi}=0$;
\item For $\xi$ with $ v_0(\xi)=-1$, we have $\CE_{\CY,\xi}=0$.
\end{enumerate}

\item When $h=n$, we have $k_{1, v_0}=-c_{ v_0}-1$ and  $k_{2, v_0}=-c_{ v_0}+1$. It is sufficient to prove the following two claims:
\begin{enumerate}
\item For $\xi$ with $ v_0(\xi)=1$, we have $\CE_{\CZ,\xi}=0$;
\item For $\xi$ with $ v_0(\xi)=-1$, we have $\CE_{\CY,\xi}=0$.
\end{enumerate}
\end{enumerate}
Let us verify case (2); the other cases follow by analogous arguments. By Proposition \ref{prop:KR-gent}, it suffices to show that
\begin{equation}\label{equ:Fourier coefficient in the double induction}
\sum_{{\substack{T=\begin{psmallmatrix}\xi& *\\ *&  T_\flat\end{psmallmatrix}\in\Herm_n^+(F);\\\Diff(T)=\{w_0\}}}}\vol(K_{G,w_0})\left(\Int_{w_0}(T,\phi_{?w_0})-\PDen_{w_0}(T,\phi_{?w_0})\right)\cdot \Orb_{T}(\phi_{?}^{w_0})=0,\quad ?\in\{\CZ,\CY\}
\end{equation}
for every inert places $w_0\nmid \fkd$ in case (i) and (ii). 
\begin{itemize}
\item If $w_0= v_0$, then \eqref{equ:Fourier coefficient in the double induction} follows from Lemma \ref{lem:cancel} and assumption \ref{SW3}. 
\item If $w_0\neq  v_0$, then since $w_0\nmid\fkd$, assumption \ref{LR2} implies that one of the diagonal entries of $T_\flat$ is a unit at $w_0$, and we may again apply Lemma \ref{lem:cancel}.
\end{itemize}
Case (ii) follows by the same reasoning, which completes the proof
\end{proof}

\begin{corollary}
The Kudla--Rapoport conjecture (Theorem \ref{thm:KR-conj}) holds for $\mathrm{KR}(F_v/F_{0,v_0},n,h,\bw,\bu)$.
\end{corollary}
\begin{proof}
By Theorem \ref{thm:vanishing generating series}, we have $\CE_{T_\flat}(\tau,\phi_{1\CZ},\phi_\flat)=0$ and $\CE_{T_\flat}(\tau,\phi_{1\CY},\phi_\flat)=0$. Taking the $\bt$-th Fourier coefficient, we have
\begin{equation*}
    \sum_{w_0\in\Sigma_f^{\fkd}}\vol(K_{G, w_0})\log q^2_{w_0}
\sum_{{\substack{T=\begin{psmallmatrix}\bt& *\\ *& T_\flat\end{psmallmatrix}\in\Herm_n^+(F);\\\Diff(T)=\{ v_0\}}}}\left(\Int_{w_0}(T,\phi_{? w_0})-\PDen_{w_0}(T,\phi_{ ?w_0})\right)\cdot \Orb_{T}(\phi_f^{ w_0})=0,\quad ?\in\{\CZ,\CY\}.
\end{equation*}
By our choice of gauge form on $V^n$ and $\Herm_n$, the volumes $\vol(K_{G,w_0})$ are in $\BQ$.
The log terms $\log q_{w_0}$ for $w_0\in\Sigma_{F_0,f}^{\fkd}$ are $\ov{\BQ}$-linearly independent. By \ref{SW1}, $\phi$ takes values in $\BQ$, hence 
\begin{equation*}
\sum_{{\substack{T=\begin{psmallmatrix}\bt& *\\ *& T_\flat\end{psmallmatrix}\in\Herm_n^+(F);\\\Diff(T)=\{v_0\}}}}\left(\Int_{v_0}(T,\phi_{? v_0})-\PDen_{ v_0}(T,\phi_{?, v_0})\right)\cdot \Orb_{T}(\phi^{v_0})=0,\quad ?\in\{\CZ,\CY\}.
\end{equation*}
By \ref{SW3} and Lemma \ref{lem:cancel}, in the summation only the term with index $\bT=h(\bu,\bu)$ is nonzero. Specifically, we have
\begin{equation*}
\left(\Int_{ v_0}(\bT,\phi_{? v_0})-\PDen_{ v_0}(\bT,\phi_{? v_0})\right)\cdot \Orb_{\bT}(\phi^{ v_0})=0,\quad ?\in\{\CZ,\CY\}.
\end{equation*}
By \ref{SW2}, $\Orb_{\bT}(\phi^{ v_0})\neq 0$, and we have
\begin{equation*}
    \Int_{ v_0}(\bT,\phi_{? v_0})=\PDen_{v_0}(\bT,\phi_{?v_0}),\quad ?\in\{\CZ,\CY\}.
\end{equation*}
Since $\phi_{\flat,v_0}$ can be chosen to be any admissible Schwartz function, this completes the proof.
\end{proof}

\begin{remark}\label{rmk:globalization-unramified}
When $E_\nu/F_{v}$ is unramified for all places $\nu\mid v\in\Sigma^{\fkd}_{\mathrm{max}}$, the integral model $\wt{\CM}\to \Spec O_E[\fkd^{-1}]$ is regular.
In the globalization procedure for the proof of the Kudla--Rapoport conjecture, one can always choose a global CM extension $F/F_0$ with CM type $\Phi$ and a hermitian space $V$ as described in \S \ref{sec:gp data}, such that the reflex field $E$ is unramified over $F$ at the places above $p$.
Such a CM type is called \emph{unramified at a prime $p$} in \cite[\S 6.2]{Mihatsch-Zhang}, and its existence is established in \cite[after (10.1)]{Mihatsch-Zhang}.
%To be more precise, the construction proceeds as follows\footnote{Based on a  personal communication with A. Mihatsch.}:
%\begin{altenumerate}
%\item Given is a local $F/F_0$. We can globalize to a CM field $\bF/\bF_0$ such that $\bF_p = F$, $\bF_{0,p} = F_0$. 

%\item Next, choose a local CM type $\Phi^u$ for the maximal unramified subfield $F^u/F_0^u$. Since $F/F_0$ is inert, $F^u/F_0^u$ is a quadratic extension and the notion of CM type for $F^u/F_0^u$ makes sense.

%\item Take the local $\Phi$ as its inverse image under $\Hom(F,  \ov{\BQ_p})\to \Hom(F^u, \ov{\BQ_p})$. 
%Then the local reflex field is $E_\Phi = F^u$.

%\item Now choose a local $\varphi_0$ and take the local $A$ as $\Phi \setminus \{\varphi_0, \ov{\varphi_0}\}$. Choose $B$ as $A \cup \{\varphi_0,  \ov{\varphi_0}\}$. Then $E_AE_B\varphi_0(F) = E_\Phi\varphi_0(F) = \varphi_0(F)$ (relates to the $(A,B)$-strictness in the Eisenstein condition, see \cite[\S 6.2]{LMZ25}). 

%\item Now choose an embedding $\alpha: \bar \BQ \to \bar \BQ_p$. Choose the global $\phi_0$, $\Phi$, $A$, and $B$ as preimage under $\alpha$ via the isomorphism $\alpha_*: \Hom(\bF, \bar \BQ) \simto \Hom(F, \bar \BQ_p)$.

%\item This defines global $\Phi$, $A$, $B$ and $\bE$. For every place $\nu$ of $E$ above $p$, $E_\nu = F$ and the local situation is as before.
%\end{altenumerate}
\end{remark}

\bibliographystyle{alpha}
\bibliography{reference}
\end{document}